\documentclass[11pt]{amsart}

\usepackage{amsmath}
\usepackage{amscd}
\usepackage{amsthm}
\usepackage{amssymb} \usepackage{latexsym}
\usepackage{eufrak}
\usepackage{euscript}
\usepackage{epsfig}
\usepackage{graphics}
\usepackage{array}
\usepackage{enumerate}
\usepackage{dsfont}
\usepackage{color}
\usepackage{wasysym}
\usepackage{hyperref}

\newcommand{\bel}[1]{\begin{equation}\label{#1}}

\newcommand{\be}{\begin{equation}}

\newcommand{\ba}{\begin{eqnarray}}
\newcommand{\ea}{\end{eqnarray}}
\newcommand{\rf}[1]{(\ref{#1})}

\newcommand{\qe}{\end{equation}}
\newcommand{\R}{{\mathbb R}}

\theoremstyle{theorem}
\newtheorem{theo}{Theorem}[section]

\theoremstyle{example}
\newtheorem{example}{Example}[section]
\theoremstyle{corollary}
\newtheorem{coro}{Corollary}[section]
\theoremstyle{lemma}
\newtheorem{lemma}{Lemma}[section]
\theoremstyle{definition}
\newtheorem{defi}{Definition}[section]
\theoremstyle{proof}

\theoremstyle{remark}
\newtheorem{rem}{Remark}[section]

\begin{document}

\title[Dual Cheeger constant and spectra of graphs]{The dual Cheeger constant and spectra of infinite graphs}
\author{Frank Bauer}
 \email{Frank.Bauer@mis.mpg.de}
\address{Max Planck Institute for Mathematics in the Sciences\\
04103 Leipzig, Germany.}
\author{Bobo Hua}
\email{bobohua@mis.mpg.de}
\address{Max Planck Institute for Mathematics in the Sciences\\
04103 Leipzig, Germany.}
\author{J\"urgen Jost}
 \email{jost@mis.mpg.de}
\address{Max Planck Institute for Mathematics in the Sciences\\
04103 Leipzig, Germany.}
\address{Department of Mathematics and Computer Science \\University of
Leipzig \\04109 Leipzig, Germany }
\address{Santa Fe Institute for the Sciences of Complexity, Santa Fe,
NM 87501, USA}

\thanks{The research leading to these results has received funding from the
European Research Council under the European Union's Seventh
Framework Programme (FP7/2007-2013) / ERC grant agreement
n$^\circ$ 267087.}

\begin{abstract}In this article we study the top of the spectrum of
the normalized Laplace operator on infinite graphs. We introduce
the dual Cheeger constant and show that it controls  the top of
the spectrum from above and below in a similar way as the Cheeger
constant controls the bottom of the spectrum. Moreover, we show
that the dual Cheeger constant at infinity can be used to
characterize that the essential spectrum of the normalized Laplace
operator shrinks to one point.
\end{abstract}
\maketitle \tableofcontents

\section{Introduction}
While global geometric properties are inherently nonlinear, they
nevertheless can often be controlled by linear techniques. A prime
example in geometry is the spectrum of the Laplace operator  which
encodes geometric information about the underlying manifold. This
principle has been particularly fertile and most intensively
developed in Riemannian geometry, see \cite{Chavel84} for an
overview. It possesses a more general validity, however. In
particular, more recently, spectral methods have been successfully
explored in graph theory. In fact, it was observed that, while
being very different from manifolds, graphs can be investigated by
modifying techniques that originally have been developed in
Riemannian geometry. This insight was very fruitful and led to
many important results in graph theory, such as a discrete version
of Courant's nodal domain theorem \cite{Davies01}, Sobolev and
Harnack inequalities \cite{Chung-Yau95, Diaconis96,
Chung94aharnack, Chung12}, heat kernel estimates \cite{Delmotte99,
Coulhon98}, and many more. In particular, it turns out that the
normalized Laplace operator on graphs is related to the
Laplace-Beltrami operator for a Riemannian manifold. However, not
only results from the continuous setting triggered new results in
the discrete case but also results in graph theory led to new
insights in geometric analysis. One example of the mutual
stimulation between both fields is for example given in the work
of Chung, Grigor'yan and Yau \cite{Chung-Grigoryan-Yau96,
Chung-Grigoryan-Yau97, Chung-Grigoryan-Yau00} where a universal
approach for eigenvalue estimates on continuous and discrete
spaces was developed.

Graph theory, however, also has specific aspects that are
different from what occurs in other parts of geometry. Our
question then is whether and how these aspects can be explored
with spectral methods. For instance, a bipartite graph, i.e. a
graph whose vertex set can be split into two subsets such that
there are only edges between vertices belonging to different
subsets, has no counterpart in Riemannian geometry. Another
difference to Riemannian geometry is that for graphs, the spectrum
of the normalized Laplace operator is bounded from above by two.
For finite graphs, these two properties are connected by the fact
that 2 is an eigenvalue of the normalized Laplace operator if and
only if the graph is bipartite.

This leads us to the specific  purpose of this paper which is to
investigate the top of the spectrum of the normalized Laplace
operator for infinite graphs. In order to do this we have
developed on the one hand new techniques that have no counterpart
in Riemannian geometry, and on the other hand have discovered
results that are similar to important theorems in the continuous
setting.

In Section \ref{Section3} we introduce the Dirichlet Laplace
operator and its basic properties. Let $\Gamma$ be an infinite
graph and $\Omega$ a finite, connected subset of $\Gamma$. Let
$\lambda_1(\Omega)$ ($\lambda_\mathrm{max}(\Omega)$) be the
  first (largest) eigenvalue of the Dirichlet Laplace operator on $\Omega.$
We then have
\begin{equation}
  \label{01}
\lambda_1(\Omega) + \lambda_\mathrm{max}(\Omega) \le 2,
\end{equation}
and we prove in Theorem \ref{Lembi} that we have the equality
$\lambda_1(\Omega) + \lambda_\mathrm{max}(\Omega) = 2$ if and only
if $\Omega$
 is bipartite. In Section \ref{Section4} we
introduce a geometric quantity, the so-called dual Cheeger
constant $\bar{h}(\Gamma)$ that roughly speaking measures how
close a graph is to a bipartite one. We will show in this section
that the dual Cheeger constant $\bar{h}(\Gamma)$ is closely
related to the Cheeger constant $h(\Gamma)$. The Cheeger constant
$h(\Gamma)$ of a infinite graph $\Gamma$ is one of the most
fundamental geometric quantities which gives rise to important
analytic consequences such as the first eigenvalue estimate, heat
kernel estimates and so on. Also in the setting of finite graphs
the Cheeger constant is related to the growth behavior of the
graph since it is closely connected to the important notion of
expanders, see \cite{Lubotzky94} and the references therein for
more details. The Cheeger constant for a finite subset $\Omega$ of
the vertex set $V$ of $\Gamma$ is defined as
$$h(\Omega) = \inf_{\substack{\emptyset\neq U\subset
\Omega\\\sharp U
    <\infty}}\frac{|\partial U|}{\mathrm{vol}(U)}, $$
where the volume of $U$ is the sum of the degrees of its vertices,
and $|\partial U|$ measures the edges going from a vertex in $U$
to one outside $U$. A sequence of finite subsets of $\Gamma$
satisfying
$\Omega_1\subset\Omega_2\subset\cdots\subset\Omega_n\subset\cdots$
and $\Gamma=\cup_{n=1}^{\infty}\Omega_n,$ denoted by
$\Omega\uparrow\Gamma,$ is called an exhaustion of $\Gamma.$ Since
$h(\Omega)$ is non-increasing when $\Omega$ increases, we can
define the Cheeger constant as the limit for subsets exhausting
$\Gamma$, i.e. $h(\Gamma)=\lim_{\Omega\uparrow \Gamma}h(\Omega),$
which does not depend on the choice of the exhaustion. The Cheeger
constant tells us how difficult it is to carve out large subsets
with small boundary. For a finite subset $\Omega\subset V$ we then
define the dual Cheeger constant  by
$$\bar{h}(\Omega) =
\sup_{\substack{V_1, V_2\subset \Omega\\\sharp V_1, \sharp
V_2<\infty}}\frac{2|E(V_1,V_2)|}{\mathrm{vol}(V_1)+\mathrm{vol}(V_2)},$$
where $V_1,V_2$ are two disjoint nonempty subsets of $\Omega$ and
$|E(V_1,V_2)|$ counts the edges between them. The dual Cheeger
constant of $\Gamma$ then is obtained as the limit for
exhaustions, i.e. $\bar{h}(\Gamma)=\lim_{\Omega\uparrow
\Gamma}\bar{h}(\Omega).$ The dual Cheeger constant tells us how
easy it is to find large subsets with few inside connections, but
many between them. Observe that a bipartite graph is one that we
can divide into two subsets without any internal connections. This
is the paradigm behind the dual Cheeger constant.

In Theorem \ref{h and hbar}, we prove
\begin{equation}
  \label{02}
 h(\Omega)+ \bar{h}(\Omega) \le 1,
\end{equation}
with equality $h(\Omega)+ \bar{h}(\Omega)=1$ if the graph is
bipartite (the converse for equality is not true, see Example
\ref{counterexample1}). There is an obvious analogy between
(\ref{01})  and (\ref{02}), and their equality cases for bipartite
graphs. We shall explore this analogy in this paper. In fact, the
dual Cheeger constant was first introduced in \cite{Bauer12} to
effectively estimate the largest eigenvalue of a finite graph. In
this paper, we develop this technique for the setting of the
Dirichlet Laplace operator on  finite subsets $\Omega$ of an
infinite graph $\Gamma.$ Our first main result is that the largest
eigenvalue of the normalized Dirichlet Laplace operator can be
controlled from above and below in terms of the dual Cheeger
constant, see Theorem \ref{Mtheo1}. Thus, we obtain inequalities
relating the largest Dirichlet eigenvalue and the dual Cheeger
constant $\bar{h}(\Omega)$ that are analogous to the relationship
between the first Dirichlet eigenvalue and the Cheeger constant.
In Theorem \ref{h and hbar 3} we prove that for graphs without
self-loops there is another relation between the dual Cheeger
constant $\bar{h}(\Omega)$ and $h(\Omega)$,
\begin{equation}
  \label{03}
 \frac{1}{2}(1-h(\Omega))\leq \bar{h}(\Omega)
\end{equation}
(\ref{02}) and (\ref{03}) thus tell us  that $\bar{h}(\Omega)$ and
$h(\Omega)$ can be controlled by each other. These estimates will
be the key tool in Section \ref{Section10} where we study the
essential spectrum of the normalized Laplace operator of an
infinite graph. Nevertheless, the analogy between
$\bar{h}(\Omega)$ and $h(\Omega)$ suggested by (\ref{02}) and
(\ref{03}) has its limitations, as we shall see in Section
\ref{Section8} (but the analogy will make a surprising comeback in
Section \ref{Section10}).

But let us first describe results where the analogy holds. In
Section \ref{Section5} we prove eigenvalue comparison theorems for
the largest Dirichlet eigenvalue that are counterparts of the
first eigenvalue comparison theorems on graphs by Urakawa
\cite{Urakawa99} which are discrete versions of Cheng's first
eigenvalue comparison for Riemannian manifolds \cite{Cheng75a,
Cheng75b}. The mostly used comparison models in the literature
(see e.g. \cite{Urakawa99,Friedman93}) for graphs are homogeneous
trees, denoted by $T_d\ (d\in \mathds{N}, \ d\geq 2)$. Here we
propose a novel comparison model for graphs, the weighted
half-line  $R_l\ (l\in \mathds{R},\ l\geq 2)$ which is inspired by
the behavior of the eigenfunctions for the first and largest
eigenvalues on balls in the
 homogeneous tree $T_l$. Compared to the homogenous tree $T_d$, the
advantage of the weighted half-line $R_l$ is that we get better
estimates in the comparison theorems since, for our model space,
in contrast to the parameter $d$ for a homogeneous tree, $l$ need
not be an integer. Our main result in this section is that the
largest Dirichlet eigenvalue of a ball in a graph can be
controlled by that of a ball in the weighted half-line $R_l$ of
the same radius and a quantity that is related to the
bipartiteness of the graph.

In Section \ref{Section6} we show how the eigenvalue comparison
theorems from Section \ref{Section5} can be used to estimate the
largest, the second largest and other largest eigenvalues of the
normalized Laplace operator for finite graphs. In Section
\ref{Section7} we use the dual Cheeger constant $\bar{h}(\Gamma)$
to estimate the top of the spectrum of $\Gamma$. In Section
\ref{Section8} we study the dual Cheeger constant and its
geometric and analytic consequences. Let us denote by
$\sigma(\Delta)\ (\subset [0,2])$ the spectrum of the normalized
Laplace operator of an infinite graph $\Gamma,$ by
$\underline{\lambda}(\Gamma):=\lim_{\Omega\uparrow\Gamma}\lambda_1(\Omega)=\inf
\sigma(\Delta)$ the bottom of the spectrum of an infinite graph
$\Gamma$ and  by
$\bar{\lambda}(\Gamma):=\lim_{\Omega\uparrow\Gamma}\lambda_{\mathrm{max}}(\Omega)=\sup
\sigma(\Delta)$ its top.  It is well known that the spectral gap
for $\underline{\lambda}(\Gamma),$ i.e.
$\underline{\lambda}(\Gamma)>0,$ implies $\Gamma$ has exponential
volume growth and a fast heat kernel decay. Moreover, the spectral
gap for $\underline{\lambda}(\Gamma)$ is a rough-isometric
invariant (see e.g. \cite{Woess00}). We shall derive the
exponential volume growth condition for the spectral gap for
$\bar{\lambda}(\Gamma),$ i.e.  $\bar{\lambda}(\Gamma)< 2,$ if we
further assume the graph $\Gamma$ is, in some sense, close to a
bipartite one. More importantly, we present some examples (see
Example \ref{ex2} and Example \ref{ex3}) to show that the spectral
gap for $\bar{\lambda}$ is not a rough-isometric invariant. Since
the dual Cheeger constant $\bar{h}(\Gamma)$ is closely related to
$\bar{\lambda}(\Gamma),$ it is follows that also
$\bar{h}(\Gamma)<1$ is not a rough-isometric invariant. Note that
in contrast, $h(\Gamma)>0$ is a rough-isometric invariant.

In Section \ref{Section9} we study the top of the spectrum of
infinite graphs with certain symmetries. A graph $\Gamma$ is
called quasi-transitive if it has only finitely many orbits by the
action of the group of automorphisms, $\mathrm{Aut}(\Gamma)$, i.e.
$\sharp\Gamma\slash \mathrm{Aut}(\Gamma)<\infty,$ see
\cite{Woess00}. It is obvious that Cayley graphs are
quasi-transitive. We prove in Theorem \ref{mainth1}
 that for a non-bipartite quasi-transitive graph $\Gamma$,
$\bar{h}(\Gamma)\leq 1-\delta,$ where $\delta=\delta(\Gamma)>0.$
As a corollary (see Corollary \ref{mthlast}), we obtain that for a
quasi-transitive graph $\Gamma,$ $\bar{\lambda}(\Gamma)=2$ implies
that $\Gamma$ is bipartite. This is a generalization of a result
for Cayley graphs in \cite{HarpeRobertsonValette}. The proof in
\cite{HarpeRobertsonValette} is based on techniques from
functional analysis, in particular C*-algebras. Our proof is
completely different, since we use a combinatorial argument to
estimate the geometric quantity $\bar{h}(\Gamma)$. From the
geometric point of view this proof has the advantage that it can
easily be extended to quasi-transitive graphs, which is not true
for the proof technique used in \cite{HarpeRobertsonValette}.

In Section \ref{Section10} we investigate the essential spectrum
of infinite graphs. The main result is that the dual Cheeger
constant at infinity can be used to characterize that the
essential spectrum shrinks to one point. Let
$\sigma^{\mathrm{ess}}(\Gamma)$ denote the essential spectrum of
an infinite graph $\Gamma.$
 Note that the essential spectrum, which is related to the geometry at infinity, cannot be empty since the
normalized Laplace operator is bounded. Fujiwara \cite{Fujiwara96}
discovered (see also \cite{Keller10}) that the Cheeger constant at
infinity, denoted by $h_{\infty}$, is equal to 1 if and only if
the essential spectrum is smallest possible (i.e.
$\sigma^{\mathrm{ess}}(\Gamma)=\{1\}$). We show the following
criteria of concentration of the essential spectrum by the dual
Cheeger constant at infinity, denoted by $\bar{h}_{\infty}$.
\begin{theo}[see Theorem \ref{hbar and h at infinity}]
Let $\Gamma$ be an infinite graph without self-loops. Then
$$\bar{h}_{\infty}(\Gamma)=0\Longleftrightarrow h_{\infty}(\Gamma)=1\Longleftrightarrow \sigma^{\mathrm{ess}}(\Gamma)=\{1\}.$$
\end{theo}
The typical example here is a rapidly branching tree. This result
is somehow surprising since our results in Section \ref{Section8}
suggest that the dual Cheeger constant is weaker than the Cheeger
constant. However it turns out that at infinity (in the extreme
case $\bar{h}_\infty=0$ and $h_\infty=1$) both quantities contain
the same information. This is a new phenomenon for discrete
structures that has no
  analogue for Riemannian manifolds. For graphs with
self-loops, we give an example which has
$\bar{h}_{\infty}=h_{\infty}=0$ and $\sigma^{\mathrm{ess}}=\{0\}.$
In another direction, we generalize the results about the
essential spectrum in \cite{Urakawa99} by the comparison method,
see Theorem \ref{lasts1}, \ref{ees1} and Corollary \ref{lasts2}.
These results are discrete analogues of results in
\cite{Donnelly81,Donnelly79}.

\section{Preliminaries}\label{Preliminaries}
Let $\Gamma=(V,E)$ denote a locally finite, connected graph with
infinitely many vertices. Here $V$ denotes the vertex and
$E\subset V\times V$ the edge set of $\Gamma$. In this article we
study weighted graphs, i.e. we consider a positive symmetric
weight function $\mu$ on the edge set $\mu:E\to \mathbb{R}_+$. In
particular, for the edge $e=(x,y)$ connecting $x$ and $y$ (also
denoted by $x\sim y$) we write $\mu(e) = \mu_{xy}$ and we extend
$\mu$ to the whole of $V\times V$ by setting $\mu_{xy}=0$ if
$(x,y)\notin E$. We point out that we do allow self-loops in the
graph, that is $\mu_{xx}>0$ is possible for all $x\in V$.
Moreover, we consider a measure $\mu$ (by abuse of notation, we
will also denote this measure by $\mu$ but this should not lead to
any confusion) on the vertex set $\mu:V\to \mathbb{R}_+$ defined
by $\mu(x) = \sum_{y\in V}\mu_{xy}$ for all $x\in V$. In this
paper, for simplicity, we always assume that $\Omega$ is a finite,
connected subset of $V$ (otherwise one can easily extend the
results to the connected components of $\Omega$), and the
cardinality of $\Omega$ satisfies $\sharp \Omega\geq2$ (if $\sharp
\Omega =1$, then the Dirichlet Laplace operator is trivial, i.e.
has one single eigenvalue which is equal to
$1-\frac{\mu_{xx}}{\mu(x)}$ - see below for more details). We
define the volume of $\Omega$ as $\mathrm{vol}(\Omega):=
\mu(\Omega) = \sum_{x\in \Omega}\mu(x)$. The boundary $\partial
\Omega$ of $\Omega$ is defined as the set of all vertices $y\notin
\Omega$ for which there exists a vertex $x\in \Omega$ such that
$x\sim y$ and we define $|\partial \Omega| = \sum_{x\in
\Omega}\sum_{y\in \Omega^c}\mu_{xy}$, where $\Omega^c$ denotes the
complement of $\Omega$.

 For any two subsets $\Omega_1,\Omega_2\subset V$ we denote
$E(\Omega_1,\Omega_2)=\{(x,y)\in E: x\in \Omega_1, y\in
\Omega_2\}$ and $|E(\Omega_1,\Omega_2)| = \sum_{x\in
\Omega_1}\sum_{y\in \Omega_2}\mu_{xy}$. If $\Omega_1$ and
$\Omega_2$ are disjoint subsets of a simple ($\mu_{xx}=0$ for all
$x$) and unweighted (i.e. $\mu_{xy}=1$ for any $x\sim y$) graph
$\Gamma$, then $|E(\Omega_1,\Omega_2)|=\sharp
E(\Omega_1,\Omega_2)$ and $|E(\Omega_1,\Omega_1)|=2\sharp
E(\Omega_1,\Omega_1)$ where $\sharp E(\Omega_1,\Omega_2) (\sharp
E(\Omega_1,\Omega_1))$ is the number of edges in
$E(\Omega_1,\Omega_2) (E(\Omega_1,\Omega_1)).$ In particular, we
have $|\partial \Omega| = |E(\Omega, \Omega^c)|$. The following
formula connecting the volume and the boundary of a subset
$\Omega\subset V$ will be repeatedly used throughout this paper
\begin{equation}
\label{volumeformula}\mathrm{vol}(\Omega)=|E(\Omega,\Omega)| +
|E(\Omega,\Omega^c)|=|E(\Omega,\Omega)|+|\partial
\Omega|.\end{equation}

An important concept in this article is that of an exhaustion of
an infinite graph by finite subsets. A sequence of finite subsets
of an infinite graph $\Gamma$ satisfying
$\Omega_1\subset\Omega_2\subset\cdots\subset\Omega_n\subset\cdots$
and $\Gamma=\cup_{n=1}^{\infty}\Omega_n,$ denoted by
$\Omega\uparrow\Gamma,$ is called an exhaustion of $\Gamma.$ For
quantities that are monotone in $\Omega$, i.e. if
$\Omega_1\subset\Omega_2$ we have $f(\Omega_1) \leq f(\Omega_2)$
(or $f(\Omega_1) \geq f(\Omega_2)$) we write
$\lim_{\Omega\uparrow\Gamma}f(\Omega) := \lim_{n\to\infty}
f(\Omega_n)= f(\Gamma)$. Note that for monotone functions in
$\Omega$ this limit exists and it does not depend on our choice of
the exhaustion.

Now we are going to introduce the main object of interest of this
article - the normalized Laplace operator of a graph. We introduce
the Hilbert space,
$$\ell^2(V,\mu) := \{f:V\to \mathbb{R}\; |\; (f,f)_\mu <\infty\}$$
where we denote the inner product on $\ell^2(V,\mu)$ by $(f,g)_\mu
= \sum_{x\in V}\mu(x)f(x)g(x)$. Note that the space of functions
with finite support denoted by $C_0(V)$ is dense in
$\ell^2(V,\mu)$, i.e. $\overline{C_0(V)}= \ell^2(V,\mu)$.
 The normalized Laplace operator $\Delta:\ell^2(V,\mu)\to
 \ell^2(V,\mu)$ is pointwise defined by
 $$\Delta f(x) =f(x) - \frac{1}{\mu(x)}\sum_{y\in V} \mu_{xy}f(y).$$
It is well known (see e.g. \cite{DodziukKendall}) that $\Delta$ is
a nonnegative, self-adjoint operator whose spectrum is bounded
from above by two. We use the convention that the bottom and the
top of the spectrum of $\Delta$ are denoted by
$\underline{\lambda}(\Gamma) = \inf \sigma(\Delta)$ and
$\overline{\lambda}(\Gamma) = \sup \sigma(\Delta)$, respectively.

Besides the normalized Laplace operator $\Delta$ one can also
study the normalized Laplace operator with Dirichlet boundary
conditions. Let $\Omega$ be a finite subset of $V$ and
$\ell^2(\Omega,\mu)$ be the space of real-valued functions on
$\Omega$. Note that every function $f\in\ell^2(\Omega,\mu)$ can be
extended to a function $\tilde{f}\in\ell^2(V,\mu)$ by setting
$\tilde{f}(x)=0$ for all $x\in \Omega^c$. The Laplace operator
with Dirichlet boundary conditions $\Delta_\Omega$ is defined as
$\Delta_\Omega: \ell^2(\Omega,\mu) \to \ell^2(\Omega,\mu)$,
$$\Delta_\Omega f = (\Delta \tilde{f})_{|\Omega}.$$ Thus for $x\in
\Omega$ the Dirichlet Laplace operator is pointwise defined by
$$\Delta_\Omega f(x) = f(x) - \frac{1}{\mu(x)}\sum_{y\in\Omega} \mu_{xy}f(y) =
\tilde{f}(x) - \frac{1}{\mu(x)}\sum_{y\in V}
\mu_{xy}\tilde{f}(y).$$ A simple calculation shows that
$\Delta_\Omega$ is a positive self-adjoint operator. We arrange
the eigenvalues of the Dirichlet Laplace operator $\Delta_\Omega$
in increasing order, i.e. $\lambda_1(\Omega)\leq \lambda_2(\Omega)
\leq \ldots \leq \lambda_N(\Omega) $, where $N$ is the cardinality
of the set $\Omega$, i.e. $N = \sharp \Omega$. In the following we
will also denote the largest Dirichlet eigenvalue
$\lambda_N(\Omega)$ by $\lambda_\mathrm{max}(\Omega)$.

It is well known that the spectra of the Laplace operator $\Delta$
and the Laplace operator with Dirichlet boundary conditions
$\Delta_\Omega$ are connected to each other. In particular, for an
exhaustion  $\Omega\uparrow\Gamma$ we have \cite{DodziukKendall}
\begin{equation}\label{Exhaustion}\lim_{\Omega\uparrow\Gamma}\lambda_1(\Omega)=
\underline{\lambda}(\Gamma) \text{ and }
\lim_{\Omega\uparrow\Gamma}\lambda_\mathrm{max}(\Omega)=
\overline{\lambda}(\Gamma).\end{equation} Note that these limits
are well defined since $\lambda_1(\Omega)$ and
$\lambda_\mathrm{max}(\Omega)$ are monotone in $\Omega$. Because
of this connection our strategy is to first study the eigenvalues
of the Dirichlet Laplace operator and then use an exhaustion of
the graph to estimate the top and the bottom of the spectrum of
the Laplace operator $\Delta$. In particular, we obtain estimates
for the essential spectrum of $\Delta$.

\section{The Dirichlet Laplace operator}\label{Section3}
The Dirichlet Laplace operator has the following basic properties:
\begin{lemma}[Basic properties of $\Delta_\Omega$]\label{BasicProperties}
\begin{itemize}\item[]
\item[$(i)$] $0<\lambda_1(\Omega)\leq 1-\frac{1}{\sharp
\Omega}\sum_{x\in\Omega}\frac{\mu_{xx}}{\mu(x)}\leq 1.$
\item[$(ii)$] $\lambda_1(\Omega)$ is a simple eigenvalue.
\item[$(iii)$] The eigenfunction $f_1$ corresponding to
$\lambda_1(\Omega)$ satisfies $f_1(x)>0$ or $f_1(x)<0$ for all
$x\in \Omega$.
\item[$(iv)$]$\lambda_1(\Omega)+\lambda_\mathrm{max}(\Omega) \leq
2- 2\min_{x\in \Omega}\frac{\mu_{xx}}{\mu(x)}\leq 2$. \item[$(v)$]
$\lambda_1(\Omega)$ ($\lambda_\mathrm{max}(\Omega)$) is
nonincreasing (nondecreasing) when $\Omega$ increases.
 \item[$(vi)$] The first eigenvalue is given by the Rayleigh
 quotient
  \begin{eqnarray} \label{Rayleigh
Lambda1}\lambda_1(\Omega) &=& \inf_{f\in\ell^2(\Omega,\mu)\neq
0}\frac{(\Delta_\Omega f,f)_\mu}{(f,f)_\mu}
\\&=&
\inf_{\tilde{f}\neq 0}\frac{\frac{1}{2} \sum_{x,y\in
V}\mu_{xy}(\tilde{f}(x)-\tilde{f}(y))^2}{\sum_{x\in V} \mu(x)
\tilde{f}^2(x)}\end{eqnarray} and the largest eigenvalue is given
by
\begin{eqnarray} \label{Rayleigh LambdaN}\lambda_\mathrm{max}(\Omega) &=&
\sup_{f\in\ell^2(\Omega,\mu)\neq 0}\frac{(\Delta_\Omega
f,f)_\mu}{(f,f)_\mu}
\\ &=& \sup_{\tilde{f}\neq 0}\frac{\frac{1}{2}
\sum_{x,y\in V}\mu_{xy}(\tilde{f}(x)-\tilde{f}(y))^2}{\sum_{x\in
V} \mu(x) \tilde{f}^2(x)}\end{eqnarray}
\end{itemize}
\end{lemma}
\begin{proof}
All these facts are well-known and can for instance be found in
\cite{Dodziuk84, Friedman93,  Grigoryan09}.
\end{proof}
Before we continue we make the following two technical remarks.
\begin{rem}\begin{itemize}
\item[$(i)$] From the definition of the Dirichlet Laplace operator
one immediately observes that in the trivial case
$\sharp\Omega=1$, the Dirichlet Laplace operator has one single
eigenvalue which is equal to $1-\frac{\mu_{xx}}{\mu(x)}$. Because
of this we restrict ourselves from now on to subsets
$\Omega\subset V$ with $\sharp \Omega>1$. \item[$(ii)$] Often it
is convenient to sum over edges instead of vertices or vice versa.
However if there are loops in the graph we have
$$\frac{1}{2}\sum_{x,y\in V}\mu_{xy}\neq \sum_{e=(x,y)\in E}\mu_{xy}.$$ This is the case
because in the first sum the loops are only counted once. Hence in
order to replace a sum over vertices by a sum over edges we have
to adjust the edge weights. We define a new weight function
$\theta:E\to \mathbb{R}_+$ on the edge set $E$ by
$\theta_{xy}=\mu_{xy}$ for all $x\neq y\in V$ and
$\theta_{xx}=\frac{1}{2}\mu_{xx}$ for all $x\in V$. For these new
edge weights we have
$$\frac{1}{2}\sum_{x,y\in V}\mu_{xy}= \sum_{e=(x,y)\in
E}\theta_{xy}.$$ Note that $\theta:E\to \mathbb{R}_+$ coincides
with $\mu:E\to \mathbb{R}_+$ in the case $\Gamma$ has no loops.
\end{itemize}
\end{rem}

\begin{lemma}[Green's formula]\label{GreenLemma}
Let $f,g\in \ell^2(\Omega,\mu)$ then,
\begin{eqnarray}\label{Green} (\Delta_\Omega f,g)_\mu
&=& \frac{1}{2}\sum_{x,y\in V}\mu_{xy}
(\nabla_{xy}\tilde{f})(\nabla_{xy}\tilde{g})
\\&=&
\frac{1}{2}\sum_{x,y\in
V}\mu_{xy}(\tilde{f}(x)-\tilde{f}(y))(\tilde{g}(x)-\tilde{g}(y))
\\&=& \sum_{e=(x,y)\in E}\theta_{xy}(\tilde{f}(x)-\tilde{f}(y))(\tilde{g}(x)-\tilde{g}(y)),
\end{eqnarray}where $\nabla$ is the co-boundary operator. More precisely, if we fix an orientation on the edge set
the co-boundary operator of an edge $e= (x,y)$ from $x$ to $y$ is
given by $\nabla_{xy}f = f(x)-f(y)$.
\end{lemma}
\begin{proof}See \cite{Dodziuk84, Grigoryan09}.
\end{proof}
We recall here the well-known eigenvalue interlacing theorem
sometimes also referred to as the {\it inclusion principle} (cf.
\cite{Horn90}).
\begin{lemma}\label{interlacing}
Let $A$ be a $N\times N$ Hermitian matrix, let $r$ be an integer
with $1\leq r\leq N-1$ and let $A_{N-r}$ denote the
$(N-r)\times(N-r)$ principle submatrix of $A$ obtained by deleting
$r$ rows and the corresponding columns from $A$. For each integer
$k$ such that $1\leq k\leq N-r$ the eigenvalues of $A$ and
$A_{N-r}$ satisfy
$$\lambda_k(A)\leq \lambda_k(A_{N-r})\leq \lambda_{k+r}(A).$$
\end{lemma}
Lemma \ref{interlacing} yields the following generalization of
Lemma \ref{BasicProperties} $(v)$.
\begin{coro}
Let $\Omega_2\subset\Omega_1\subset V$ such that
$\sharp\Omega_1=N$ and $\sharp\Omega_2=N-r$. Then the Dirichlet
eigenvalues of $\Delta_{\Omega_1}$ and $\Delta_{\Omega_2}$
interlace, i.e.
\begin{equation}\label{interlacing3}\lambda_k(\Omega_1)\leq
\lambda_k(\Omega_2)\leq \lambda_{k+r}(\Omega_1)\end{equation} for
all integers $1\leq k\leq N-r$. In particular,
$$\lambda_1(\Omega_1) \leq \lambda_1(\Omega_2)
\text{ and } \lambda_\mathrm{max}(\Omega_2)\leq
\lambda_\mathrm{max}(\Omega_1).$$
\end{coro}
\begin{proof}
Looking at a matrix representation of the Dirichlet Laplace
operator of a subset $\Omega$ (also denoted by $\Delta_\Omega$) we
observe that $\Delta_\Omega = I_\Omega - D_\Omega^{-1}W_\Omega$ is
not Hermitian or in this case real symmetric. Here $I_\Omega$,
$D_\Omega$ and $W_\Omega$ are the identity matrix, the diagonal
matrix of vertex degrees and the weighted adjacency matrix of
$\Omega$, respectively. Note that since $\Omega$ is connected,
$\mu(x)>0$ for all $x\in \Omega$ and hence $D^{-1}$ always exists.
So we cannot directly apply Lemma \ref{interlacing} to our
Dirichlet Laplace operator $\Delta_\Omega$. However, Chung's
version of the Dirichlet Laplace operator \cite{Chung97}
$\mathcal{L}_\Omega = I_\Omega - D_\Omega^{-\frac{1}{2}}W_\Omega
D_\Omega^{-\frac{1}{2}}$ is real symmetric. Hence we can apply
Lemma \ref{interlacing} and obtain that the eigenvalues of
$\mathcal{L}_{\Omega_1}$ and $\mathcal{L}_{\Omega_2}$ interlace.
Now we observe that for any $\Omega$ the Laplacians
$\Delta_\Omega$ and $\mathcal{L}_\Omega$ satisfy
$$\Delta_\Omega = D_\Omega^{-\frac{1}{2}}\mathcal{L}_\Omega D_\Omega^{\frac{1}{2}},$$
i.e. $\Delta_\Omega$ and $\mathcal{L}_\Omega$ are similar to each
other and hence have the same spectrum. Since this holds for all
$\Omega$ it follows that if the eigenvalues of
$\mathcal{L}_{\Omega_1}$ and $\mathcal{L}_{\Omega_2}$ satisfy
(\ref{interlacing3}), then the same is true for the eigenvalues of
$\Delta_{\Omega_1}$ and $\Delta_{\Omega_2}$.
\end{proof}
For the rest of this section we have a closer look at the
Dirichlet eigenvalues of subsets $\Omega$ that are bipartite. We
say a subset $\Omega\subset V$ is bipartite if $\Omega$ is a
bipartite induced subgraph of $\Gamma$. Recall that a graph
(subgraph) is bipartite if and only if it does not contain a cycle
of odd length.
\begin{lemma}\label{symmetric eigenvalues}Let $\Omega$ be bipartite. Then the eigenvalues of
the Dirichlet Laplace operator satisfy: with $\lambda(\Omega)$,
$2-\lambda(\Omega)$ is also an eigenvalue of $\Delta_\Omega$, i.e.
the spectrum is symmetric about $1$.
\end{lemma}
\begin{proof} The proof is very simple - it relies on the observation that
if $f$ is an eigenfunction for $\lambda(\Omega)$, then for a
bipartition $U,\overline{U}$ of $\Omega$
$$g(x) = \left\{\begin{array}{cc}f(x) & \text{ if } x\in U\\
-f(x) & \text{ if } x\in \overline{U}
\end{array}\right.$$
is also an eigenfunction for $\Delta_\Omega$ corresponding to the
eigenvalue $2-\lambda(\Omega)$.
\end{proof}For this result we obtain immediately the following
useful corollaries.
\begin{coro}\label{Largest Eigenvalue is simple}
The largest eigenvalue $\lambda_\mathrm{max}(\Omega)$ is simple if
$\Omega$ is bipartite.
\end{coro}
\begin{proof}This is a direct consequence of Lemma
\ref{BasicProperties} $(ii)$ and Lemma \ref{symmetric
eigenvalues}.
\end{proof}
\begin{coro}Let $\Omega$ be bipartite. Then the eigenfunction
$f_\mathrm{max}$ corresponding to the largest eigenvalue satisfies
$f_\mathrm{max}(x) \neq 0$ for all $x \in \Omega$.
\end{coro}
\begin{proof}From the proof of the last lemma we know that if
$f_1$ is the eigenfunction for the smallest eigenvalue then the
eigenfunction for the largest eigenvalue is given by
\begin{equation}\label{explicit expression for fmax}f_\mathrm{max}(x) = \left\{\begin{array}{cc}f_1(x) & \text{ if } x\in U\\
-f_1(x) & \text{ if } x\in \overline{U}.
\end{array}\right.\end{equation} By Lemma \ref{BasicProperties} $(iii)$ we have $f_1(x) \neq 0$ for all $x\in \Omega$ and
hence $f_\mathrm{max}(x)\neq 0$ for all $x\in \Omega$.
\end{proof}

\begin{theo}\label{Lembi}
We have $\lambda_1(\Omega) + \lambda_\mathrm{max}(\Omega) = 2$ iff
$\Omega$ is bipartite.
\end{theo}
\begin{proof}By Lemma \ref{symmetric eigenvalues}, it suffices to prove that $\lambda_1(\Omega) + \lambda_\mathrm{max}(\Omega) = 2$ implies
$\Omega$ is bipartite. It is well known (see c.f.
\cite{Grigoryan09}) that $\lambda_1(\Omega) +
\lambda_\mathrm{max}(\Omega) \leq 2.$ Let $f$ be an eigenfunction
for $\lambda_\mathrm{max}(\Omega)$. Then we have
$$\lambda_\mathrm{max}(\Omega) = \frac{\frac{1}{2}\sum_{x,y}\mu_{xy}(f(x)-f(y))^2}{\sum_x
\mu(x)f(x)^2}$$ and \begin{eqnarray}\label{Estimate3}
\lambda_1(\Omega) &=& \inf_g
\frac{\frac{1}{2}\sum_{x,y}\mu_{xy}(g(x)-g(y))^2}{\sum_x
\mu(x)g(x)^2} \\&\leq& \label{Estimate4}
\frac{\frac{1}{2}\sum_{x,y}\mu_{xy}(|f|(x)-|f|(y))^2}{\sum_x
\mu(x)|f|(x)^2}.
\end{eqnarray}
We have
\begin{eqnarray}(f(x) - f(y))^2 + (|f|(x) - |f|(y))^2 &=&
 2 f(x)^2 + 2f(y)^2 - 2f(x)f(y) -
2|f|(x)|f|(y)\nonumber\\&\leq& 2 f(x)^2 + 2f(y)^2.
\label{Estimate} \end{eqnarray} Thus we have
\begin{eqnarray}\lambda_1(\Omega) + \lambda_\mathrm{max}(\Omega) \nonumber  &\overset{(*)}{\leq}&
\frac{\frac{1}{2}\sum_{x,y}\mu_{xy}[(f(x)-f(y))^2+
(|f|(x)-|f|(y))^2]}{\sum_x \mu(x)f(x)^2}\\&\overset{(**)}{\leq}&
\nonumber \frac{\sum_{x,y}\mu_{xy}(f(x)^2 + f(y)^2)}{\sum_x
\mu(x)f(x)^2}
\\&=&\frac{2\sum_{x,y}\mu_{xy}f(x)^2}{\sum_x \mu(x)f(x)^2}\nonumber \\
&=&\label{Estimate2} 2.
\end{eqnarray}
Now we have a closer look at equation (\ref{Estimate}). We have
strict inequality in (\ref{Estimate}) and hence in $(**)$ of
(\ref{Estimate2}) if $f(x)f(y)>0$ for some $x\sim y$. Suppose that
$\Omega$ is not bipartite, then $\Omega$  contains a cycle $C$ of
odd length with no repeated vertices (called circuit) and hence
there exists at least one pair of neighbors $x\sim y$ such that
$f(x)f(y)\geq 0$, (if not, $f$ has alternating signs along the
cycle $C$ which contradicts to the odd length of $C$). If
$f(x)f(y)> 0$ we are done. Otherwise there exists at least one
vertex $x\in \Omega$ such that $f(x)=0$ and hence $|f|(x) =0$.
From Lemma \ref{BasicProperties} $(iii)$ it follows that $|f|$ is
not an eigenfunction for $\lambda_1(\Omega).$ Hence we have strict
inequality in (\ref{Estimate4}) and then also in $(*)$ of
(\ref{Estimate2}). Thus $\lambda_1(\Omega) +
\lambda_\mathrm{max}(\Omega)< 2$ if $\Omega$ is not bipartite. It
remains to show that $\lambda_1(\Omega) +
\lambda_\mathrm{max}(\Omega)= 2$ if $\Omega$ is bipartite. This
follows immediately from Lemma \ref{symmetric eigenvalues} since
the eigenvalues of the Dirichlet Laplace operator $\Delta_\Omega$
are symmetric about 1 if $\Omega$ is bipartite.
 \end{proof}

\section{The Cheeger and the dual Cheeger estimate}\label{Section4}
In this section we introduce the Cheeger constant $h$ and the dual
Cheeger constant $\bar{h}$ of a graph and show how they can be
used to estimate the smallest and the largest eigenvalue of the
Dirichlet Laplace operator from above and below. Moreover, we
discuss in detail the connection between $h$ and $\bar{h}$.  This
will be particularly important in Section \ref{Section10} when we
study the essential spectrum of the Laplace operator $\Delta$.
\begin{defi}[Cheeger constant]
For any subset $\Omega\subset V$ we define the Cheeger constant
$h(\Omega)$ by
$$h(\Omega) = \inf_{\substack{\emptyset\neq U\subset \Omega\\\sharp U <\infty}}\frac{|\partial U|}{\mathrm{vol}(U)}.$$
\end{defi}

\begin{defi}[Dual Cheeger constant]
For any subset $\Omega\subset V$ we define the dual Cheeger
constant $\bar{h}(\Omega)$ by
\begin{equation}\label{dcc}\bar{h}(\Omega) =
\sup_{\substack{V_1, V_2\subset \Omega\\\sharp V_1, \sharp
V_2<\infty}}\frac{2|E(V_1,V_2)|}{\mathrm{vol}(V_1)+\mathrm{vol}(V_2)},\end{equation}
where $V_1,V_2$ are two disjoint nonempty subsets of $\Omega$.
\end{defi} We observe that the dual Cheeger constant can only be defined
if $\sharp\Omega>1$. However, as discussed above we exclude the
case $\sharp\Omega=1$ since it is trivial anyway. The Cheeger
constant and the dual Cheeger constant are related to each other
in the following way:
\begin{theo}\label{h and hbar} We have
 $$\bar{h}(\Omega)\leq 1-h(\Omega)$$ and
equality holds if $\Omega$ is bipartite.
\end{theo}
\begin{proof}
Let $V_1,V_2\subset \Omega$ be two nonempty disjoint subsets of
$\Omega$. Note that the volume of the subset $V_1$, can be written
in the form (cf. eq. (\ref{volumeformula}))
$$\mathrm{vol}(V_1) = |E(V_1,V_2)| + |E(V_1,V_1)| + |E(V_1,(V_1\cup
V_2)^c)|$$and a similar expression holds for the volume of $V_2$.
Hence we have
\begin{eqnarray}\label{VolumeV1V2}\nonumber \mathrm{vol}(V_1)+\mathrm{vol}(V_2) &=& 2 |E(V_1,V_2)| + |E(V_1,V_1)| + |E(V_2,V_2)|
+ |\partial(V_1\cup V_2)|
\\&\geq& 2 |E(V_1,V_2)| + |\partial (V_1\cup V_2)|.
\end{eqnarray}
From this it follows that
$$\frac{2|E(V_1,V_2)|}{\mathrm{vol}(V_1)+\mathrm{vol}(V_2)}
\leq 1 - \frac{|\partial(V_1\cup V_2)|}{\mathrm{vol}(V_1\cup
V_2)}.$$ Taking the supremum over all nonempty disjoint subsets
$V_1,V_2\subset \Omega$ we get \begin{eqnarray}\ \ \
\bar{h}(\Omega) &\overset{(*)}{=}& \sup_{V_1,V_2\subset
\Omega}\frac{2|E(V_1,V_2)|}{\mathrm{vol}(V_1)+\mathrm{vol}(V_2)}
\leq \sup_{V_1,V_2\subset \Omega} \left(1 - \frac{|\partial
(V_1\cup V_2)|}{\mathrm{vol}(V_1\cup V_2)}\right)\label{estimate5}
\\&=& 1 - \inf_{U\subset \Omega, \sharp U\geq 2}\frac{|\partial
U|}{\mathrm{vol}(U)} \overset{(**)}{=} 1 -
\inf_{U\subset\Omega}\frac{|\partial U|}{\mathrm{vol}(U)} =
1-h(\Omega),\nonumber\end{eqnarray} where $(**)$ follows from the
fact that $\inf_{U\subset\Omega}\frac{|\partial
U|}{\mathrm{vol}(U)}$ cannot be achieved on singletons, i.e.
$\sharp U=1$. More precisely, if $\sharp U =1$ then clearly
$\frac{|\partial U|}{\mathrm{vol}(U)}=1$. However, since $\sharp
\Omega>1$ and $\Omega$ is connected, we can find $W\subset \Omega$
such that $|E(W,W)|>0$. Thus we have $$\frac{|\partial
W|}{\mathrm{vol}(W)}= \frac{|\partial W|}{|\partial W| +
|E(W,W)|}<1$$ which contradicts that $U$ achieves the infimum.

Now if $\Omega$ is bipartite, we claim that in $(*)$ of
\eqref{estimate5} the supremum is obtained for two subsets
$V_1,V_2\subset\Omega$ that satisfy $|E(V_1,V_1)| =
|E(V_2,V_2)|=0$. If it is not the case, there exists $V_1^\prime,
V_2^\prime\subset \Omega$ that achieve the supremum in $(*)$ of
\eqref{estimate5} and satisfy $|E(V_1^\prime, V_1^\prime)|\neq 0$
or $|E(V_2^\prime, V_2^\prime)|\neq 0$. Since $\Omega$ is
bipartite (i.e. in particular $\mu_{xx}=0$ for any $x\in V$), we
can find nonempty disjoint subsets $V_1, V_2\subset \Omega$ that
satisfy $V_1\cup V_2 = V_1^\prime\cup V_2^\prime$ and
$|E(V_1,V_1)| = |E(V_2,V_2)|=0$. Then we have
$$\frac{1}{2}\sum_{x,y \in V_1\cup V_2}\mu_{xy} = |E(V_1,V_2)| +\frac{1}{2}
|E(V_1,V_1)|+ \frac{1}{2}|E(V_2,V_2)| = |E(V_1,V_2)| $$ and
$$\frac{1}{2}\sum_{x,y \in V_1^\prime\cup V_2^\prime}\mu_{xy} =
|E(V_1^\prime,V_2^\prime)| + \frac{1}{2}|E(V_1^\prime,
V_1^\prime)| +\frac{1}{2} |E(V_2^\prime,
V_2^\prime)|>|E(V_1^\prime, V_2^\prime)|,$$ where we used in the
last equation that  $|E(V_1^\prime, V_1^\prime)|\neq 0$ or
$|E(V_2^\prime, V_2^\prime)|\neq 0$. By construction we have
$V_1\cup V_2 = V_1^\prime\cup V_2^\prime$ which implies
$|E(V_1^\prime,V_2^\prime)|< |E(V_1,V_2)|$ and $\mathrm{vol}(V_1)
+ \mathrm{vol}(V_2) =\mathrm{vol}(V_1^\prime) +
\mathrm{vol}(V_2^\prime)$. This is a contradiction to the
assumption that $V_1^\prime, V_2^\prime$ achieve the supremum in
$(*)$ of \eqref{estimate5}.

Using \eqref{VolumeV1V2} and the claim in (\ref{estimate5}), we
have
\begin{eqnarray}\ \ \ \bar{h}(\Omega) &=&
\sup_{\substack{V_1,V_2\subset
\Omega\\|E(V_1,V_1)|=0\\|E(V_2,V_2)|=0}}\frac{2|E(V_1,V_2)|}{\mathrm{vol}(V_1)+\mathrm{vol}(V_2)}
= \sup_{\substack{V_1,V_2\subset
\Omega\\|E(V_1,V_1)|=0\\|E(V_2,V_2)|=0}} \left(1 - \frac{|\partial
(V_1\cup V_2)|}{\mathrm{vol}(V_1\cup V_2)}\right)\nonumber  \\
&=& \sup_{V_1,V_2\subset \Omega} \left(1 - \frac{|\partial
(V_1\cup V_2)|}{\mathrm{vol}(V_1\cup V_2)}\right)\label{estimate6}
= 1 - \inf_{\substack{U\subset\Omega\\ \sharp U\geq 2}}\frac{|\partial U|}{\mathrm{vol}(U)}\\
&=& 1 - \inf_{U\subset\Omega}\frac{|\partial U|}{\mathrm{vol}(U)}
= 1-h(\Omega),\nonumber\end{eqnarray} where in \eqref{estimate6}
we use the fact (since $\Omega$ is bipartite) that for any
disjoint $V_1,V_2$ there exist disjoint $U_1$ and $U_2$ such that
$V_1\cup V_2 = U_1\cup U_2$ and $|E(U_1,U_1)| = |E(U_2,U_2)|=0$.
\end{proof}
The next example shows that the converse of the second assertion
in Theorem \ref{h and hbar} is in general not true, i.e.
$h(\Omega) +\bar{h}(\Omega) =1$ does not imply that $\Omega$ is
bipartite.
\begin{example}\label{counterexample1} Let $G$ be the standard
lattice $\mathds{Z}^2$ with one more edge, $((0,1),(1,0))$, and
$\Omega'$ a finite subset of $G$ containing the origin and the
additional edge, i.e. $(0,0)\in \Omega'$ and $((0,1),(1,0))\in
E(\Omega',\Omega').$ Denote by $M:=\mathrm{vol}(\Omega')$ the
volume of $\Omega'.$ Moreover, let $K_{m,n}$ be a large complete
bipartite graph such that $K:=\mathrm{vol}(K_{m,n})\geq M.$ By
adding an edge that connects the origin in $\Omega'$ to a vertex
in $K_{m,n},$ we obtain an infinite graph, $\Gamma=G\cup K_{m,n}$,
see Figure \ref{Fig.lattice}.  Let $\Omega:=\Omega'\cup K_{m,n}.$
First of all we note that $\Omega$ is not bipartite. However, we
will show that $h(\Omega) +\bar{h}(\Omega) =1$ holds. We claim
that $U_0:=K_{m,n}$ achieves the Cheeger constant in $\Omega$,
i.e. $\frac{|\partial U_0|}{\mathrm{vol}(U_0)}\leq \frac{|\partial
U|}{\mathrm{vol}(U)},$ for any $U\subset \Omega$. Note that by
construction of $\Omega$, $U_0$ is the only subset of $\Omega$
s.t. $|\partial U|=1$, i.e. $|\partial U|=1$ implies that $U =
U_0$. Thus for all $U\neq U_0$
$$\frac{|\partial U|}{\mathrm{vol}(U)}\geq
\frac{2}{\mathrm{vol}(U)}\geq \frac{2}{\mathrm{vol}(\Omega)} \geq
\frac{2}{K+M+2}.$$ Since $M\leq K,$
$$\frac{|\partial U|}{\mathrm{vol}(U)}\geq \frac{1}{K+1}=
\frac{|\partial U_0|}{\mathrm{vol}(U_0)}.$$ This proves our claim
that $U_0$ achieves the Cheeger constant. Moreover by choosing
$V_1,\ V_2$ to be the bipartition of $K_{m,n}$ we have
$$\bar{h}(\Omega) = \sup_{V_1,V_2\subset \Omega}\frac{2|E(V_1,V_2)|}{\mathrm{vol}(V_1) + \mathrm{vol}(V_2)}\geq \frac{K}{K+1}$$ and thus
$$h(\Omega) + \bar{h}(\Omega)\geq \frac{1}{K+1}+
\frac{K}{K+1}=1.$$ Together with Theorem \ref{h and hbar} this
implies that $h(\Omega) +\bar{h}(\Omega) =1$ although $\Omega$ is
not bipartite.
\end{example}
\begin{figure}\begin{center}
\includegraphics[width =
 7cm]{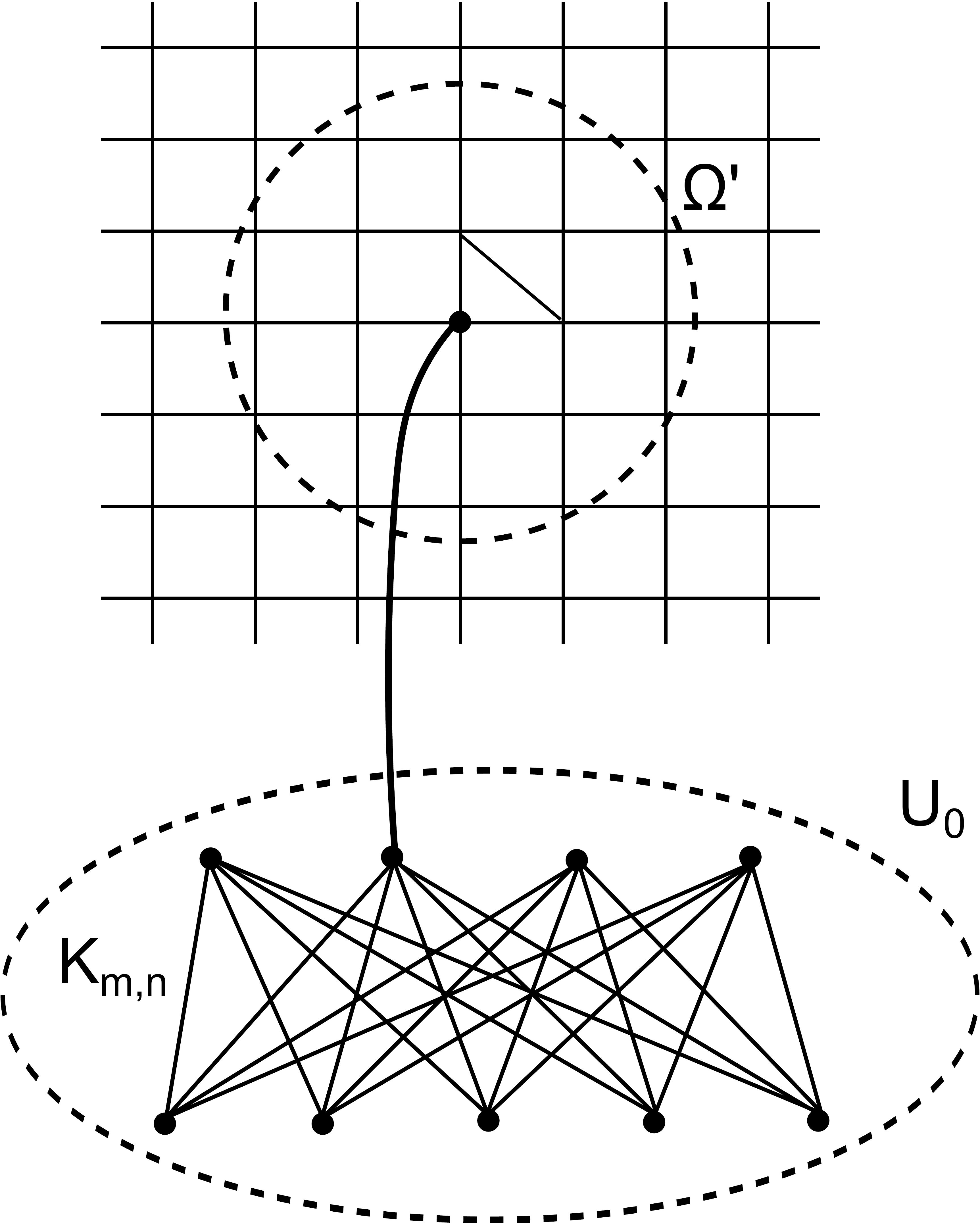}\caption{\label{Fig.lattice}The graph in  Example \ref{counterexample1}.}
\end{center}
\end{figure}

\begin{rem}The last example showed that $h(\Omega)
+\bar{h}(\Omega) =1$ does not imply that $\Omega$ is bipartite.
However, if $h(\Omega) +\bar{h}(\Omega) =1$ we can show that the
partition $V_1,V_2$ that achieves $\bar{h}(\Omega)$ is bipartite.
This can be seen as follows: Let $V_1,V_2$ be the partition that
achieves $\bar{h}(\Omega)$ and let $U = V_1\cup V_2$. Then by
definition  $\frac{|\partial U|}{\mathrm{vol}(U)}\geq h(\Omega)$
and thus we have
$$1 = h(\Omega) +\bar{h}(\Omega) \leq \frac{2|E(V_1,V_2)| + |\partial(V_1\cup V_2)|}
{2|E(V_1,V_2)| + |E(V_1,V_1)| + |E(V_2,V_2)| + |\partial(V_1\cup
V_2)|}.$$ This implies that $|E(V_1,V_1)| =  |E(V_2,V_2)|=0$ which
yields that the partition $V_1,V_2$ is bipartite.
\end{rem}
In order to give a lower bound for $\bar{h}(\Omega)$ in terms of
$(1-h(\Omega))$ we recall the following theorem from
\cite{Bauer12}:
\begin{theo}\label{surgery} Let $\Gamma$ be a graph without self-loops. Then for any finite $U \subset V$ there exists a partition
$V_1\cup V_2 =U$ such that \begin{equation}\label{surgery
equation}|E(V_1,V_2)| \geq \max\{|E(V_1,V_1)|,
|E(V_2,V_2)|\}.\end{equation}
\end{theo}
\begin{proof} For completeness, we include a proof here.
Suppose the assertion is not true, that is for any partition
$V_1\cup V_2=U$ we have
\begin{equation}\label{converse1}
|E(V_1,V_2)| < \max\{|E(V_1,V_1)|, |E(V_2,V_2)|\}.
\end{equation}
We start with an arbitrarily partition $V_1\cup V_2=U.$ Without
loss of generality, we assume $|E(V_1,V_2)|<|E(V_1,V_1)|,$ i.e.
$$\sum_{x\in V_1}\sum_{y\in V_2}\mu_{xy}<\sum_{x\in V_1}\sum_{z\in V_1}\mu_{xz}.$$
Then there exists a vertex $x\in V_1$ such that $$0\leq \sum_{y\in
V_2}\mu_{xy}=|E(\{x\}, V_2)|<\sum_{z\in
V_1}\mu_{xz}=|E(\{x\},V_1)|.$$ Since $|E(\{x\},V_1)|>0$ it follows
that $\sharp V_1\geq 2$ because otherwise $V_1=\{x\}$ and hence
$E(\{x\},V_1)=0$. We define a new partition, $V_1'\cup V_2'=U,$ as
$$V_1'=V_1\backslash \{x\}\neq\emptyset,\ V_2'=V_2\cup \{x\}.$$
Then it is evident that
$$|E(V_1',V_2')|=|E(V_1,V_2)|+|E(\{x\},V_1)|-|E(\{x\},V_2)|\geq |E(V_1,V_2)|+\epsilon_0,$$
where
\begin{eqnarray*}\epsilon_0&:=&\min \{|E(\{x\},V_1)|-|E(\{x\},V_2)|:\\ && V_1\cup V_2=U, V_1\cap V_2=\emptyset, x\in U, |E(\{x\},V_1)|-|E(\{x\},V_2)|>0\}.\end{eqnarray*}
Note that since $U$ is finite and by our assumption
\eqref{converse1}, it follows that $\epsilon_0>0.$

Since \eqref{converse1} holds for all partitions, we may carry out
this process for infinitely many times. That is we may obtain
arbitrary large $|E(V_1,V_2)|$ which contradicts that
$\mathrm{vol}(U)<\infty.$
\end{proof}

 \begin{theo}\label{h and hbar 3} If there are no self-loops
in the graph, then
\begin{equation} \label{h and hbar 2}\frac{1}{2}(1 - h(\Omega))\leq
\bar{h}(\Omega).\end{equation}
\end{theo}
\begin{proof}We have
\begin{eqnarray*}
\bar{h}(\Omega) &=& \sup_{\substack{U\subset \Omega\\\sharp U \geq
2}} \sup_{\substack{V_1,V_2\\V_1\cup V_2 = U}}
\frac{2|E(V_1,V_2)|}{\mathrm{vol}(V_1) + \mathrm{vol}(V_2)}\\&=&
\sup_{\substack{U\subset \Omega\\\sharp U \geq 2}}
\sup_{\substack{V_1,V_2\\V_1\cup V_2 = U}}
\left(\frac{2|E(V_1,V_2)| + \frac{1}{2}|\partial(V_1\cup
V_2)|}{\mathrm{vol}(V_1) + \mathrm{vol}(V_2)} - \frac{1}{2}\frac{
|\partial(V_1\cup V_2)|}{\mathrm{vol}(V_1) +
\mathrm{vol}(V_2)}\right)\\&=& \sup_{\substack{U\subset
\Omega\\\sharp U \geq 2}} \sup_{\substack{V_1,V_2\\V_1\cup V_2 =
U}} \left(\frac{2|E(V_1,V_2)| + \frac{1}{2}|\partial(V_1\cup
V_2)|}{2|E(V_1,V_2)|+ |E(V_1,V_1)| +|E(V_2,V_2)| +
|\partial(V_1\cup V_2)| }\right.\\&&- \left.\frac{1}{2}\frac{
|\partial(V_1\cup V_2)|}{\mathrm{vol}(V_1) +
\mathrm{vol}(V_2)}\right).
\end{eqnarray*}
Using Theorem \ref{surgery} we obtain
\begin{eqnarray*}
\bar{h}(\Omega) &\geq& \sup_{\substack{U\subset \Omega\\\sharp U
\geq 2}} \left(\frac{1}{2} - \frac{1}{2}\frac{|\partial
U|}{\mathrm{vol}(U)}\right)\\&=& \frac{1}{2} - \frac{1}{2}
\inf_{U\subset \Omega} \frac{|\partial U|}{\mathrm{vol}(U)}\\&=&
\frac{1}{2}(1 - h(\Omega)),
\end{eqnarray*}where we again observe that the infimum of $\frac{|\partial
U|}{\mathrm{vol}(U)}$ cannot be achieved by a singleton.
\end{proof}

\begin{rem}
\begin{itemize}\item[$(i)$]
It is clear that for weighted graphs with self-loops
$\bar{h}(\Omega)$ can be arbitrarily close to zero. Consider for
instance a subgraph consisting of two vertices connected by an
edge. If one of the two vertices (say vertex $x$) has a self-loop
with weight $\mu_{xx}$, then both $\bar{h}(\Omega)$ and
$h(\Omega)\to 0$ as $\mu_{xx}\to \infty$. \item[$(ii)$]If we allow
self-loops in the graph the best lower bound for $\bar{h}(\Omega)$
that we can obtain is $\bar{h}(\Omega)\geq \frac{2 \min_{x,y\in
\Omega}\mu_{xy}}{\mathrm{vol(\Omega)}}$. This estimate is sharp
for the graph discussed in $(i)$ as $\mu_{xx}\to \infty$.
\item[$(iii)$] Using an argument by Alon \cite{Alon96} and
Hofmeister and Lefmann \cite{Hofmeister98} (see also Scott
\cite{Scott05}) we can show that even strict inequality holds in
\eqref{h and hbar 2}.
\end{itemize}
\end{rem}

We introduce the following notation:
\begin{defi}For a function $g\in
\ell^2(\Omega,\mu)$ we define $P(g):= \{x\in \Omega: g(x)>0\}$ and
$N(g):= \{x\in \Omega: g(x)<0\}$.
\end{defi}
\begin{defi}[Auxiliary Cheeger constant]
For a function $g\in \ell^2(\Omega,\mu)$ the auxiliary Cheeger
constant $h(\Omega,g)$ is defined by \be \label{hg} h(\Omega,g) :=
\min_{\emptyset \neq U \subseteq P(g)}\frac{|E(U,U^c)|}
{\mathrm{vol}(U)}.\qe
\end{defi}
\begin{rem}\label{Remark} Since for every function $g\in\ell^2(\Omega,\mu)$ we have
$P(g)\subset\Omega$ it is obvious that $h(\Omega,g)\geq h(\Omega)$
for all $g\in\ell^2(\Omega,\mu)$.
\end{rem}We
recall the following two lemmata from \cite{Diaconis91}, see also
\cite{Bauer12}.
\begin{lemma}\label{Lemma1}
For every function $g\in \ell^2(\Omega,\mu)$ we have
\[1+\sqrt{1-h^2(\Omega, g)}\geq \frac{\sum_{e=(x,y)}\mu_{xy}(\tilde{g}_+(x)-\tilde{g}_+(y))^2}
{\sum_x \mu(x)\tilde{g}_+(x)^2}\geq 1-\sqrt{1-h^2(\Omega, g)}, \]
where $g_+$ is the positive part of $g$, i.e.
\[g_+(x)=\left\{\begin{array}{ccc} g(x) & \mbox{ if}&  x\in P(g)\\0 &&
    else. \end{array}\right.\]
\end{lemma}
\begin{proof}
First, we write
\begin{eqnarray*}W&:=& \frac{\sum_{e=(x,y)}\mu_{xy}(\tilde{g}_+(x)-\tilde{g}_+(y))^2}{\sum_x
\mu(x)\tilde{g}_+(x)^2}\\&=&
\frac{\sum_{e=(x,y)}\theta_{xy}(\tilde{g}_+(x)-\tilde{g}_+(y))^2}{\sum_x \mu(x)\tilde{g}_+(x)^2}\\
&=&
\frac{\sum_{e=(x,y)}\theta_{xy}(\tilde{g}_+(x)-\tilde{g}_+(y))^2
\sum_{e=(x,y)}\theta_{xy}(\tilde{g}_+(x)+\tilde{g}_+(y))^2}{\sum_x
\mu(x)\tilde{g}_+(x)^2
\sum_{e=(x,y)}\theta_{xy}(\tilde{g}_+(x)+\tilde{g}_+(y))^2}
\\&=:& \frac{I}{II}.\end{eqnarray*} Using the Cauchy-Schwarz
inequality we obtain
\[I \geq \left(\sum_{e=(x,y)}\theta_{xy}|\tilde{g}_+(x)^2-\tilde{g}_+(y)^2|\right)^2
=\left(\sum_{e=(x,y)}\mu_{xy}|\tilde{g}_+(x)^2-\tilde{g}_+(y)^2|\right)^2.\]
Now we have
\begin{eqnarray*}\sum_{e=(x,y)}\mu_{xy}|\tilde{g}_+(x)^2-\tilde{g}_+(y)^2|&=& \sum_{e=(x,y):\tilde{g}_+(x)> \tilde{g}_+(y)}\mu_{xy}(\tilde{g}_+(x)^2-\tilde{g}_+(y)^2)\\&=& 2
\sum_{e=(x,y):\tilde{g}_+(x)>
\tilde{g}_+(y)}\mu_{xy}\int_{\tilde{g}_+(y)}^{\tilde{g}_+(x)}t
dt\\&=& 2 \int_{0}^{\infty} \sum_{e=(x,y):\tilde{g}_+(y)\leq t<
\tilde{g}_+(x)}\mu_{xy} \,tdt.
\end{eqnarray*}Note that $\sum_{e=(x,y):\tilde{g}_+(y)\leq t< \tilde{g}_+(x)}\mu_{xy}=|E(P_t,P_t^c)|$
where $P_t:= \{x:\tilde{g}_+(x)>t\}$. Using \rf{hg} we obtain,
\begin{eqnarray*}\sum_{e=(x,y)}\mu_{xy}|\tilde{g}_+(x)^2-\tilde{g}_+(y)^2|&\geq& 2
h(\Omega, g)\int_0^\infty\mathrm{vol}(P_t)tdt\\&=&  2 h(\Omega,
g)\int_0^\infty\sum_{x:\tilde{g}_+(x)>t}\mu(x)tdt\\&=&  2
h(\Omega, g)\sum_{x\in V}\mu(x)\int_0^{\tilde{g}_+(x)}tdt
\\&=& h(\Omega, g)\sum_x\mu(x)\tilde{g}_+(x)^2
\end{eqnarray*}and so it follows  that \[I\geq h^2(\Omega,g)(\sum_x\mu(x)\tilde{g}_+(x)^2)^2.\]
\begin{eqnarray*}II&=& \sum_x \mu(x)\tilde{g}_+(x)^2
\sum_{e=(x,y)}\theta_{xy}(\tilde{g}_+(x)+\tilde{g}_+(y))^2 \\&=&
\sum_x \mu(x)\tilde{g}_+(x)^2 (\sum_x\mu(x)\tilde{g}_+(x)^2 +
\sum_{x,y}\mu_{xy}\tilde{g}_+(x)\tilde{g}_+(y))\\&=& \sum_x
\mu(x)\tilde{g}_+(x)^2 (2\sum_x\mu(x)\tilde{g}_+(x)^2 -
\sum_{e=(x,y)}\mu_{xy}(\tilde{g}_+(x)-\tilde{g}_+(y))^2)\\&=&
(2-W)(\sum_x \mu(x)\tilde{g}_+(x)^2)^2.
\end{eqnarray*}
Combining everything we obtain,
\[W\geq \frac{h^2(\Omega, g)}{(2-W)}\] and  consequently
\[1+\sqrt{1-h^2(\Omega, g)}\geq W \geq 1-\sqrt{1-h^2(\Omega, g)}.\]
\end{proof}
The second observation that we need to prove the Cheeger
inequality is the following lemma, see \cite{Diaconis91}:
\begin{lemma}\label{Lemma2}
For every non-negative real number $\lambda$ and
$g\in\ell^2(\Omega,\mu)$ we have
\[\lambda \geq \frac{\sum_{e=(x,y)}\mu_{xy}(\tilde{g}_+(x)-\tilde{g}_+(y))^2}{\sum_x
\mu(x)\tilde{g}_+(x)^2} \] if $\Delta_\Omega g(x)\leq \lambda
g(x)$ for all $x\in P(g)$.
\end{lemma}
\begin{proof}On the one hand we have
\begin{eqnarray*}(\Delta_\Omega g, g_+)_\mu &=& \sum_{x\in
\Omega}\mu(x)\Delta_\Omega g(x)g_+(x)  \\&=& \sum_{x\in
P(g)}\mu(x)\Delta_\Omega g(x)g_+(x)\\&\leq& \lambda\sum_{x\in
P(g)}\mu(x) g_+(x)g_+(x) \\&=& \lambda\sum_{x\in V}\mu(x)
\tilde{g}_+(x)\tilde{g}_+(x),\end{eqnarray*} where we used our
assumption. On the other hand we have
\begin{eqnarray*}(\Delta_\Omega g, g_+)_\mu &=&
\sum_{e=(x,y)\in
E}\theta_{xy}(\tilde{g}(x)-\tilde{g}(y))(\tilde{g}_+(x)-\tilde{g}_+(y))
\\&\geq&\sum_{e=(x,y)\in E}\theta_{xy}(\tilde{g}_+(x)-\tilde{g}_+(y))^2,
\end{eqnarray*}where we used the Green's formula, see Lemma \ref{GreenLemma}.
\end{proof}

\begin{theo}[Cheeger inequality cf. \cite{DodziukKarp}]\label{CheegerInequality}
We have
\begin{equation}\label{eq7}1 - \sqrt{1 - h^2(\Omega)}\leq \lambda_1(\Omega)\leq h(\Omega).\end{equation}
\end{theo}
\begin{proof}
For completeness we give a proof here.

First we show that $\lambda_1(\Omega)\leq h(\Omega)$.  We consider
the following function:
$$f(x) = \left\{\begin{array}{cc} 1&\text{if } x\in U\subset \Omega\\ 0 &
\text{else.}
\end{array}\right.$$
Using this function in the Rayleigh quotient (\ref{Rayleigh
Lambda1}) yields
\begin{eqnarray*}
\lambda_1(\Omega) &\leq& \frac{\frac{1}{2}
\sum_{x,y}\mu_{xy}(f(x)-f(y))^2}{\sum_x \mu(x) f^2(x)}\\&=&
\frac{\sum_{x\in U}\sum_{y\in U^c}\mu_{xy}}{\mathrm{vol}(U)} \\&=&
\frac{|\partial U|}{\mathrm{vol}(U)}.
\end{eqnarray*}
Since this holds for all $U\subset \Omega$ we have
$$\lambda_1(\Omega) \leq \inf_{\emptyset \neq U\subset \Omega}\frac{|\partial
U|}{\mathrm{vol}(U)} =h(\Omega).$$

The inequality $1 - \sqrt{1 - h^2(\Omega)}\leq \lambda_1(\Omega)$
follows from Lemma \ref{Lemma1}, Lemma \ref{Lemma2} and Remark
\ref{Remark} by taking $\lambda=\lambda_1(\Omega)$ and $g=u_1$ an
eigenfunction for $\lambda_1(\Omega)$.
\end{proof}
The next theorem is the main result of this section.
\begin{theo}[Dual Cheeger inequality]\label{Mtheo1}
We have
\begin{equation}\label{eq6}2\bar{h}(\Omega) + h(\Omega) \leq \lambda_\mathrm{max}(\Omega)\leq 1 + \sqrt{1 - (1- \bar{h}(\Omega))^2}.\end{equation}
\end{theo}
\begin{proof}Let $V_1,V_2$ be two disjoint nonempty subsets of
$\Omega$. We consider the following function:
$$f(x) = \left\{\begin{array}{cc} 1&\text{if } x\in V_1\\
-1&\text{if } x\in V_2 \\
0 & \text{else.}
\end{array}\right.$$
Using this function in the Rayleigh quotient (\ref{Rayleigh
LambdaN}) yields
\begin{eqnarray*}
\lambda_\mathrm{max}(\Omega) &\geq& \frac{\frac{1}{2}
\sum_{x,y}\mu_{xy}(\tilde{f}(x)-\tilde{f}(y))^2}{\sum_x \mu(x)
\tilde{f}^2(x)}\\&=& \frac{\sum_{x\in V_1}\sum_{y\in V_2}\mu_{xy}4
+ \sum_{x\in V_1\cup V_2}\sum_{y \in (V_1\cup
V_2)^c}\mu_{xy}}{\mathrm{vol}(V_1)+ \mathrm{vol}(V_2)}
\\&=& 2 \frac{2 |E(V_1,V_2)|}{\mathrm{vol}(V_1)+\mathrm{vol}(V_2)} + \frac{|\partial (V_1\cup V_2)|}{\mathrm{vol}(V_1\cup V_2)}
\\&\geq& 2 \frac{2 |E(V_1,V_2)|}{\mathrm{vol}(V_1)+\mathrm{vol}(V_2)} + \inf_{\substack{U\subset\Omega\\\sharp U\geq 2}} \frac{|\partial U|}{\mathrm{vol}(U)}\\
&=& 2 \frac{2 |E(V_1,V_2)|}{\mathrm{vol}(V_1)+\mathrm{vol}(V_2)} +
\inf_{U\subset\Omega} \frac{|\partial U|}{\mathrm{vol}(U)},
\end{eqnarray*} where we again used the observation that the
infimum is not achieved by a singleton.  Since this holds for all
disjoint nonempty subsets $V_1, V_2\subset \Omega$ we have
$$2\bar{h}(\Omega) + h(\Omega) \leq \lambda_\mathrm{max}(\Omega).$$

Now we prove the remaining inequality
$\lambda_\mathrm{max}(\Omega) \leq
1+\sqrt{1-(1-\overline{h}(\Omega))^2}$. When one studies the
largest eigenvalue of $\Delta_\Omega$ it is convenient to
introduce the operator $Q_\Omega=2I_\Omega -\Delta_\Omega =
I_\Omega + P_\Omega$ where $I_\Omega$ is the identity operator on
$\Omega$ and $P_\Omega$ the transition probability operator. If
$\lambda(\Omega)$ is an eigenvalue of $\Delta_\Omega$ and
corresponding eigenfunction $f$ then $f$ is also an eigenfunction
for $Q_\Omega$ and corresponding eigenvalue $\xi(\Omega) =
2-\lambda(\Omega)$. Thus, controlling the largest eigenvalue
$\lambda_\mathrm{max}(\Omega)$ of $\Delta_\Omega$ from above is
equivalent to controlling the smallest eigenvalue $\xi_1(\Omega)$
of $Q_\Omega$ from below. The smallest eigenvalue $\xi_1(\Omega)$
of $Q_\Omega$ is given by
\begin{eqnarray*}\xi_1(\Omega) &=& \inf_{\tilde{f}\neq 0} \frac{\frac{1}{2}\sum_{x,y\in V}\mu_{xy}(\tilde{f}(x) +
\tilde{f}(y))^2}{\sum_{x\in
V}\mu(x)\tilde{f}(x)^2}\\&=&\inf_{\tilde{f}\neq 0}
\frac{\sum_{e=(x,y)\in E}\theta_{xy}(\tilde{f}(x) +
\tilde{f}(y))^2}{\sum_{x\in V}\mu(x)\tilde{f}(x)^2}
\end{eqnarray*} where as above $\theta_{xy}=\mu_{xy}$ for all $x\neq y$ and
$\theta_{xx}=\frac{1}{2}\mu_{xx}$ for all $x\in V$. This simply
follows from the standard minmax characterization of the
eigenvalues
\[\xi_1(\Omega)=\inf_{f\neq 0}\frac{(Q_\Omega f,f)_\mu}{(f,f)_\mu},\] and the following formula for any
$f,g\in\ell^2(\Omega,\mu)$
\begin{eqnarray*}
(Q_\Omega f,g)_\mu&=& \sum_{x\in\Omega}\mu(x)Q_\Omega f(x)g(x)=
\sum_{x\in V}\sum_{y\in V}\mu_{xy}
(\tilde{f}(x)+\tilde{f}(y))\tilde{g}(x)\\&=& \sum_{y\in
V}\sum_{x\in V}\mu_{yx} (\tilde{f}(y)+\tilde{f}(x))\tilde{g}(y)
\end{eqnarray*} where we just exchanged $x$ and $y$. Adding the last two
lines and setting $\tilde{f}=\tilde{g}$ yields\[(Q_\Omega
f,f)_\mu=
\frac{1}{2}\sum_{x,y}\mu_{xy}(\tilde{f}(x)+\tilde{f}(y))^2.\]

In order to prove the lower bound for $\xi_1(\Omega)$ we will use
a technique developed in \cite{Desai94}. The idea is the
following: Construct a graph $\Gamma^\prime$ out of $\Gamma$ s.t.
the quantity $h^\prime(\Omega,g)$ defined  in (\ref{hg}) for the
new graph $\Gamma^\prime$ controls $\xi_1(\Omega)$ from below. In
a second step, we show that $h^\prime(\Omega, g)$ in turn can be
controlled by the quantity $1-\overline{h}(\Omega)$ of the
original graph. This then yields the desired estimate.

Let $f$ be an eigenfunction for the eigenvalue $\xi_1(\Omega)$ of
$Q_{\Omega}$ and define as above $P(f)= \{x\in \Omega : f(x)>0\}$
and $N(f)= \{x\in \Omega : f(x) <0\}$. Then the new graph
$\Gamma^\prime =(V^\prime, E^\prime)$ is constructed from
$\Gamma=(V,E)$ in the following way. Duplicate all vertices in
$P(f)\cup N(f)$ and denote the copies by a prime, e.g. if $x\in
P(f)$ then the copy of $x$ is denoted by $x^\prime$. The copies of
$P(f)$ and $N(f)$ are denoted by $P^\prime(f)$ and $N^\prime(f)$
respectively. The vertex set $V^\prime $ of $\Gamma^\prime$ is
given by $V^\prime= V\cup P^\prime(f)\cup N^\prime(f)$. Every edge
$(x,y)\in E(P(f),P(f))$ in $\Gamma$ is replaced by two edges
$(x,y^\prime)$ and $(y,x^\prime)$ in $\Gamma^\prime$ s.t.
$\mu_{xy} = \mu^\prime_{xy^\prime}=\mu^\prime_{yx^\prime}$.
Similarly, if the edge is a loop, then $e=(x,x)$ is replaced by
one edge $(x,x^\prime)$ s.t. $\mu_{xx}=\mu_{xx^\prime}$. The same
is done with edges in $E(N(f),N(f))$. All other edges remain
unchanged, i.e. if $(k,l)\in E\setminus (E(P(f),P(f))\cup
E(N(f),N(f)))$ then $(k,l)\in E^\prime$ and $\mu_{kl} =
\mu^\prime_{kl}$. It is important to note that this construction
does not change the degrees of the vertices in
$V=V^\prime\setminus (P^\prime(f)\cup N^\prime(f))$.

Consider the function $\tilde{g}:V^\prime \rightarrow \R$,
\[\tilde{g}(x)=\left\{\begin{array}{ccc} |f(x)| & \mbox{ if}&  x \in P(f)\cup N(f) \\0 &&
else. \end{array}\right.\] It can easily be checked that by
construction of $\Gamma^\prime$ we have
\begin{eqnarray*}\xi_1(\Omega) &=& \frac{\sum_{e=(x,y)\in E}\theta_{xy}(\tilde{f}(x) +
\tilde{f}(y))^2}{\sum_{x\in V}\mu(x)\tilde{f}(x)^2}\\&\geq&
\frac{\sum_{e^\prime=(x,y)\in
E^\prime}\mu^\prime_{xy}(\tilde{g}(x) -
\tilde{g}(y))^2}{\sum_{x\in
V^\prime}\mu^\prime(x)\tilde{g}(x)^2}\\&\geq& 1-
\sqrt{1-(h^\prime(\Omega, g))^2}
\end{eqnarray*} where we used  Lemma \ref{Lemma1} to obtain the last
inequality. For any non-empty subset $W\subseteq P(g) = P(f)\cup
N(f)$ we define $P(W)=W\cap P(f)$ and $N(W)=W\cap N(f)$. Let
$\emptyset\neq U\subseteq P(g)$ the subset that realizes the
infimum, i.e.
\begin{eqnarray*}h^\prime(\Omega, g) &=& \inf_{\emptyset\neq W\subseteq P(g)}
\frac{|E^\prime(W,W^c)|}{\mathrm{vol}(W)}=
\frac{|E^\prime(U,U^c)|}{\mathrm{vol}(U)}\\&=&
\frac{|E(P(U),P(U))|+ |E(N(U),N(U))| + |\partial(P(U)\cup N(U))|}
{\mathrm{vol}(P(U))+{\mathrm{vol}(N(U))}}\\&=& 1- \frac{2
|E(P(U),N(U))|}{\mathrm{vol}(P(U))+{\mathrm{vol}(N(U))}}\\&\geq&
1- \sup_{V_1,V_2\subset\Omega}\frac{2
|E(V_1,V_2)|}{\mathrm{vol}(V_1)+{\mathrm{vol}(V_2)}}
\\&=&1-\overline{h}(\Omega),
\end{eqnarray*}where we used that $P(U),N(U)\subset \Omega$.
 Thus we have
\[2-\lambda_\mathrm{max}(\Omega)=\xi_1(\Omega) \geq 1 - \sqrt{1-(1-\overline{h}(\Omega))^2}\] and so
\[\lambda_\mathrm{max}(\Omega)\leq 1+\sqrt{1-(1-\overline{h}(\Omega))^2}.\]
\end{proof}
\begin{rem}
\begin{itemize}\item[$(i)$]
Note that if $\Omega$ is bipartite, then the upper bound for
$\lambda_\mathrm{max}(\Omega)$ follows directly from Lemma
\ref{BasicProperties} $(iv)$, the Cheeger inequality, and Theorem
\ref{h and hbar}. \item[$(ii)$] For finite graphs it is also
possible to obtain a Cheeger and a dual Cheeger estimate for the
smallest nonzero and the largest eigenvalue, respectively (see for
instance \cite{Chung97} and \cite{Bauer12}). For finite graphs the
Cheeger together with the dual Cheeger estimate are very powerful
since many graph properties such as the diameter, the independence
number or the convergence of a random walk to its stationary
distribution are controlled by the maximal difference from the
smallest nonzero and the largest eigenvalue from $1$, see
\cite{Chung97, Lubotzky94}.
%Moreover, for a $k$-regular finite
%graphs with $N$ vertices a positive Cheeger constant implies that
%the graph is an $(N,k, \frac{h}{k})$-expander (see
%\cite{Lubotzky94} for more details) which is important in many
%applications.

\end{itemize}
\end{rem}
\section{Eigenvalue comparison theorems}\label{Section5}
In this section we prove some eigenvalue comparison theorems for
the largest eigenvalue of the Dirichlet Laplace operator that do
not have analogues in Riemannian geometry. Similar eigenvalue
comparison theorems for the smallest eigenvalue were obtained by
Urakawa \cite{Urakawa99}. Urakawa's results are discrete versions
of Cheng's eigenvalue comparison theorem for Riemannian manifolds
\cite{Cheng75a, Cheng75b}. In the literature (see e.g.
\cite{Urakawa99,Friedman93}), the mostly used comparison models
for graphs are homogeneous trees, since regular trees are
considered to be discrete analogues of simply connected space
forms. Here we propose a novel comparison model for graphs, the
weighted half-line  $R_l\ (l\in \mathds{R},\ l\geq 2)$ which is
defined by $R_l:=(V,E),$ $V=\mathds{N}\cup\{0\}, E=\{(i,i+1),i\geq
0, i\in V\}$ and the edge weights are defined as
$\mu_{(i,i+1)}=(l-1)^i,$ see Figure \ref{Fig.5}. More precisely,
we want to compare the largest Dirichlet eigenvalue of a ball in
any graph with the largest Dirichlet eigenvalue of a ball with the
same radius in $R_l$ (in contrast to the existing literature,
where homogenous trees $T_d$ were used instead). The advantage of
the weighted half-line $R_l$ is that we get better estimates in
the comparison theorems since we can overcome the restriction that
$d$ must be an integer if we use the homogenous tree $T_d$ as a
comparison model. In particular, the results using $R_l$ or $T_d$
as a comparison model coincide if $l=d$ are integers.

\begin{figure}\begin{center}
\includegraphics[width =
 10cm]{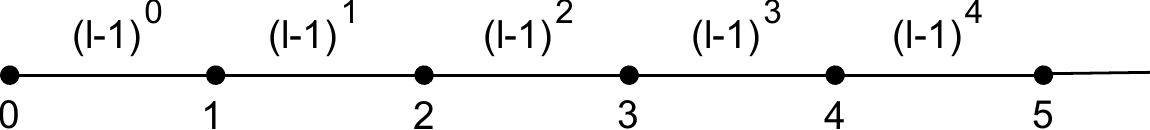}\caption{\label{Fig.5}The weighted half-line $R_l$}
\end{center}
\end{figure}
The next lemma shows that the eigenvalues and eigenfunctions of
the Dirichlet operator of $R_l$ behave like those on infinite
$d$-homogenous trees $T_d$, since a similar result was obtained
for $T_d$ by Friedmann \cite{Friedman93} and Urakawa
\cite{Urakawa99}.

\begin{lemma}\label{Dirichlet Eigenvalues of Rl} Let
$V_l(r)=\{x\in R_l: d(x,0)\leq r\}$ and let $l\geq 2$. The first
eigenvalue of the Dirichlet Laplace operator on $V_l(r)$ in $R_l$
is
$$\lambda_1(V_l(r))=1-\frac{2\sqrt{l-1}}{l}\cos\theta_{l,r},$$ where
$\theta_{l,r}$ is the smallest positive solution $\theta$ of
$g_{r}(\theta)=\frac{l}{2(l-1)},$ where
$$g_r(\theta)=\frac{\sin((r+1)\theta)\cos\theta}{\sin(r\theta)}.$$ We have $\theta_{2,r}=\frac{\pi}{2(r+1)}$ and for $l>2$ we have $$\max\{\frac{\pi}{r+\eta},\frac{\pi}{2(r+1)}\}\leq\theta_{l,r}\leq\frac{\pi}{r+1},$$ where $\eta=\eta(l)=\frac{2(l-1)}{l-2}.$ The first eigenfunction $f_1$ is
$$f_1(x)=(l-1)^{-\frac{r(x)}{2}}\sin(r+1-r(x))\theta_{l,r}$$ which is positive and decreasing (as a function of one variable) on $V_l(r)$ where $r(x)=d(x,0).$
\end{lemma}
\begin{proof} We omit the proof of this lemma since it is
straightforward and similar to the proof in the case of
$d$-regular trees $T_d$, see \cite{Urakawa99} (Lemma $3.1$) and
\cite{Friedman93} (Proposition 3.2).
\end{proof}

The following corollary is a straightforward consequence of the
previous lemma by the bipartiteness of $R_l,$ see Theorem
\ref{Lembi}.
\begin{lemma}\label{Estimates Eigenvalues Rl}
For $l>2,$
\begin{eqnarray}\label{estim3}
1-\frac{2\sqrt{l-1}}{l}\cos\frac{\pi}{r+\eta}&\leq& \lambda_1(V_l(r))\ \ \ \leq 1-\frac{2\sqrt{l-1}}{l}\cos\frac{\pi}{r+1},\\
1+\frac{2\sqrt{l-1}}{l}\cos\frac{\pi}{r+1}&\leq&
\lambda_{\mathrm{max}}(V_l(r))\leq
1+\frac{2\sqrt{l-1}}{l}\cos\frac{\pi}{r+\eta} ,
\end{eqnarray} where $\eta=\frac{2(l-1)}{l-2}$, and for $l=2$
$$\lambda_1(V_l(r))= 1 - \cos\frac{\pi}{2(r +1)} \text{ and }
\lambda_{\mathrm{max}}(V_l(r))= 1 + \cos\frac{\pi}{2(r +1)}.$$
\end{lemma}

Now we study the eigenfunction $f_\mathrm{max}$ for the largest
Dirichlet eigenvalue of a ball in $R_l$.
\begin{lemma}\label{Eigenfunction tree} The eigenfunction
$f_\mathrm{max}$ for the largest Dirichlet eigenvalue
$\lambda_\mathrm{max}(V_l(r))$ satisfies:
\begin{enumerate}
\item[$(i)$] $|f_\mathrm{max}(s)|$ is decreasing in $s$, i.e.
$|f_\mathrm{max}(s)|> |f_\mathrm{max}(s+1)|$  $\forall$ $0\leq
s=r(x)\leq r.$ \item[$(ii)$]
$f_\mathrm{max}(s)f_\mathrm{max}(s+1)<0$ for all $0\leq s< r-1.$
\end{enumerate}
\end{lemma}
\begin{proof} Since the weighted half line $R_l$ is
bipartite, we know from the proof of Lemma \ref{symmetric
eigenvalues} that the eigenfunction for the largest eigenvalue
$f_\mathrm{max}$ is given by
\begin{equation}\label{3} f_\mathrm{max}(x) = \left\{\begin{array}{cc}f_1(x) & \text{ if } x\in U\\
-f_1(x) & \text{ if } x\in \overline{U}
\end{array}\right.
\end{equation} where $U\cup\overline{U}=V_l(r)$ is a bipartition of $V_l(r)$
and $f_1$ is the first eigenfunction. Since by Lemma
\ref{Dirichlet Eigenvalues of Rl} the first eigenfunction $f_1$ is
positive and decreasing both statements follow immediately.
\end{proof}

In order to prove the eigenvalue comparison theorem we also need
the following Barta-type theorem for the largest Dirichlet
eigenvalue.
\begin{lemma}[Barta-type theorem for $\lambda_\mathrm{max}(\Omega)$]\label{Barta}
Let $f\in\ell^2(\Omega,\mu)$ be any function s.t. $f(x)\neq 0$ for
all $x\in \Omega$. Then we have $$\lambda_\mathrm{max}(\Omega)\geq
\inf_{x\in \Omega}\frac{\Delta_\Omega f(x)}{f(x)}.$$
\end{lemma}
\begin{proof}For all $f\in \ell^2(\Omega,\mu)$ with $f(x)\neq 0$ for all $x\in \Omega$ we
have
\begin{eqnarray*}
(\Delta_\Omega f, f)_\mu &=& \sum_{x\in \Omega} \mu(x)
\Delta_\Omega f(x) f(x)\\ &=& \sum_{x\in \Omega} \mu(x)
\frac{\Delta_\Omega f(x)}{f(x)} f^2(x)\\&\geq&  \sum_{x\in \Omega}
\mu(x) \inf_{x\in \Omega}\frac{\Delta_\Omega f(x)}{f(x)}
f^2(x)\\&=& \inf_{x\in \Omega}\frac{\Delta_\Omega f(x)}{f(x)}
(f,f)_\mu.
\end{eqnarray*}Now
the lemma follows since for all such functions $f$
$$\lambda_\mathrm{max}(\Omega) = \sup_{g\in \ell^2(\Omega,\mu)}
\frac{(\Delta_\Omega g,g)_\mu}{(g,g)_\mu}\geq \frac{(\Delta_\Omega
f,f)_\mu}{(f,f)_\mu}\geq \inf_{x\in \Omega}\frac{\Delta_\Omega
f(x)}{f(x)}.$$
\end{proof}

We introduce the following notation:
\begin{defi}For a graph $\Gamma=(V,E)$ we  define
\begin{eqnarray*}V_+(x) &=& \{y\in V: y\sim x, d(x_0,y) = d(x_0,x) +1\}
\\V_-(x) &=& \{y\in V: y\sim x, d(x_0,y) = d(x_0,x) -1\}\\
V_0(x) &=& \{y\in V: y\sim x, d(x_0,y) = d(x_0,x)\}
\end{eqnarray*} for some
fixed $x_0\in V$. Moreover, we define \begin{eqnarray*} \mu_+(x)
:= \sum_{y\in V_+(x)}\mu_{xy},\quad \mu_-(x) := \sum_{y\in
V_-(x)}\mu_{xy}\quad \text{ and }\ \mu_0(x) := \sum_{y\in
V_0(x)}\mu_{xy}.
\end{eqnarray*}
\end{defi}
Clearly, $\mu(x) = \mu_+(x) + \mu_-(x)+ \mu_0(x)$.\\
For the sequel, let us observe that for the weighted half-line
  $R_l$ and reference point $x_0 =0$, at every vertex, $\frac{\mu(x)}{\mu_-(x)}=l$ and $\mu_0(x)=0$.\\

\begin{theo}[Eigenvalue comparison theorem I]\label{ecpt1} Let
$B(x_0,r)$ be a ball in $\Gamma$ centered at $x_0$ with radius $r$
and $2 \leq l:= \sup_{x\in B(x_0,r)}\frac{\mu(x)}{\mu_-(x)}<
\infty$. Then the largest Dirichlet eigenvalue
$\lambda_\mathrm{max}(B(x_0,r))$ satisfies
$$\lambda_\mathrm{max}(B(x_0,r))\geq \lambda_\mathrm{max}(V_l(r))- 2\kappa(x_0,r),$$
where $V_l(r)$ is a ball of radius $r$ centered at $0$ in the
weighted half-line $R_l$ and $\kappa(x_0,r):=\sup_{x\in
B(x_0,r)}\frac{\mu_0(x)}{\mu(x)}$.
\end{theo}
Note that if $x_0$ is not contained in a circuit of odd length
$\leq 2r+1$ then $\kappa(x_0,r)=0$. In particular, if the ball
$B(x_0,r)$, considered as an induced subgraph in $\Gamma$, is
bipartite, then $\kappa(x_0,r)=0$. In particular we have the
following corollary:
\begin{coro}Let $B(x_0,r)$ be a ball on $\Gamma$ that does not contain any
cycle of odd length, then under the same assumptions as in the
last theorem we have
$$\lambda_\mathrm{max}(B(x_0,r))\geq \lambda_\mathrm{max}(V_l(r)).$$
\end{coro}

\begin{rem}
The choice of $l$ in the last theorem yields the best possible
results. However, all we need is that $l\geq
\frac{\mu(x)}{\mu_-(x)}$ for all $x\in B(x_0,r)$. Thus it is
sufficient to choose $l\geq
\frac{\sup_x\mu(x)}{\inf_{x,y}\mu_{xy}}$ in Theorem \ref{ecpt1}.
This also shows why we restrict ourselves to $R_l$ for $l\geq2$.
If $l<2$ and $l\geq \frac{\sup_{x\in V}\mu(x)}{\inf_{x,y\in
V}\mu_{xy}}$ then this immediately implies that the graph is
finite - in fact the graph consists of two vertices connected by
one edge, provided the graph is connected.
\end{rem}
Before we prove this theorem, we show that the quantity
$\kappa(x_0,r)$ is related to a local clustering coefficient
through the following quantity
\begin{equation}\label{definition local clustering} C(x_0,
r):= \sup_{x\in B(x_0,r)}\frac{\sum_{y\in
\mathcal{C}(x,r)}\mu_{xy}}{\mu(x)}\end{equation} where
$\mathcal{C}(x,r):= \{y\in V: y\sim x$ $x$ and $y$ are both
contained in a circuit of odd length $\leq 2r+1\}$ and a circuit
is a cycle without repeated vertices. Recall that a graph is
bipartite if and only if there are no cycles of odd length in the
graph. Hence $C(x_0,r)$ expresses how different a graph is from a
bipartite graph. Moreover, $C(x_0,r)$ is related to a local
clustering coefficient since $C(x_0,r)=0$ implies that all $x\in
B(x_0,r)$ are not contained in any triangle.  Estimates for the
largest eigenvalue of the Laplace operator for finite graphs in
terms of a local clustering coefficient were obtained in
\cite{Bauer12}.
\begin{lemma}\label{kappa and local clustering} We have for all $x_0\in V$ and $r\geq 0,$ \begin{equation}\label{clustering}\kappa(x_0,r)\leq C
(x_0,r)\end{equation}
\end{lemma}
\begin{proof}It suffices to show that if $z\in V_0(x)$ for
some fixed $x_0$, then $z\in \mathcal{C}(x,r)$, i.e.
$V_0(x)\subset \mathcal{C}(x,r)$. We assume that for a fixed
vertex $x_0\in V$ we have $d(x_0,x)= d(x_0,z)= N\leq r$ and $z\sim
x$. Let $P_x$ and $P_z$ denote a shortest path from $x_0$ to $x$
and $z$, respectively. We label the vertices in the path $P_x$ by
$x_0, x_1, \ldots, x_N = x$ and the vertices in $P_z$ by $x_0=z_0,
z_1, \ldots, z_N = z$. Let $0\leq k\leq N-1$ be such that $x_l\neq
z_l$ for all $l>k$, i.e. $x_k= z_k$ is the last branching point of
the two paths $P_x$ and $P_z$. We claim that $x_k,x_{k+1}, \ldots,
x_N=x,z=z_N, z_{N-1}, \ldots, z_k = x_k$ is a circuit of odd
length. Clearly, this is a cycle of length $2(N-k)+1$, i.e. a
cycle of odd length. Moreover, by construction the vertices in
this cycle are all different from each other. Thus $z$ and $x$ are
contained in a circuit of odd length $\leq 2(N-k)+1\leq 2r+1$ and
hence $z\in \mathcal{C}(x,r)$. This completes the proof.
\end{proof}
Using techniques developed by Urakawa \cite{Urakawa99} we now give
a proof of Theorem \ref{ecpt1}.
\begin{proof}[Proof of Theorem \ref{ecpt1}] Let $B(x_0,r)$ be a ball centered
at $x_0$ in $\Gamma$ and $V_l(r)$ be a ball in $R_l$ centered at
$0$ with the same radius $r$. Using the eigenfunction
$f_\mathrm{max}$ on $V_l(r)$ we define a function (denoted by the
same letter) on the ball $B(x_0,r)$ in $\Gamma$ by
$$f_\mathrm{max}(x) = f_\mathrm{max} (r(x)) \text{ \qquad for all $x\in B(x_0,r)$},$$ where $r(x)=d(x,x_0).$
Now let $\Delta$ and $\Delta^l$ be the Laplace operators on
$B(x_0,r)$ and $V_l(r)$ respectively. For all $x\in B(x_0,r)$ and
$z\in V_l(r)$ with $0<s = r(x) = r_l(z) \leq r$ where
$r_l(z)=d(z,0)=|z|$ we obtain
\begin{eqnarray*}\Delta f_\mathrm{max}(x)&=& f_\mathrm{max}(x) -
\frac{\mu_0(x)f_\mathrm{max}(x) + \mu_+(x) f_\mathrm{max}(x+1)+
\mu_-(x)f_\mathrm{max}(x-1)}{\mu(x)}
\\&=& \frac{\mu_-(x)}{\mu(x)}(f_\mathrm{max}(x) -
f_\mathrm{max}(x-1))+ \frac{\mu_+(x)}{\mu(x)}(f_\mathrm{max}(x) -
f_\mathrm{max}(x+1))
\end{eqnarray*}
and
\begin{eqnarray*}
\Delta^l f_\mathrm{max}(z) = \frac{1}{l}(f_\mathrm{max}(x) -
f_\mathrm{max}(x-1)) + \frac{l-1}{l}(f_\mathrm{max}(x) -
f_\mathrm{max}(x+1)).
\end{eqnarray*}
Putting these two equations together we obtain:

\begin{eqnarray}\label{evc1}\frac{\Delta^l f_\mathrm{max}(z)}{f_\mathrm{max}(z)}- \frac{\Delta
f_\mathrm{max}(x)}{f_\mathrm{max}(x)} &=& \left(\frac{1}{l} -
\frac{\mu_-(x)}{\mu(x)}\right)\frac{f_\mathrm{max}(s) - f_\mathrm{max}(s-1)}{f_\mathrm{max}(s)}\nonumber\\
&+& \left(\frac{l-1}{l} -
\frac{\mu_+(x)}{\mu(x)}\right)\frac{f_\mathrm{max}(s) -
f_\mathrm{max}(s+1)}{f_\mathrm{max}(s)}.
\end{eqnarray}

 Moreover, in the case $r(x) = r_l(z)=0$ we obtain
\begin{equation}\label{evc2}\frac{\Delta^lf_\mathrm{max}(0)}{f_\mathrm{max}(0)} - \frac{\Delta
f_\mathrm{max}(x_0)}{f_\mathrm{max}(x_0)} = 0.\end{equation} It is
important to note that by Lemma \ref{Eigenfunction tree}
$$1<\frac{f_\mathrm{max}(s) - f_\mathrm{max}(s+1)}{f_\mathrm{max}(s)}\leq \frac{f_\mathrm{max}(s) -
f_\mathrm{max}(s-1)}{f_\mathrm{max}(s)} \text{ for all }0\leq
s\leq r$$ and by our assumption it holds that $\left(\frac{1}{l} -
\frac{\mu_-(x)}{\mu(x)}\right)\leq 0$ for all $x\in B(x_0,r)$.
Hence we obtain
\begin{eqnarray*}\frac{\Delta^l f_\mathrm{max}(z)}{f_\mathrm{max}(z)}- \frac{\Delta f_\mathrm{max}(x)}{f_\mathrm{max}(x)}
&\leq& \left(\frac{1}{l} - \frac{\mu_-(x)}{\mu(x)} + \frac{l-1}{l}
- \frac{\mu_+(x)}{\mu(x)}\right)\frac{f_\mathrm{max}(s) -
f_\mathrm{max}(s+1)}{f_\mathrm{max}(s)}
\\&=& \frac{\mu_0(x)}{\mu(x)}\left(1 -
\frac{f_\mathrm{max}(s+1)}{f_\mathrm{max}(s)}\right)\\&\leq& 2
\frac{\mu_0(x)}{\mu(x)},
\end{eqnarray*}where used again Lemma \ref{Eigenfunction tree}.
Using the Barta-type estimate in Lemma \ref{Barta} it follows
\begin{eqnarray*} \lambda_\mathrm{max}(B(x_0,r))&\geq& \inf_{x\in B(x_0,r)}\frac{\Delta
f_\mathrm{max}(x)}{f_\mathrm{max}(x)}\\&\geq& \inf_{z\in
V_l(r)}\frac{\Delta^l f_\mathrm{max}(z)}{f_\mathrm{max}(z)} -
2\sup_{x\in B(x_0,r)}\frac{\mu_0(x)}{\mu(x)}\\&=&
\lambda_\mathrm{max}(V_l(r)) - 2\kappa(x_0,r).
\end{eqnarray*}
\end{proof}

The largest Dirichlet eigenvalue of a ball of radius $r$ in the
weighted half-line $R_l$ is precisely known, see Lemma
\ref{Dirichlet Eigenvalues of Rl}. In particular, by combining
Theorem \ref{ecpt1} with Lemma \ref{Estimates Eigenvalues Rl}  we
obtain the following corollary:
\begin{coro} \label{Estimates Balls 2}Let $B(x_0,r)$ be a
ball in $\Gamma$ and $2\leq l= \sup_{x\in
B(x_0,r)}\frac{\mu(x)}{\mu_-(x)}<\infty.$  Then the largest
Dirichlet eigenvalue $\lambda_\mathrm{max}(B(x_0,r))$ satisfies
for $l>2$
\begin{eqnarray*}\lambda_\mathrm{max}(B(x_0,r))&\geq& 1 +
\frac{2\sqrt{l-1}}{l}\cos\left(\frac{\pi}{r+1}\right)-
2\kappa(x_0,r),\end{eqnarray*} and for $l=2$
\begin{eqnarray*}\lambda_\mathrm{max}(B(x_0,r))&\geq& 1 +
\cos\left(\frac{\pi}{2(r+1)}\right)-
2\kappa(x_0,r).\end{eqnarray*}
\end{coro}

By a similar argument as in Theorem \ref{ecpt1} we obtain a lower
bound estimate of $\lambda_1$ which implies the upper bound
estimate for $\lambda_{\mathrm{max}}$. The following theorem
generalizes $(ii)$ of Theorem $3.3$ in \cite{Urakawa99}.

\begin{theo}[Eigenvalue comparison theorem II]\label{Theob1}
Let $B(x_0,r)$ be a ball in $\Gamma$. Suppose $2\leq l:=\inf_{x\in
B(x_0,r)}\frac{\mu(x)}{\mu_-(x)}<\infty,$ then
\begin{equation}\label{estim0}
\lambda_1(B(x_0,r))\geq\lambda_1(V_l(r))-\kappa(x_0,r),
\end{equation}
\begin{equation}
\lambda_{\mathrm{max}}(B(x_0,r))\leq\lambda_{\mathrm{max}}(V_l(r))+\kappa(x_0,r),
\end{equation} where $\kappa(x_0,r):=\sup_{x\in
B(x_0,r)}\frac{\mu_0(x)}{\mu(x)}.$
\end{theo}
\begin{proof}
By $(iv)$ of Lemma \ref{BasicProperties}, the bipartiteness of the
weighted half-line $R_l$ and Theorem \ref{Lembi}, it suffices to
prove \eqref{estim0}. Let $f_1$ be the first eigenfunction of the
Dirichlet Laplace operator on $V_l(r)$ in $R_l.$ By Lemma
\ref{Dirichlet Eigenvalues of Rl}, $f_1$ is a positive, decreasing
function on $V_l(r).$ We define a function on $B(x_0,r)$ in
$\Gamma$ as $f_1(x)=f_1(r(x))$ where $r(x)=d(x,x_0).$ The same
calculation as \eqref{evc1} and \eqref{evc2} yields for any $x\in
B(x_0,r)$ and $z\in V_l(r)$ with $0<s = r(x) = r_l(z) \leq r$
where $r_l(z)=d(z,0)$

\begin{eqnarray*}\frac{\Delta^l f_1(z)}{f_{1}(z)}- \frac{\Delta
f_1(x)}{f_{1}(x)} &=& \left(\frac{1}{l} -
\frac{\mu_-(x)}{\mu(x)}\right)\frac{f_1(s) - f_1(s-1)}{f_1(s)}\nonumber\\
&+& \left(\frac{l-1}{l} -
\frac{\mu_+(x)}{\mu(x)}\right)\frac{f_1(s) - f_1(s+1)}{f_1(s)},
\end{eqnarray*} and

\begin{equation*}\frac{\Delta^lf_1(0)}{f_1(0)} - \frac{\Delta
f_1(x_0)}{f_1(x_0)} = 0.\end{equation*} By our assumption it holds
that $\left(\frac{1}{l} - \frac{\mu_-(x)}{\mu(x)}\right)\geq 0$
for all $x\in B(x_0,r)$. Note that $f_1(s)$ is a decreasing
function of $s$ which implies that $\frac{f_1(s) -
f_1(s-1)}{f_1(s)}\leq 0$ Hence
\begin{eqnarray*}\frac{\Delta^l f_1(z)}{f_1(z)}- \frac{\Delta f_1(x)}{f_1(x)}
&\leq& \left(\frac{l-1}{l} -
\frac{\mu_+(x)}{\mu(x)}\right)\frac{f_1(s) - f_1(s+1)}{f_1(s)}
\\&=& \left(\frac{\mu_0(x)}{\mu(x)}+\frac{\mu_{-}(x)}{\mu(x)}-\frac{1}{l}\right)\left(1 -
\frac{f_1(s+1)}{f_1(s)}\right)\\&\leq& \frac{\mu_0(x)}{\mu(x)}\leq
\kappa(x_0,r).
\end{eqnarray*}

Then by the Barta's theorem for the first eigenvalue, see Theorem
$2.1$ in \cite{Urakawa99},
\begin{eqnarray*}\lambda_1(B(x_0,r))&\geq& \inf_{x\in B(x_0,r)}\frac{\Delta f_1}{f_1}\geq \inf_{z\in V_l(r)}\frac{\Delta^l f_1}{f_1}-\kappa(x_0,r)\\
&=&\lambda_1(V_l(r))-\kappa(x_0,r).
\end{eqnarray*}
\end{proof}

By Lemma \ref{Estimates Eigenvalues Rl} we obtain the following
corollary which is the counterpart to Corollary \ref{Estimates
Balls 2} (we omit the easier case of $l=2$).
\begin{coro}\label{Lowercomp1}
Let $B(x_0,r)$ be a ball in $\Gamma$ and $2< l=\inf_{x\in
B(x_0,r)}\frac{\mu(x)}{\mu_-(x)}<\infty.$ Then
\begin{eqnarray}\label{estim1}
\lambda_1(B(x_0,r))&\geq&
1-\frac{2\sqrt{l-1}}{l}\cos\frac{\pi}{r+\eta}-\kappa(x_0,r),\nonumber
\end{eqnarray}
\begin{eqnarray*}
\lambda_{\mathrm{max}}(B(x_0,r))&\leq&
1+\frac{2\sqrt{l-1}}{l}\cos\frac{\pi}{r+\eta}+\kappa(x_0,r)
\end{eqnarray*} where as before $\eta=\frac{2(l-1)}{l-2},$ and $\kappa(x_0,r)=\sup_{x\in
B(x_0,r)}\frac{\mu_0(x)}{\mu(x)}$.
\end{coro}

\section{Eigenvalue estimates for finite graphs}\label{Section6}
In this section, we show how the eigenvalue comparison theorem
obtained in the last section can be used to estimate the largest
eigenvalues of finite graphs.
\begin{defi}We say two balls $B_1$ and $B_2$ are edge disjoint if
$|E(B_1,B_2)|=0$.
\end{defi}
\begin{lemma}\label{Estimate Balls}
Let $G=(V,E)$ be a finite graph with $\sharp V=N$ and let $B_1,
\ldots B_m$ be $m\geq 1$ edge disjoint balls in $\Gamma$. Then
$$\theta_{N-m} \geq \min_{i = 1,\ldots m}\lambda_\mathrm{max}(B_i),$$ where
$\theta_i$, $i=0,\ldots,N-1$ are the eigenvalues of the Laplace
operator of the finite graph and $\lambda_i$ are the eigenvalues
of the Laplace operator with Dirichlet boundary conditions.
\end{lemma}
\begin{proof}We use the same technique as Quenell \cite{Quenell96} who
proved a similar result for the smallest eigenvalue of a finite
graph.By the variational characterization of the eigenvalues we
have
$$\theta_{N-1} = \sup_{g\neq 0}\frac{(\Delta g,g)_G}{(g,g)_G},$$
and
\begin{equation}\label{variational principle}\theta_{N-m} = \sup_{g\perp g_{N-1}, \ldots
g_{N-m+1}}\frac{(\Delta g,g)_G}{(g,g)_G},\end{equation} where
$(f,g)_G= \sum_{x\in V}\mu(x) f(x)g(x)$ is the scalar product on
the finite graph $G$ and $g_i$ is an eigenfunction for the
eigenvalue $\theta_i$.
 From now
on we use the following notation: $\Delta_i = \Delta_{B_i}$
denotes the Dirichlet Laplace operator on the ball $B_i$.

Let $f_i: B_i\to \mathbb{R} $ be the eigenfunction corresponding
to the largest eigenvalue $\lambda_\mathrm{max}(B_i)$ of the
Dirichlet Laplace operator $\Delta_i$. We extend this to a
function $\tilde{f}_i: V \to \mathbb{R}$ on the whole of $V$ by
setting
$$\tilde{f}_i(x) = \left\{\begin{array}{cc}f_i(x) & \text{ if } x\in B_i\\
0 & \text{ if } x\notin B_i.
\end{array}\right.$$
Because the balls $B_1,\ldots B_m$ are disjoint the functions
$\tilde{f}_i,\tilde{f}_j$ satisfy $\mathrm{supp}\tilde{f}_i\cap
\mathrm{supp}\tilde{f}_j = \emptyset$ and hence $(\tilde{f}_i,
\tilde{f}_j)_G=0$ for all $i\neq j$. From this it follows that the
functions $\tilde{f}_1,\ldots \tilde{f}_{m}$ span a
$m$-dimensional subspace in $\ell^2(G,\mu)$. Hence there exists
coefficients $a_i$ such that the function
$$\tilde{f} := \sum_{i=1}^ma_i\tilde{f}_i$$ is orthonogal to the $m-1$
dimensional subspace spanned by the $g_i$, i.e. $(\tilde{f},
g_i)=0$ for all $i = N-1, \ldots, N-m+1$. From the variational
principle (\ref{variational principle}) we immediately obtain
\begin{equation}\label{4} \theta_{N-m}\geq \frac{(\Delta\tilde{f},
\tilde{f})_G}{(\tilde{f},\tilde{f})_G}.\end{equation}

We wish to show that
\begin{equation} \label{5}\frac{(\Delta\tilde{f}_i,
\tilde{f}_i)_G}{(\tilde{f}_i,\tilde{f}_i)_G}=\frac{(\Delta_i f_i,
f_i)_{B_i}}{(f_i,f_i)_{B_i}}.\end{equation} This can be seen from
the following straightforward calculation:
\begin{eqnarray*}
(\Delta\tilde{f}_i, \tilde{f}_i)_G &=& \sum_{x\in
V}\mu(x)\Delta\tilde{f}_i(x)\tilde{f}_i(x)\\&=& \sum_{x\in
V}\mu(x)[(\tilde{f}_i(x) - \frac{1}{\mu(x)}\sum_{y\in
V}\mu_{xy}\tilde{f}_i(y))\tilde{f}_i(x)]
\\&=&\sum_{x\in B_i}\mu(x)[(f_i(x) -
\frac{1}{\mu(x)}\sum_{y\in B_i}\mu_{xy}f_i(y))f_i(x)]\\&=&
\sum_{x\in B_i}\mu(x)\Delta_if_i(x)f_i(x)\\&=& (\Delta_if_i,
f_i)_{B_i}.
\end{eqnarray*} Together with  $(\tilde{f}_i,\tilde{f}_i)_G =
\sum_{x\in V} \mu(x)\tilde{f}_i(x)^2 = \sum_{x\in B_i}
\mu(x)f_i(x)^2 = (f_i,f_i)_{B_i}$ we conclude
$$\frac{(\Delta\tilde{f}_i,
\tilde{f}_i)_G}{(\tilde{f}_i,\tilde{f}_i)_G}=\frac{(\Delta_i f_i,
f_i)_{B_i}}{(f_i,f_i)_{B_i}} = \lambda_\mathrm{max}(B_i).$$

Next we show that
\begin{equation}\label{6}(\Delta\tilde{f}_i,
\tilde{f}_j)_G  = \sum_{x\in V}\mu(x)[\tilde{f}_i(x) -
\frac{1}{\mu(x)}\sum_{y\in V}
\mu_{xy}\tilde{f}_i(y)]\tilde{f}_j(x)= 0\text{ if $i\neq
j$.}\end{equation} We have a look at each term
\begin{eqnarray} \label{2}[\tilde{f}_i(x) -
\frac{1}{\mu(x)}\sum_{y\in V}
\mu_{xy}\tilde{f}_i(y)]\tilde{f}_j(x)
\end{eqnarray}in the sum
separately: We distinguish the following two cases: $(i)$ If
$x\notin B_j$ equation (\ref{2}) is obviously equal to zero.
$(ii)$ If $x\in B_j$ it follows from the edge disjointness of the
balls that $x \notin B_i$ and that all neighbors $y \sim x$ are
not contained in $B_i$. Hence also in the latter case equation
(\ref{2}) is equal to zero. Hence if $i\neq j$, then
$(\Delta\tilde{f}_i, \tilde{f}_j)_G=0$.

Using (\ref{5}) and (\ref{6}) we compute
\begin{eqnarray*}
(\Delta\tilde{f},\tilde{f})_G &=&
(\Delta\sum_{i=1}^ma_i\tilde{f}_i,
\sum_{j=1}^ma_j\tilde{f}_j)_G\\&=& \sum_{i,j=1}^m a_ia_j
(\Delta\tilde{f}_i,\tilde{f}_j)_G\\&=&\sum_{i=1}^ma_i^2
(\Delta\tilde{f}_i,\tilde{f}_i)_G\\&=& \sum_{i=1}^ma_i^2
(\Delta_if_i,f_i)_{B_i}\\&=& \sum_{i=1}^m
a_i^2\lambda_\mathrm{max}(B_i)(f_i,f_i)_{B_i}\\&\geq&
\min_{i=1,\ldots m}\lambda_\mathrm{max}(B_i)\sum_{i=1}^m
a_i^2(f_i,f_i)_{B_i}\\&=& \min_{i=1,\ldots
m}\lambda_\mathrm{max}(B_i)(\tilde{f},\tilde{f})_G.
\end{eqnarray*} Combining this with (\ref{4}) completes the proof.
\end{proof}

\begin{theo}\label{theofinitegraphs}Let $G=(V,E)$ be a finite graph
and suppose that $2<l=\sup_{x\in
V}\frac{\mu(x)}{\mu_-(x)}<\infty$.
 The $(N-m)$-th eigenvalue $\theta_{N-m}$, $m = 1,2,\ldots \lfloor
\frac{D}{2}\rfloor$ of the Laplace operator on $G$ satisfies
\begin{eqnarray*}\theta_{N-m} &\geq& 1 +
2\frac{\sqrt{l-1}}{l}\cos\left(\frac{\pi}{ \lfloor
\frac{D}{2m}\rfloor}\right) - 2\kappa(G),\end{eqnarray*} where
$D\geq 2$ is the diameter of the graph, $\kappa(G) = \sup_{x_0\in
G}\kappa(x_0,G)$ and $\kappa(x_0,G) = \sup_{x\in
G}\frac{\mu_0(x)}{\mu(x)}$.
\end{theo}The case $D=1$ is not interesting since the graph is
then a complete graph and hence all eigenvalues are precisely
known anyway.
\begin{proof}
Let $m = 1,2,\ldots,\lfloor \frac{D}{2}\rfloor$ be given. We can
find $m$ vertices $x_1,\ldots,x_m$ in $G$ that satisfy
$$d(x_i,x_{i+1})\geq 2r,$$ where $r= \lfloor
\frac{D}{2m}\rfloor\geq 1$. Then the balls $B_i:=B(x_i, r-1)$,
$i=1,\ldots m$ are edge disjoint. Indeed if we assume that there
exists $x\in B_i$ and $y\in B_{i+1}$ such that $x\sim y$, then we
obtain from the triangle inequality
$$2r\leq d(x_i,x_{i+1}) \leq d(x_i,x)+ d(x,y)+d(y,x_{i+1})\leq r-1 +1 +r -1= 2r-1 $$
which is a contradiction. Hence we have shown that we can find $m
=1,\ldots, \lfloor \frac{D}{2}\rfloor$ edge disjoint balls $B(x_i,
\lfloor \frac{D}{2m}\rfloor-1)$ in $G$. Using Corollary
\ref{Estimates Balls 2} and Lemma \ref{Estimate Balls} we obtain
\begin{eqnarray*}\theta_{N-m}&\geq&
\min_{i=1,\ldots,m}\lambda_\mathrm{max}(B(x_i,\lfloor
\frac{D}{2m}\rfloor-1))\\&\geq&
 1 + \frac{2\sqrt{l-1}}{l}\cos\left(\frac{\pi}{\lfloor \frac{D}{2m}\rfloor}\right)
- 2\max_{i=1,\ldots,m}\kappa(x_i,G)\\ &\geq& 1 +
\frac{2\sqrt{l-1}}{l}\cos\left(\frac{\pi}{\lfloor
\frac{D}{2m}\rfloor}\right) - 2\kappa(G).
\end{eqnarray*}
\end{proof}
\begin{rem}
\begin{itemize}\item[$(i)$] Under the same
assumptions as in the last theorem, Urakawa \cite{Urakawa99}
showed that \begin{equation} \label{Urakawa} \theta_m\leq 1 -
2\frac{\sqrt{l-1}}{l}\cos\left(\frac{\pi}{ \lfloor
\frac{D}{2m}\rfloor}\right),\end{equation} for all $m =
1,2,\ldots,\lfloor \frac{D}{2}\rfloor$. Thus Theorem
\ref{theofinitegraphs} is a counterpart of Urakawa's result.
\item[$(ii)$] Using our results in this section together with the
results in \cite{Urakawa99} one immediately obtains a proof of the
famous Alon-Boppana-Serre theorem for the smallest nonzero
eigenvalue and a Alon-Boppana-Serre-type theorem for the largest
eigenvalue if there are no short circuits of odd length in the
graph, see \cite{Friedman93, Lubotzky94, Davidoff03}. In fact one
even obtains slightly stronger results since we do not need the
assumption that the graph is regular nor that $l$ in the above
theorems is an integer, but we only need that the degrees are
bounded from above. A generalization of the Alon-Boppana-Serre
theorem (for both the smallest non-zero and the largest
eigenvalue) for non-regular graphs with bounded degree but integer
$l$ were obtained recently in \cite{Mohar10}, by using a different
method based on the eigenvalue interlacing property of induced
subgraphs.
\end{itemize}
\end{rem}

\section{The top of the spectrum of $\Delta$}\label{Section7}
In Section \ref{Preliminaries} we discussed that the spectrum of
the Laplace operator is connected to the spectrum of the Dirichlet
Laplace operator, see \eqref{Exhaustion}. Now we are going to use
the results obtained for the Dirichlet Laplace operator in the
previous sections to estimate the top of $\sigma(\Delta)$.

The following simple observation will be often used throughout
this section.
\begin{lemma}\label{finitetoinfinite} Let $K$ be a finite subset
of $\Gamma$ and $\Omega\uparrow\Gamma\setminus K$ be an exhaustion
of $\Gamma\setminus K$.  We have
$$\lim_{\Omega\uparrow\Gamma\setminus K}\lambda_{\mathrm{max}}(\Omega)
= \overline{\lambda}(\Gamma\setminus K),$$ where
$\overline{\lambda}(\Gamma\setminus K) = \sup
\sigma(\Delta_{\Gamma\setminus K})$ and $\Delta_{\Gamma\setminus
K}$ is the Laplace operator with Dirichlet boundary condition on
$\Gamma\setminus K$.
\end{lemma}
\begin{proof}
First we note that the limit on the l.h.s. exists. This follows
from the monotonicity of the largest eigenvalue (see Lemma
\ref{BasicProperties} (v)).

For any $f\in C_0(\Gamma\setminus K)$ (as before $C_0$ denotes the
space of finitely supported functions) there exists a finite
$\Omega(f)\subset\Gamma\setminus K$ such that
$\mathrm{supp}f\subset \Omega(f)$ and consequently
$$\frac{(df,df)}{(f,f)} \leq \sup_{g\in C_0(\Omega(f))}
\frac{(dg,dg)}{(g,g)}.$$  If $\Omega\uparrow\Gamma\setminus K$ is
an exhaustion of $\Gamma\setminus K$, we obtain for all $f\in
C_0(\Gamma\setminus K)$
$$\frac{(df,df)}{(f,f)} \leq \lim_{\Omega\uparrow\Gamma\setminus K} \sup_{g\in C_0(\Omega)}
\frac{(dg,dg)}{(g,g)}.$$ Since $\overline{C_0(\Gamma\setminus K)}=
\ell^2(\Gamma\setminus K)$ one can show that \cite{DodziukKarp}
$$\overline{\lambda}(\Gamma\setminus K) =
\sup_{f\in\ell^2(\Gamma\setminus K)}\frac{(df,df)}{(f,f)} =
\sup_{f\in C_0(\Gamma\setminus K)}\frac{(df,df)}{(f,f)}.$$
Altogether we conclude that
$$\overline{\lambda}(\Gamma\setminus K) =
\sup_{f\in C_0(\Gamma\setminus K)}\frac{(df,df)}{(f,f)} \leq
\lim_{\Omega\uparrow\Gamma\setminus K} \sup_{g\in
C_0(\Omega)}\frac{(dg,dg)}{(g,g)}=
\lim_{\Omega\uparrow\Gamma\setminus K}
\lambda_\mathrm{max}(\Omega).$$ On the other hand, for any finite
$\Omega\subset \Gamma\setminus K$ it follows from the monotonicity
of the largest Dirichlet eigenvalue that
$$\lambda_\mathrm{max}(\Omega)\leq \overline{\lambda}(\Gamma\setminus
K).$$ Taking an exhaustion $\Omega\uparrow\Gamma\setminus K$ we
obtain
$$\lim_{\Omega\uparrow\Gamma\setminus K} \lambda_\mathrm{max}(\Omega)
\leq \overline{\lambda}(\Gamma\setminus K).$$ This completes the
proof.
\end{proof}
\begin{lemma}\label{bottom1}
Similarly, the smallest eigenvalue satisfies
$$\lim_{\Omega\uparrow\Gamma\setminus K}\lambda_{1}(\Omega)
= \underline{\lambda}(\Gamma\setminus K),$$ where
$\underline{\lambda}(\Gamma\setminus K) = \inf
\sigma(\Delta_{\Gamma\setminus K})$.
\end{lemma}
\begin{proof}The proof is similar to the one of the last lemma, so
we omit it here.
\end{proof}

The last two lemmata will be very useful in the following since
they allows us to transfer results form finite to infinite
subsets.

\begin{lemma}For any finite $K\subset \Gamma$ we have
\begin{equation}\label{finitetoinfinite2}2\bar{h}(\Gamma\setminus K) + h(\Gamma\setminus K)\leq
 \overline{\lambda}(\Gamma\setminus K)\leq 1+ \sqrt{1 - (1-\bar{h}(\Gamma\setminus
 K))^2},\end{equation}where $\bar{h}(\Gamma\setminus
 K)= \lim_{\Omega\uparrow\Gamma\setminus K}\bar{h}(\Omega)$ and $h(\Gamma\setminus
 K)= \lim_{\Omega\uparrow\Gamma\setminus K}h(\Omega)$.
\end{lemma}
\begin{proof}
From Lemma \ref{finitetoinfinite} and \eqref{eq6} we immediately
obtain for an exhaustion $\Omega\uparrow\Gamma\setminus K$:
\begin{eqnarray*}\overline{\lambda}(\Gamma\setminus
K)&=&\lim_{\Omega\uparrow\Gamma\setminus
K}\lambda_\mathrm{max}(\Omega)\\&\leq&
\lim_{\Omega\uparrow\Gamma\setminus K}\left( 1+ \sqrt{1 -
(1-\bar{h}(\Omega))^2}\right) \\&=& 1+ \sqrt{1 -
(1-\bar{h}(\Gamma\setminus K))^2}\end{eqnarray*} and
\begin{eqnarray*}
\overline{\lambda}(\Gamma\setminus K) &=&\lim_{\Omega\uparrow\Gamma\setminus K}\lambda_\mathrm{max}(\Omega)\\
&\geq& \lim_{\Omega\uparrow\Gamma\setminus K} \left(
2\bar{h}(\Omega) + h(\Omega)\right)\\&=& 2\bar{h}(\Gamma\setminus
K) + h(\Gamma\setminus K).\end{eqnarray*}
\end{proof}
Now we obtain an estimate for the top of the spectrum by the
Cheeger and the dual Cheeger constant.

\begin{theo}\label{Estimates top of the spectrum} The top of the spectrum satisfies:
$$2\bar{h}(\Gamma) + h(\Gamma)\leq
\overline{\lambda}(\Gamma)\leq 1+ \sqrt{1 -
(1-\bar{h}(\Gamma))^2}$$
\end{theo}
\begin{proof}Let $K=\ \emptyset$ in \eqref{finitetoinfinite2}.
This completes the proof.
\end{proof}

For completeness, we include the estimate of the bottom of the
spectrum which is an easy consequence of Theorem
\ref{CheegerInequality} and Lemma \ref{bottom1}.
\begin{theo}[cf. \cite{Fujiwara96}]\label{bottom4} For any finite $K\subset \Gamma$ we have
\begin{equation}\label{bottom2}
1-\sqrt{1-h^2(\Gamma\setminus K)}\leq
\underline{\lambda}(\Gamma\setminus K)\leq h(\Gamma\setminus K),
\end{equation}
\begin{equation}\label{bottom3}
1-\sqrt{1-h^2(\Gamma)}\leq \underline{\lambda}(\Gamma)\leq
h(\Gamma).
\end{equation}
\end{theo}

\section{The largest eigenvalue and geometric properties of
graphs}\label{Section8} For an infinite graph $\Gamma$ with
positive spectrum (called a nonamenable graph), i.e.
$\underline{\lambda}(\Gamma)>0$ which is equivalent to
$h(\Gamma)>0$ by \eqref{bottom3} in Theorem \ref{bottom4}, it is
known that the graph has exponential volume growth and very fast
heat kernel decay, (see $(10.4)$ and Lemma 8.1 in \cite{Woess00}).
In addition, nonamenablity of graphs is a rough-isometric
invariant property (see Theorem 4.7 in \cite{Woess00}). In this
section we ask the question what the top of the spectrum can tell
us about the graph. More precisely we are asking what we can infer
from $\bar{\lambda}(\Gamma)<2$ about the graph. In contrast to
$\underline{\lambda}(\Gamma)>0$ we will present some examples (see
Example \ref{ex2} and Example \ref{ex3}) which show that
$\bar{\lambda}(\Gamma)<2$ is not rough-isometric invariant. Before
giving these examples, we prove some affirmative results which
indicate that the top of the spectrum controls the geometry of the
graph to the same extent as the first eigenvalue does if the graph
is in some sense close to a bipartite one. Note that Theorem
\ref{Estimates top of the spectrum} implies that
$\overline{\lambda}(\Gamma)<2$ is equivalent to
$\bar{h}(\Gamma)<1.$

In the sequel, we fix some vertex $x_0\in\Gamma.$ We denote the
geodesic sphere of radius $r$ ($r\in\mathds{N}\cup\{0\}$) centered
at $x_0$ by $S_r:=S_r(x_0)=\{y\in\Gamma: d(y,x_0)=r\}.$ Let
$p_r:=|E(S_r,S_{r+1})|, q_r:=|E(S_r,S_r)|,$
$P_r:=\sum_{i=0}^rp_i,$ $Q_r:=\sum_{i=0}^rq_i$ and
$p_{-1}=0,p_0=\mu(x_0)-\mu_{x_0,x_0},q_0=\mu_{x_0,x_0}.$ Then
$\mathrm{vol}(B_r):=\mathrm{vol}(B_r(x_0))=P_{r-1}+P_r+Q_r.$
\begin{theo}\label{EPV}
Let $\Gamma$ be an infinite graph with $\bar{h}(\Gamma)\leq
1-\epsilon_0$ for some $0<\epsilon_0<1.$ If
\begin{equation}\label{eq2}
\limsup_{r\rightarrow\infty}\frac{Q_r}{\mathrm{vol}(B_r)}<\epsilon_0,
\end{equation} then
$$\mathrm{vol}(B_r)\geq C_1e^{C_2r},$$ for some $C_1,C_2>0$ and any $r\geq1.$
\end{theo}
\begin{rem} The condition \eqref{eq2} means that the graph $\Gamma$
is close to a bipartite graph in the sense that $Q_r$ is dominated
by $\mathrm{vol}(B_r)$ (or equivalently $P_r$). Of course
condition \eqref{eq2} is trivially satisfied for bipartite graphs
since $(Q_r=0)$ in this case. Obviously
$\lim_{r\rightarrow\infty}\frac{Q_r}{\mathrm{vol}(B_r)}=0$ is
stronger than \eqref{eq2} which will be used in Corollary
\ref{Coro1} and \ref{Coro2}.
\end{rem}

\begin{proof}
By \eqref{eq2} there exists $r_0>0$ and $\theta<\epsilon_0$ such
that $Q_r\leq \theta \mathrm{vol}(B_r)$ for any $r\geq r_0.$ Then
we have
$$\mathrm{vol}(B_r)=P_{r-1}+P_r+Q_r\leq P_{r-1}+P_r+\theta\mathrm{vol}(B_r).$$ This yields
\begin{equation}\label{eq1}
(1-\theta)\mathrm{vol}(B_r)\leq P_{r-1}+P_r\leq 2P_r.
\end{equation}
We introduce an alternating partition of $B_r$ as follows. Let
$V_1=S_0\cup S_2\cup\cdots\cup S_r$ and $V_2=S_1\cup
S_3\cup\cdots\cup S_{r-1}$ if $r$ is even, or $V_1=S_0\cup
S_2\cup\cdots\cup S_{r-1}$ and $V_2=S_1\cup S_3\cup\cdots\cup
S_{r}$ if $r$ is odd. Then $\bar{h}(\Gamma)\leq 1-\epsilon_0$
implies that
$$2|E(V_1,V_2)|\leq (1-\epsilon_0)\mathrm{vol}(B_r).$$ That is
$$2P_{r-1}\leq(1-\epsilon_0)\mathrm{vol}(B_r)\leq \frac{2(1-\epsilon_0)}{1-\theta}P_r,$$ where the last inequality follows from \eqref{eq1}. Then we have for any $r\geq r_0+1$
\begin{eqnarray*}
p_r&=&P_r-P_{r-1}\geq\left(\frac{1-\theta}{1-\epsilon_0}-1\right)P_{r-1}\\
&\geq&\left(\frac{1-\theta}{1-\epsilon_0}-1\right)\left(\frac{1-\theta}{2}\right)\mathrm{vol}(B_{r-1})\\
&=&\delta(\epsilon_0)\mathrm{vol}(B_{r-1}),
\end{eqnarray*} where $\delta(\epsilon_0)=\left(\frac{1-\theta}{1-\epsilon_0}-1\right)\left(\frac{1-\theta}{2}\right)>0$ since $\theta<\epsilon_0.$
Then for $r\geq r_0+1$
\begin{eqnarray*}
\mathrm{vol}(B_r)&=&\mathrm{vol}(B_{r-1})+p_r+p_{r-1}+q_r\\
&\geq&\mathrm{vol}(B_{r-1})+p_r\\
&\geq&(1+\delta(\epsilon_0))\mathrm{vol}(B_{r-1}).
\end{eqnarray*}
The iterations imply that
$$\mathrm{vol}(B_r)\geq(1+\delta)^{r-r_0}\mathrm{vol}(B_{r_0}),$$
for any $r\geq r_0.$ The theorem follows.
\end{proof}

Next we give some sufficient conditions for the previous theorem.
We assume in the next two corollaries that
$\mathrm{vol}(B_r)\uparrow\infty$ $(r\rightarrow\infty)$. A
sufficient condition for this is for example that $\mu_{xy}\geq
C_0>0$ for any $x\sim y.$
\begin{coro}\label{Coro1}
Let $\Gamma$ be an infinite graph with $\bar{h}(\Gamma)\leq
1-\epsilon_0$ for some $0<\epsilon_0<1$ and
\begin{equation}\label{eq3}
\sum_{r=0}^\infty\frac{q_r}{\mathrm{vol}(B_r)}<\infty,
\end{equation} where we assume that $\mathrm{vol}(B_r)\uparrow\infty$
$(r\rightarrow\infty)$. Then we have
$$\mathrm{vol}(B_r)\geq C_1e^{C_2r},$$ for some $C_1,C_2>0$ and any $r\geq1.$
\end{coro}
\begin{proof}
Kronecker's lemma, which is well-known in probability theory (see
\cite{Durrett96}), \eqref{eq3}, and
$\mathrm{vol}(B_r)\uparrow\infty$ imply that
$$\frac{Q_r}{\mathrm{vol}(B_r)}\rightarrow 0,\ \ \ \ r\rightarrow\infty.$$ Then the corollary follows from Theorem \ref{EPV}.
\end{proof}

\begin{coro}\label{Coro2}
Let $\Gamma$ be an infinite graph with
$\mathrm{vol}(B_r)\uparrow\infty$ $(r\rightarrow\infty)$,
$\bar{h}(\Gamma)\leq 1-\epsilon_0$ for some $0<\epsilon_0<1$ and
$$p_r\rightarrow\infty,\ \frac{q_r}{p_r}\rightarrow0,\ \ \ r\rightarrow\infty.$$ Then we have
$$\mathrm{vol}(B_r)\geq C_1e^{C_2r},$$ for some $C_1,C_2>0$ and any $r\geq1.$
\end{coro}
\begin{proof}
The corollary follows from
$$\frac{Q_r}{\mathrm{vol}(B_r)}\leq\frac{Q_r}{P_r}\rightarrow0,\ \ \ \ r\rightarrow\infty,$$ since $p_r\rightarrow\infty$ and $\frac{q_r}{p_r}\rightarrow0$ (L'Hospital's Rule).
\end{proof}

We recall the definition of Cayley graphs. Let $G$ be a group. A
subset $S\subset G$ is called a generating set of $G$ if any $x\in
G$ can be written as $x=s_1s_2\cdots s_n$ for some $n\geq1$ and
$s_i\in S,$ $1\leq i\leq n.$
\begin{defi}\label{Cayley graph} For a group $G$ and a finite symmetric generating set $S$ of $G$ (i.e. $S=S^{-1}$), there exists a graph $\Gamma=(V,E)$
associated with the pair $(G,S)$ where the set of vertices $V=G$
and $(x,y)\in E$ iff $x=ys$ for some $s\in S.$ It is called the
Cayley graph associated with $(G,S).$
\end{defi}

In the following examples, we calculate the largest eigenvalues of
Dirichlet Laplace operator. These calculations will show that
$\bar{\lambda}<2$ is not a rough-isometric invariant.

\begin{example}\label{ex1}
Let $P_{\infty}$ be the graph of an infinite line which is the
Cayley graph of $(\mathds{Z},\{\pm 1\})$, see Figure \ref{Fig.1}.
Denote $\Omega_n^{k_0}:=\{k_0+i: 0\leq i\leq n-1\}.$ By the
symmetry, $\lambda(\Omega_n^{k_0})=\lambda(\Omega_n^{k_1})$ for
any $k_0,k_1\in\mathds{Z}.$ Without loss of generality, we
consider the Dirichlet eigenvalues of $\Omega_n:=\Omega_n^0.$ Then
the eigenvalues of $\Delta_{\Omega_n}$ are
$$1-\cos\frac{j\pi}{n+1},\ \ \ \ j=1,2,\cdots,n.$$ Hence
$$\lambda_1(\Omega_n)=1-\cos\frac{\pi}{n+1}\rightarrow 0,$$ and
$$\lambda_{\mathrm{max}}(\Omega_n)=1+\cos\frac{\pi}{n+1}\rightarrow 2,$$
as $n\rightarrow\infty.$ By \eqref{Exhaustion} this implies that
$\underline{\lambda}(\Gamma)=0$ and
$\overline{\lambda}(\Gamma)=2$.
\end{example}
\begin{figure}\begin{center}
\includegraphics[width =
7cm]{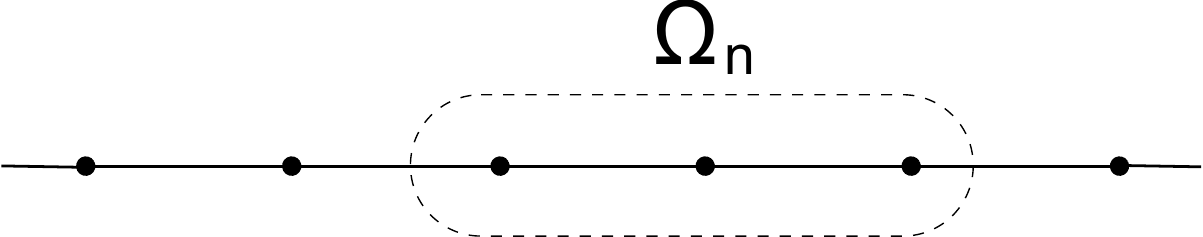}\caption{\label{Fig.1} The infinite path
$P_{\infty}$.  }
\end{center}
\end{figure}

\begin{example}\label{ex2}
Let $\Gamma=(V,E),$
$V=\{i\in\mathds{Z}\}\cup\{i'\in\mathds{Z}\}=V_1\cup V_2,$ i.e.
the disjoint union $\mathds{Z}\sqcup\mathds{Z},$ and
$E=E(V_1,V_1)\cup E(V_2,V_2)\cup E(V_1,V_2)$ where
$$E(V_1,V_1)=\{(i,j):i,j\in V_1, |i-j|=1\},$$
$$E(V_2,V_2)=\{(i',j'):i',j'\in V_2, |i'-j'|=1\},$$
and
$$E(V_1,V_2)=\{(i,j'):i\in V_1, j'\in V_2, |i-j'|=0\ or\ 1\},$$ see Figure
\ref{Fig.2}. Let $\Omega_n=\{i\in V_1: 0\leq i\leq
n-1\}\cup\{j'\in V_2: 0\leq j'\leq n-1\}.$ Then
$$\lambda_1(\Omega_n)=\frac{4}{5}(1-\cos\frac{\pi}{n+1})\rightarrow 0,$$ and
$$\lambda_{\mathrm{max}}(\Omega_n)=\frac{4}{5}(1+\cos\frac{\pi}{n+1})\rightarrow \frac{8}{5}<2,$$
as $ n\rightarrow\infty.$ Again by \eqref{Exhaustion} we have
$\underline{\lambda}(\Gamma)=0$ and
$\overline{\lambda}(\Gamma)=\frac{8}{5}$.
\end{example}
\begin{figure}\begin{center}
\includegraphics[width =
7cm, bb=0 0 400 300]{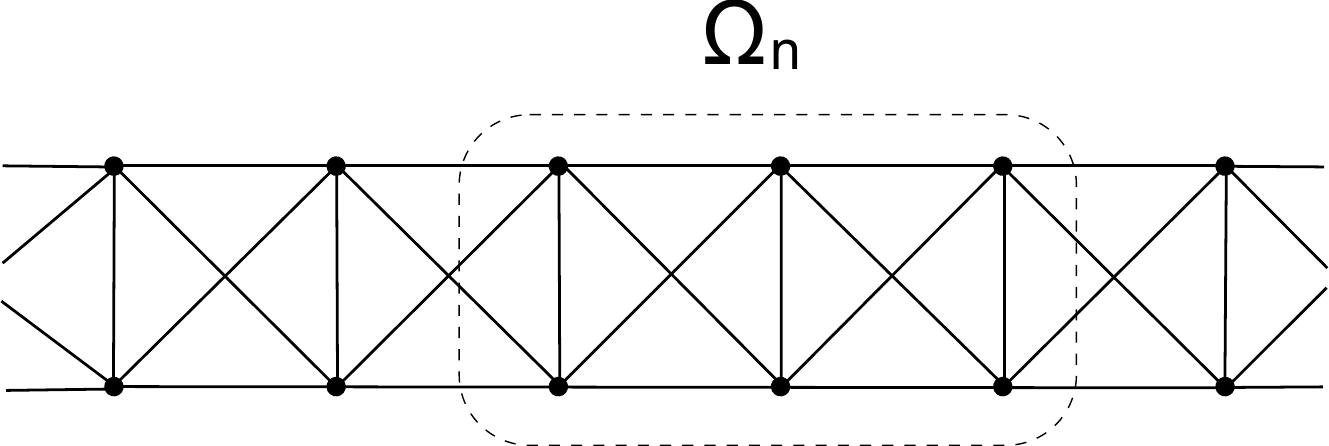}\caption{\label{Fig.2}  The
graph from Example \ref{ex2}.}
\end{center}
\end{figure}

\begin{example}\label{ex3}
Let $\Gamma$ be the Cayley graph of
$(\mathds{Z}\times\mathds{Z}_3,\{(\pm 1,\pm 1)\})$ and
$\Omega_n=\{i\in \mathds{Z}: 0\leq i\leq n-1\}\times\mathds{Z}_3$,
see Figure \ref{Fig.3}.  Then
$$\lambda_1(\Omega_n)=\frac{1}{2}(1-\cos\frac{\pi}{n+1})\rightarrow 0,$$
and
$$\lambda_{\mathrm{max}}(\Omega_n)=\frac{5}{4}+\frac{1}{2}\cos\frac{\pi}{n+1}\rightarrow \frac{7}{4}<2,$$
$n\rightarrow\infty.$ Again by \eqref{Exhaustion} we have
$\underline{\lambda}(\Gamma)=0$ and
$\overline{\lambda}(\Gamma)=\frac{7}{4}$.
\end{example}
\begin{figure}\begin{center}
\includegraphics[width =
7cm]{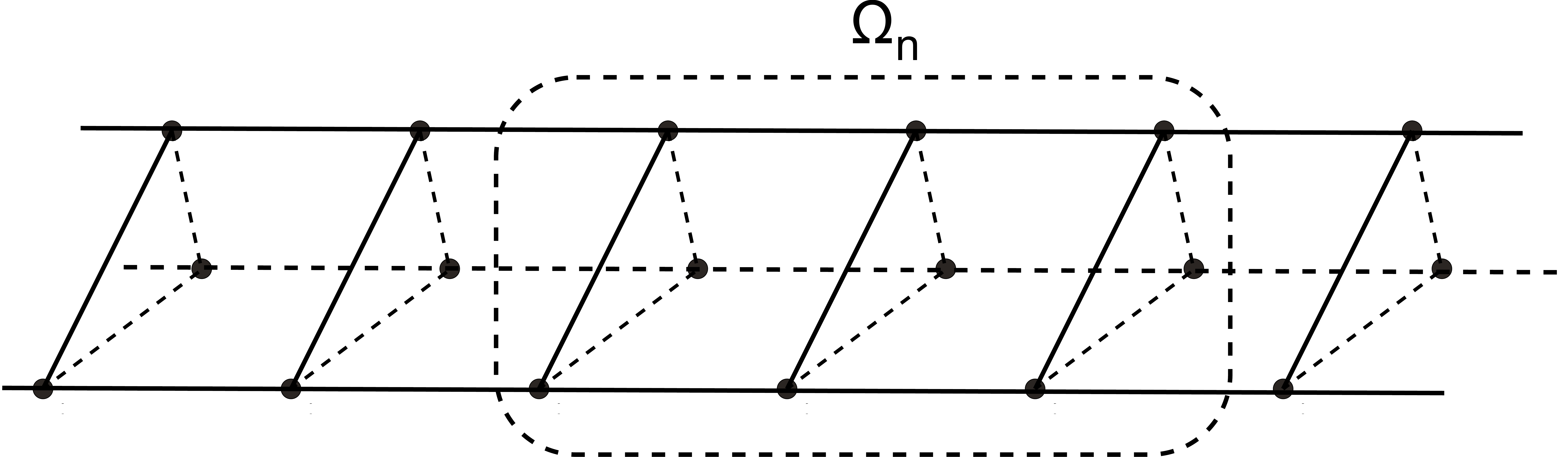}\caption{\label{Fig.3}  The Cayley
graph of $(\mathds{Z}\times\mathds{Z}_3,\{(\pm 1,\pm 1)\})$.}
\end{center}
\end{figure}

\begin{rem}\label{Remark hbar weak}
Example \ref{ex2} and Example \ref{ex3} show that in general
without additional conditions like \eqref{eq2} one cannot hope
that $\Gamma$ has exponential volume growth under the assumption
of $\bar{h}(\Gamma)<1$ (or, by Theorem \ref{Estimates top of the
spectrum}, equivalently $\bar{\lambda}(\Gamma)<2$). The heat
kernels $p_n(x,y)$ in Example \ref{ex2} and Example \ref{ex3}
decay as $\frac{1}{\sqrt{n}}$ which is the slowest possible rate
of heat kernel on infinite graph, see \cite{Coulhon99}. Hence the
spectral gap of $\bar{\lambda}(\Gamma)$
($\bar{\lambda}(\Gamma)<2$) does not imply any nice heat kernel
estimate.
\end{rem}

We recall the definition of rough isometries between metric
spaces, also called quasi-isometries (see
\cite{Gromov81,Woess00,Burago01}).
\begin{defi} Let $(X,d^X), (Y,d^Y)$ be two metric spaces. A rough isometry is a mapping $\phi: X\rightarrow Y$ such that
$$a^{-1}d^X(x,y)-b\leq d^Y(\phi(x),\phi(y))\leq a d^X(x,y)+b,$$ for all $x,y\in X,$ and
$$d^Y(z,\phi(X))\leq b,$$ for any $z\in Y,$ where $a\geq 1,$ $b\geq0$ and $d^Y(z,\phi(X)):=\inf\{d^Y(z,\phi(x)): x\in X\}.$ It is called an $(a,b)$-rough isometry.
\end{defi}

For infinite graphs, we consider the metric structure defined by
the graph distance. It is easy to see that the graphs in Example
\ref{ex1}, Example \ref{ex2}, and Example \ref{ex3} are
rough-isometric to each other, but in the first example
$\bar{\lambda}(\Gamma)=2$ whereas in the other two examples
$\bar{\lambda}(\Gamma)<2$. Hence the spectral gap of
$\bar{\lambda}(\Gamma)$ is not a rough-isometric invariant
although this is true for the spectral gap of
$\underline{\lambda}(\Gamma)$, see \cite{Woess00}. We summarize
these insights in the following corollary:

\begin{coro}\label{coro hbar weak }
For infinite graphs, the property that $\bar{h}(\Gamma)<1$, or
equivalently $\bar{\lambda}(\Gamma)<2$, is not a rough-isometric
invariant.
\end{coro}

\section{The largest eigenvalue of graphs with certain symmetries}\label{Section9}
In this section, we shall consider an upper bound estimate for the
top of the spectrum which comes from the geometric properties of
graphs with certain symmetries, see quasi-transitive graphs in
\cite{Woess00}.

Let $\Gamma$ be a locally finite, connected, unweighted (infinite)
graph with the graph distance $d.$ An automorphism of $\Gamma$ is
a self-isometry of the metric space $(\Gamma, d).$ The group of
all automorphisms of $\Gamma$ is denoted by
$\mathrm{Aut}(\Gamma).$ For any $x\in \Gamma,$ as a subset of
$\Gamma,$ $\mathrm{Aut}(\Gamma) x:=\{gx: g\in \mathrm{Aut}(\Gamma)
\}$ is called an orbit of $x$ with respect to $\mathrm{Aut}
(\Gamma).$ It is easy to see that all orbits, denoted by
$\Gamma\slash \mathrm{Aut}(\Gamma)=\{O_i\}_{i\in \mathcal{I}}$,
compose a partition of $\Gamma.$
\begin{defi}
A graph $\Gamma$ is called quasi-transitive if there are finitely
many orbits for $\mathrm{Aut}(\Gamma$), i.e. $\sharp
\mathcal{I}<\infty.$ It is called vertex-transitive if there is
only one orbit, i.e. $\sharp \mathcal{I}=1.$
\end{defi}

By the definition, vertex-transitive graphs are quasi-transitive.
In particular, Cayley graphs (see Definition \ref{Cayley graph})
are vertex-transitive and hence quasi-transitive. Figure
\ref{quasitran} is an example of a quasi-transitive graph which is
not vertex-transitive.
\begin{figure}\begin{center}
\includegraphics[width =
 12cm]{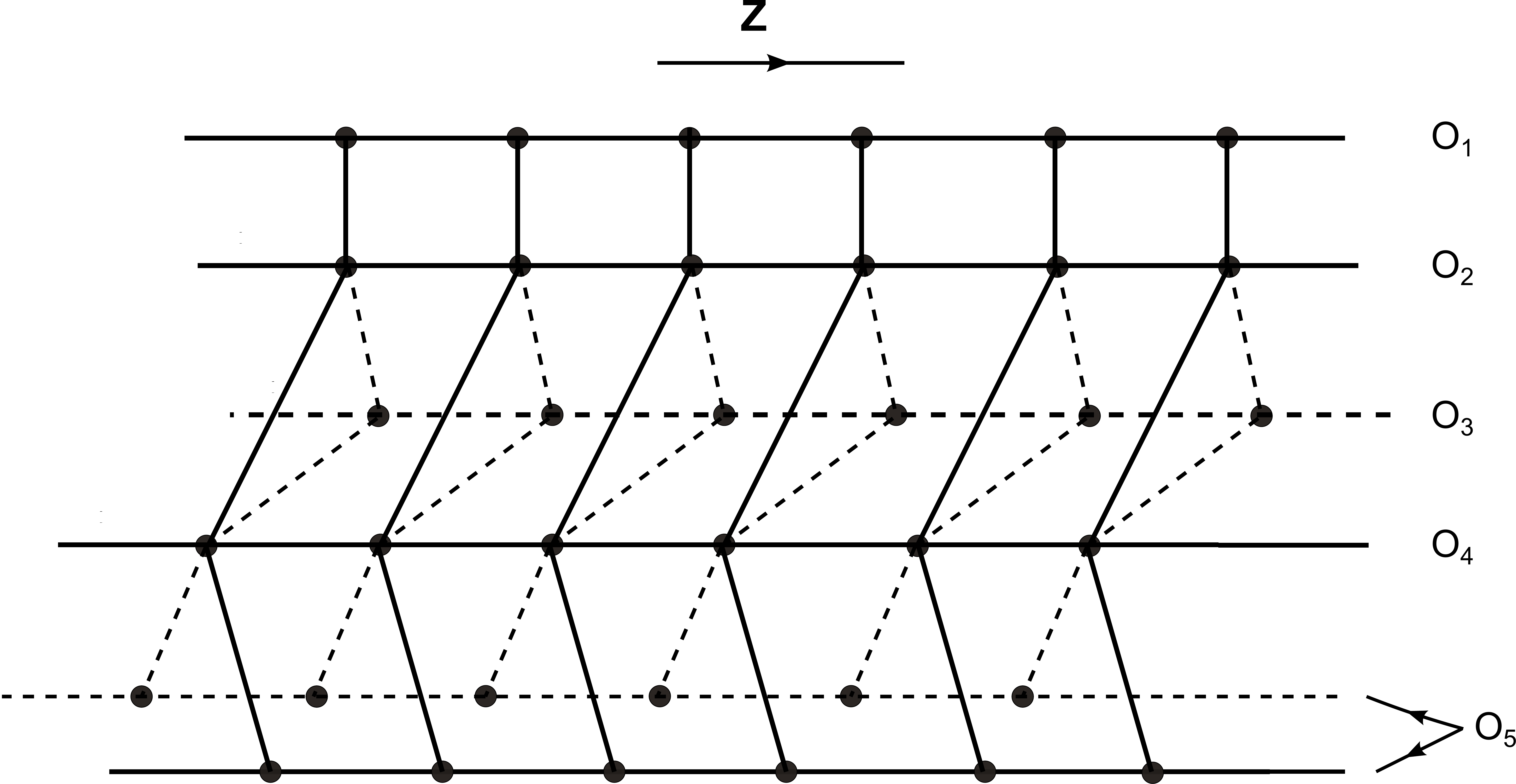}\caption{\label{quasitran}A quasi-transitive graph.}
\end{center}
\end{figure}

Let $\Gamma$ be a quasi-transitive graph and $\Gamma\slash
\mathrm{Aut}(\Gamma)=\{O_i\}_{i\in \mathcal{I}}$. Denote by
$d(\Gamma):=\max\{d_x: x\in \Gamma\}$ the maximal degree of
$\Gamma$ (in accordance with the existing literature we denote in
this section the degree of an unweighted graph by $d_x$ instead of
$\mu(x)$). For any (infinite) graph $\Gamma,$ a cycle of length
$n$, is a closed path as $x_0\sim x_1\sim\cdots\sim x_{n-1}\sim
x_0.$ It is called a circuit (or simple cycle) if there are no
repeated vertices on the cycle. Let us denote by $OG(\Gamma)$ the
odd girth of $\Gamma,$ i.e. the length of the shortest odd cycles
in the graph $\Gamma$ ($OG(\Gamma)=\infty$ if there is no odd
cycles). It is easy to see that there is an odd-length circuit
attaining this number if $OG(\Gamma)<\infty.$ In addition,
$OG(\Gamma)=\infty$ if and only if $\Gamma$ is bipartite.

In case $\Gamma$ is a finite graph, we consider the ordinary
normalized Laplace operator (not the Dirichlet one) defined on
$\Gamma$, see \cite{Chung97, Grigoryan09,Bauer12}. Then the
largest eigenvalue of the Laplace operator on $\Gamma$ is equal to
2 if and only if $\Gamma$ is bipartite. We want to prove a similar
result for the largest eigenvalue of the Dirichlet Laplace
operator on infinite graphs. However, for infinite graphs, $2$ is
not always contained in the spectrum if $\Gamma$ is bipartite
(think of the homogenous trees), nor is the graph bipartite if $2$
is in the spectrum (see the example in Remark \ref{Remark
counterexample}). However, if the graph is a quasi-transitive
graph we have the following theorem which is the main result of
this section:

\begin{theo}\label{mainth1}
Let $\Gamma$ be a quasi-transitive graph which is not bipartite.
Then
$$\bar{h}(\Gamma)\leq 1-\delta,$$ where $\delta=\delta(d(\Gamma),\sharp \mathcal{I}, OG(\Gamma)).$
\end{theo}

Recall that by Theorem \ref{Estimates top of the spectrum},
$\overline{\lambda}(\Gamma)=2$ is equivalent to
$\bar{h}(\Gamma)=1.$ As a direct corollary of Theorem
\ref{mainth1}, we obtain

\begin{coro}\label{mthlast}
Let $\Gamma$ be a quasi-transitive graph. If
$\bar{\lambda}(\Gamma)=2,$ then $\Gamma$ is bipartite.
\end{coro}

\begin{rem}\label{Remark counterexample}
This corollary is known in the case of Cayley graphs, see
\cite{HarpeRobertsonValette}. Instead of using C*-algebra
techniques, we estimate the geometric quantity, the dual Cheeger
constant, to prove this result so that it can be easily extended
to graphs with symmetry, e.g. quasi-transitive graphs. Note that
this corollary is not true for general graphs. For instance, the
standard lattice $\mathds{Z}^2$ with one more edge $((0,0),(1,1))$
is a counterexample. In addition, the converse of the assertion is
obviously not true (e.g. for the infinite $d$-regular tree $T_d$,
we have $\bar{\lambda}(\Gamma) =
1+\frac{2\sqrt{d-1}}{d}$).\end{rem}

In order to illustrate the main idea of the proof, we first
consider the case of Cayley graphs which is much easier. Let
$\Gamma$ be the Cayley graph of a group and a generating set
$(G,S)$. For simplicity, we assume the Cayley graph $\Gamma$ is
simple (i.e. the unit element $e\notin S$) and unweighted (i.e.
$\mu_{xy}=1$ for any $x\sim y$). Cayley graphs are regular graphs
with homogeneous structure. If there exists a circuit in $\Gamma$,
then there are isomorphic circuits passing through each vertex of
$\Gamma$.

The idea of the proof comes from the upper bound estimate of the
maximum cut of finite graphs.
\begin{defi}
The maximum cut of a finite graph $\Gamma$ is defined as
$$Mc(\Gamma):=\max_{V=V_1\cup V_2,V_1\cap V_2=\emptyset}|E(V_1,V_2)|.$$
\end{defi}

\begin{theo}\label{mthlast2}
Let $\Gamma=(V,E)$ be a finite Cayley graph of $(G,S).$ If it is
not bipartite, then
$$Mc(\Gamma)\leq m(1-\delta),$$
where $m$ is the number of edges in $\Gamma$ and
$\delta=\delta(d(\Gamma),OG(\Gamma)).$
\end{theo}
\begin{proof}
Let us denote by $n=\sharp V, m=\sharp E$ and $d=\sharp S$ the
number of vertices, the number of edges and the degree of
$\Gamma$, respectively.

Case 1. $OG(\Gamma)=3.$

Then each vertex of $\Gamma$ is contained in at least one
triangle. We denote by $\triangle$ the set of all triangles in
$\Gamma$ and by $|\triangle|:=\sharp \triangle$ the number of
triangles in $\Gamma.$ Since each vertex is contained in at least
one triangle, $3|\triangle|\geq n.$ For any partition of
$V=V_1\cup V_2$ $(V_1\cap V_2=\emptyset),$ we define a
(single-valued) mapping as
$$T:\ \triangle\rightarrow E(V_1,V_2)^c$$
$$\triangle\ni \triangle_1 \mapsto T(\triangle_1),$$
where $E(V_1,V_2)^c=E\backslash E(V_1,V_2)$ and $T(\triangle_1)$
is one of the edges of the triangle $\triangle_1$ which does not
lie in $E(V_1,V_2).$ Note that there always exists at least one
such edge in each triangle $\triangle_1$.

We shall bound the multiplicity of the mapping $T,$ i.e. the
universal upper bound of $\sharp T^{-1}(e)$ for any $e\in
E(V_1,V_2)^c$. If $\triangle_1,\cdots,\triangle_k\in T^{-1}(e)$
for some $e\in E(V_1,V_2)^c,$ then the triangles
$\triangle_1,\cdots,\triangle_k$ share one edge $e.$ It is easy to
see that $k\leq d-1$ by local finiteness. Hence
$$|E(V_1,V_2)^c|\geq\sharp T(\triangle)\geq \frac{|\triangle|}{d-1}.$$ It follows from $m=\frac{nd}{2}$ and $3|\triangle|\geq n$ that
$$|E(V_1,V_2)^c|\geq \frac{n}{3(d-1)}=\frac{2m}{3d(d-1)}.$$
This yields
\begin{equation}\label{eq4}|E(V_1,V_2)|\leq m-\frac{2m}{3d(d-1)}=m\left(1-\frac{2}{3d(d-1)}\right).\end{equation}

Case 2. $OG(\Gamma)=2s+1$ for $s\geq2.$

Let $\pentagon$ denote the set of circuits of length $2s+1$ in
$\Gamma$ and $|\pentagon|:=\sharp \pentagon.$ Then there exists at
least one circuit of length $2s+1$ passing through each vertex
which implies $(2s+1)|\pentagon|\geq n$. Given any partition of
$\Gamma=V_1\cup V_2$ $(V_1\cap V_2=\emptyset),$ we may define a
mapping as
$$T:\ \pentagon\rightarrow E(V_1,V_2)^c$$
$$\pentagon\ni \pentagon_1 \mapsto T(\pentagon_1),$$ where $T(\pentagon_1)$ is one of edges of the circuit $\pentagon_1$ which does not lie in $E(V_1,V_2).$ By the odd length of the circuit $\pentagon_1,$ the mapping $T$ is well defined.

The key point is to estimate the multiplicity of the mapping $T.$
Suppose $\pentagon_1,\cdots,\pentagon_k\in T^{-1}(e)$ for some
$e\in E(V_1,V_2)^c,$ then each of them is contained in the
geodesic ball $B_s(a)$ where $a$ is one of the vertices of $e.$
Since the number of vertices in $B_s(a)$ is bounded above by
$d^s,$ the number of circuits of length $2s+1$ in $B_s(a)$ is
bounded above by $\binom{d^s}{2s+1}.$ That is $k=\sharp
T^{-1}(e)\leq \binom{d^s}{2s+1}.$ Hence
$$|E(V_1,V_2)^c|\geq\sharp T(\pentagon)\geq \frac{|\pentagon|}{\binom{d^s}{2s+1}}\geq \frac{2m}{d(2s+1)\binom{d^s}{2s+1}}=\frac{m}{C(d,s)}$$
This implies that
\begin{equation}\label{eq5}|E(V_1,V_2)|\leq m-\frac{m}{C(d,s)}=m(1-\delta(d,s)).\end{equation}

The theorem follows from \eqref{eq4} and \eqref{eq5}.
\end{proof}

The ingredient of Theorem \ref{mthlast2} is that the ratio
$\frac{Mc(V_1,V_2)}{m}\leq 1-\delta$ has an upper bound which is
independent of the number of vertices of the finite Cayley graph.
This suggests that we may obtain similar estimates for
$\bar{h}(\Gamma)$ for infinite Cayley graphs. Indeed, we now prove
the following theorem for infinite Cayley graphs which is the
special case of Theorem \ref{mainth1}.

\begin{theo}\label{mthlast1}
Let $\Gamma$ be an infinite Cayley graph of $(G,S).$ If it is not
bipartite, then
\begin{equation}\label{eq8}\bar{h}(\Gamma)\leq 1-\delta,
\end{equation} where $\delta=\delta(d(\Gamma),OG(\Gamma))>0.$
\end{theo}
\begin{rem}
The ingredient of the proof is that once there is an odd-length
circuit in the Cayley graph $\Gamma,$ there is at least one
odd-length circuit passing through each vertex which makes the
graph systematically different from a bipartite one. Then it
becomes possible to estimate $\bar{h}(\Gamma)$ from above.
\end{rem}
\begin{proof}
By the definition of
$\bar{h}(\Gamma)=\lim_{\Omega\uparrow\Gamma}\bar{h}(\Omega)$ and
\eqref{dcc}, it suffices to show that for any finite disjoint
subsets $V_1,V_2\subset\Gamma,$
$$|E(V_1,V_2)|\leq C(|E(V_1,V_1)|+|E(V_2,V_2)|+|\partial (V_1\cup V_2)|),$$ where $C=C(d(\Gamma),OG(\Gamma)).$  We denote $V_3:=(V_1\cup V_2)^c=\Gamma\backslash (V_1\cup V_2)$ and $F(V_1,V_2):=E(V_1,V_1)\cup E(V_2,V_2)\cup \partial(V_1\cup V_2).$ Let $d=d(\Gamma).$

Case 1. $OG(\Gamma)=3.$

Then there is a triangle passing through each vertex of $\Gamma$.
We define a mapping as
$$T: E(V_1,V_2)\rightarrow F(V_1,V_2)$$
$$E(V_1,V_2)\ni e\mapsto T(e).$$ Given any $e\in E(V_1,V_2),$ we choose a
vertex $a\in e$ and a triangle $\triangle_1$ containing the vertex
$a.$ Since there exists at least one edge $e_1$ of $\triangle_1$
which lies in $F(V_1,V_2),$ we define $T(e)=e_1.$

The key point is to estimate the multiplicity of the mapping $T.$
For any $e_1\in F(V_1,V_2)$ such that $T^{-1}(e_1)\neq\emptyset,$
any $e\in T^{-1}(e_1)$ lies in the geodesic ball $B_2(b)$ where
$b$ is one of vertices of $e_1.$ Then $$\sharp T^{-1}(e_1)\leq
\mathrm{the\ number\ of\ edges\ in\ } B_2(b)\leq \binom{d^2}{2}.$$

Hence $$\sharp F(V_1,V_2)\geq \sharp T(E(V_1,V_2))\geq
\frac{|E(V_1,V_2)|}{\binom{d^2}{2}}.$$ Since $|E(V_1,V_1)|=2\sharp
E(V_1,V_1),$ we have
$$\sharp F(V_1,V_2)\leq |E(V_1,V_1)|+|E(V_2,V_2)|+|\partial (V_1\cup V_2)|.$$ This yields
$$|E(V_1,V_2)|\leq C(d)(|E(V_1,V_1)|+|E(V_2,V_2)|+|\partial (V_1\cup V_2)|).$$

Case 2. $OG(\Gamma)=2s+1\ (s\geq2).$

Then there exists at least one circuit of length $2s+1$ passing
through each vertex. We claim that for any circuit $C$ of length
$2s+1$ which intersects $V_1\cup V_2,$ $C\cap F(V_1, V_2)\neq
\emptyset,$ i.e. there exists at least one edge of $C$ contained
in $F(V_1,V_2).$ If not, all the edges of $C$ are contained in
$E(V_1,V_2)$ since the set of edges $E=E(V_1,V_2)\cup
F(V_1,V_2)\cup E(V_3,V_3)$ and $C\cap (V_1\cup V_2)\neq\emptyset.$
Then $C$ is a bipartite subgraph which contradicts to the odd
length of $C$. Then the claim follows.

We define a mapping as
$$T: E(V_1,V_2)\rightarrow F(V_1,V_2)$$
$$E(V_1,V_2)\ni e\mapsto T(e).$$
For any $e\in E(V_1,V_2)$ and a vertex $a\in e,$ there is a
circuit $C$ of length $2s+1$ passing through $a$. By the claim, we
may choose one of edges of $C$, $e_1$, which lies in $F(V_1,V_2)$
and define $T(e)=e_1.$

We shall estimate the multiplicity of the mapping $T.$ For any
$e_1\in T(E(V_1,V_2))$ and $e\in T^{-1}(e_1),$ the edge $e$ lies
in the geodesic ball $B_{s+1}(b)$ where $b$ is a vertex of $e_1.$
Then $\sharp T^{-1}(e_1)\leq \binom{d^{s+1}}{2}.$ Hence
$$\sharp F(V_1,V_2)\geq \frac{|E(V_1,V_2)|}{\binom{d^{s+1}}{2}}.$$
Then
$$|E(V_1,V_2)|\leq C(d,s)(|E(V_1,V_1)|+|E(V_2,V_2)|+|\partial (V_1\cup V_2)|).$$

Combining the case 1 and 2, we prove the theorem.
\end{proof}

Now we prove the main theorem of this section.
\begin{proof}[Proof of Theorem \ref{mainth1}]
Let $\Gamma$ be a quasi-transitive graph and $\Gamma\slash
\mathrm{Aut}(\Gamma)=\{O_i\}_{i\in \mathcal{I}}$ ($\sharp
\mathcal{I}<\infty$). Suppose $OG(\Gamma)=2s+1\ (s\geq 1),$ then
there exists a circuit of length $2s+1$, $C_{2s+1}$, passing
through a vertex in $O_j$ for some $j\in\mathcal{I}.$ By the group
action of $\mathrm{Aut}(\Gamma),$ there exists at least one
circuit of length $2s+1$ passing through each vertex in $O_j.$
Since $\Gamma$ is connected, for any $x\in \Gamma,$ there is a
path $P=\cup_{i=0}^{l-1}\{(x_i,x_{i+1})\},$ i.e. $x=x_0\sim
x_1\sim \cdots\sim x_l$ such that $x_l\in O_j$ and $l\leq \sharp
\mathcal{I}.$

For any two disjoint subset $V_1$ and $V_2,$ we define a mapping
$$T: E(V_1,V_2)\rightarrow F(V_1,V_2)$$
$$E(V_1,V_2)\ni e\mapsto T(e),$$ where $F(V_1,V_2):=E(V_1,V_1)
\cup E(V_2,V_2)\cup \partial(V_1\cup V_2).$ For any $e\in E(V_1,
V_2)$ and a vertex $x_0\in e,$ there is a path
$P=\cup_{i=0}^{l-1}\{(x_i,x_{i+1})\}$ such that $x_l\in O_j$ and
$l\leq \sharp \mathcal{I}.$ Let $C_{2s+1}$ be a circuit of length
$2s+1$ passing through $x_l.$ The same argument in Theorem
\ref{mthlast1} implies that $(P\cup C_{2s+1})\cap F(V_1,V_2)\neq
\emptyset.$ We choose one edge $e_1\in(P\cup C_{2s+1})\cap
F(V_1,V_2)$ and define $T(e)=e_1.$

We estimate the multiplicity of the mapping $T.$ For any $e_1\in
T(E(V_1,V_2))$ and $e\in T^{-1}(e_1),$ the edge $e$ lies in the
geodesic ball $B_{\sharp \mathcal{I}+s+1}(b)$ where $b$ is a
vertex of $e_1.$ Since $d(\Gamma)=\max\{d_x: x\in \Gamma\},$
$\sharp T^{-1}(e_1)\leq \binom{d(\Gamma)^{\sharp
\mathcal{I}+s+1}}{2}.$ Hence
$$\sharp F(V_1,V_2)\geq \frac{|E(V_1,V_2)|}{\binom{d(\Gamma)^{\sharp\mathcal{I}+s+1}}{2}}.$$
Then
$$|E(V_1,V_2)|\leq C(d(\Gamma),\sharp \mathcal{I},s)(|E(V_1,V_1)|+|E(V_2,V_2)|+|\partial (V_1\cup V_2)|)$$ which proves the theorem.
\end{proof}

\section{The essential spectrum of $
\Delta$}\label{Section10} In this section, we use the results in
the previous sections to estimate the essential spectrum of
infinite graphs. We denote by $\sigma^{\mathrm{ess}}(\Gamma)$ the
essential spectrum of normalized Laplace operator $\Delta$ of an
infinite graph $\Gamma$ and define the bottom and the top of the
essential spectrum as
$\underline{\lambda}^{\mathrm{ess}}(\Gamma):=\inf
\sigma^{\mathrm{ess}}(\Gamma)$ and
$\bar{\lambda}^{\mathrm{ess}}(\Gamma)=\sup
\sigma^{\mathrm{ess}}(\Gamma).$ Recall that the spectrum is given
by $\sigma(\Gamma) := \sigma^\mathrm{disc}(\Gamma)
\dot{\cup}\sigma^\mathrm{ess}(\Gamma)$, where
$\sigma^\mathrm{disc}(\Gamma)$ contains isolated eigenvalues of
finite multiplicity.

We recall the following theorem due to Fujiwara \cite{Fujiwara96}
which is also known as the decomposition principle in the
continuous setting, see \cite{Donnelly79}.
\begin{theo}[Fujiwara \cite{Fujiwara96}]\label{decomposition theorem} Let $\Gamma$ be an
infinite graph and $K$ be a finite subgraph. Then
$\sigma^\mathrm{ess}(\Delta(\Gamma)) =
\sigma^\mathrm{ess}(\Delta(\Gamma\setminus K))$ where
$\Delta(\Gamma\setminus K)$ denotes the Laplace operator with
Dirichlet boundary conditions.
\end{theo} This theorem shows that the essential spectrum of a
graph is invariant under compact perturbations of the graph. Hence
the essential spectrum does not depend on local properties of
graph but rather on the behavior of the graph at infinity. In
order to estimate the essential spectrum of $\Delta$ we are going
to define the Cheeger and the dual Cheeger constant at infinity.

\begin{defi}[cf. \cite{Brooks,Fujiwara96}]
The Cheeger constant at infinity $h_\infty$ is defined as
$$h_\infty =\lim_{K\uparrow \Gamma}h(\Gamma\setminus K).$$
\end{defi}
\begin{defi} The dual Cheeger constant at infinity
$\bar{h}_\infty$ is defined as
$$\bar{h}_\infty =\lim_{K\uparrow \Gamma}\bar{h}(\Gamma\setminus K).$$
\end{defi}
\begin{theo}\label{essential spectrum}
For the top of the essential spectrum we obtain $$2\bar{h}_\infty
+ h_\infty\leq \overline{\lambda}^\mathrm{ess}(\Gamma)\leq 1+
\sqrt{1 - (1-\bar{h}_\infty)^2}$$
\end{theo}
\begin{proof}Take an exhaustion $K\uparrow\Gamma$ in \eqref{finitetoinfinite2}.
This completes the proof.
\end{proof}
For completeness, we write down the estimate of the bottom of the
essential spectrum which is a consequence of \eqref{bottom2} in
Theorem \ref{bottom4}, see also \cite{Fujiwara96}.
\begin{theo}[cf. \cite{Fujiwara96}]\label{bottom5}
For the bottom of the essential spectrum we have
$$1-\sqrt{1-h_{\infty}^2}\leq \underline{\lambda}^{\mathrm{ess}}(\Gamma)\leq h_{\infty}.$$
\end{theo}
\begin{rem}
Remark \ref{Remark hbar weak} and Corollary \ref{coro hbar weak }
suggest that the dual Cheeger constant $\bar{h}$ is weaker than
the Cheeger constant $h$. Nevertheless, in the next theorem we
prove the somehow surprising results that at infinity both
quantities contain the same information.
\end{rem}The next theorem is the main result of this section.
\begin{theo}\label{hbar and h at infinity} Let $\Gamma$ be a
graph without self-loops. The following statements are equivalent:
\begin{itemize}
\item[$(i)$] $\bar{h}_\infty =0$\item[$(ii)$] $h_\infty
=1$\item[$(iii)$] $\sigma^\mathrm{ess}(\Gamma)=\{1\}$
\end{itemize}
\end{theo}
\begin{proof}The equivalence of $(ii)$
and $(iii)$ was proved by Fujiwara \cite{Fujiwara96}. It suffices
to show the equivalence of $(i)$ and $(ii).$ From Theorem \ref{h
and hbar} and Theorem \ref{h and hbar 3} it follows that for
finite subsets $\Omega\subset V $
$$\frac{1}{2}(1-h(\Omega))\leq \bar{h}(\Omega)\leq 1-h(\Omega).$$
This immediately translates to results for infinite subsets, i.e.
for finite $K\subset\Gamma$
\begin{eqnarray*}\frac{1}{2}(1-h(\Gamma\setminus K))&=&  \lim_{\Omega\uparrow\Gamma\setminus K}
\frac{1}{2}(1-h(\Omega)) \leq \lim_{\Omega\uparrow\Gamma\setminus
K} \bar{h}(\Omega) = \bar{h}(\Gamma\setminus K)\\ &\leq&
\lim_{\Omega\uparrow\Gamma\setminus K} (1-h(\Omega)) = 1 -
h(\Gamma\setminus K).\end{eqnarray*} Taking an exhaustion $
K\uparrow\Gamma$ yields
$$\frac{1}{2}(1-h_\infty)\leq \bar{h}_\infty \leq 1-h_\infty.$$
This proves that $\bar{h}_\infty =0$ if and only if $h_\infty =1.$
\end{proof}
\begin{rem}
\begin{itemize}
\item[$(i)$] The implication $h_\infty=1 \Rightarrow
\bar{h}_\infty =0$ is also true for graphs with self-loops by
Theorem \ref{h and hbar}. \item[$(ii)$] The reason why the proof
of Theorem \ref{hbar and h at infinity} does not work for graphs
with self-loops is that for such graphs one cannot always find a
partition that satisfies \eqref{surgery equation} and hence the
lower bound for $\bar{h}(\Omega)$ in Theorem \ref{h and hbar 3} is
not true. \item[$(iii)$] In general, for graphs with self-loops,
one cannot even expect that $h_\infty
>0$ if $\bar{h}_\infty =0$. The next example shows a graph with
self-loops for which $h_\infty=\bar{h}_\infty =0$ holds.
\end{itemize}
\end{rem}

\begin{example}Consider the graph $\Gamma$ in Figure \ref{Fig.4}. It
is easy to verify that for this graph $\bar{h}_\infty =h_\infty
=0$. Moreover, for this graph we have $\sigma^\mathrm{ess} =
\{0\}$. This can be seen as follows: We know that
$\lim_{\Omega\uparrow\Gamma\setminus K}
\lambda_\mathrm{max}(\Omega) = \overline{\lambda}(\Gamma\setminus
K)$, where $\Omega\uparrow\Gamma\setminus K$ is an exhaustion of
$\Gamma\setminus K$ and $\lambda_\mathrm{max}(\Omega)$ is the
largest eigenvalue of the Laplace operator with Dirichlet boundary
conditions. Moreover,
$\lim_{K\uparrow\Gamma}\overline{\lambda}(\Gamma\setminus K) =
\overline{\lambda}^\mathrm{ess}(\Gamma)$ which yields
$\lim_{K\uparrow\Gamma} \lim_{\Omega\uparrow\Gamma\setminus K}
\lambda_\mathrm{max}(\Omega) =
\overline{\lambda}^\mathrm{ess}(\Gamma)$. By abuse of notation we
denote in the following the closed ball centered at $0$ with
radius $K$ by $K$, i.e. $K := B(0,K)$. By considering the trace of
$\Delta_{\Omega}$ where $\Omega\uparrow \Gamma\setminus B(0,K)$ we
have
$$\lambda_\mathrm{max}(\Omega)\leq \sum_{i =1}^{\sharp
\Omega}\lambda_i(\Omega)= \mathrm{tr}(\Delta_{\Omega}) \leq
\sum_{k=K+1}^{\sharp \Omega +K} \frac{2}{2^k+2}\leq
\sum_{k=K+1}^{\infty} \frac{1}{2^{k-1}}=2^{1-K}.
$$ Taking the limit $\Omega\uparrow \Gamma\setminus B(0,K)$
and the limit $K\uparrow\Gamma$ (which corresponds to $K\to
\infty$) on both sides we arrive at
$$\overline{\lambda}^\mathrm{ess}(\Gamma) = \lim_{K\to
\infty} \lim_{\Omega\uparrow\Gamma\setminus B(0,K)}
\lambda_\mathrm{max}(\Omega)\leq 0.
$$ Since $\sigma^\mathrm{ess}(\Gamma) \subset [0,2]$ and
$\sigma^\mathrm{ess}(\Gamma) \neq \emptyset$ it follows that
$\sigma^\mathrm{ess}(\Gamma)= \{0\}$. Together with Theorem
\ref{essential spectrum} this also yields a proof for the
statement $\bar{h}_\infty =h_\infty =0$.
\end{example}
\begin{theo}\label{nonconcen} $\bar{h}_\infty +h_\infty
=1$ if $\Gamma$ is bipartite.
\end{theo}
\begin{proof}
As above one can show that for bipartite graphs Theorem \ref{h and
hbar} implies that $\bar{h}_\infty +h_\infty =1$. This completes
the proof.
\end{proof}
\begin{rem}
In particular, Theorem \ref{essential spectrum}, Theorem
\ref{bottom5} and Theorem \ref{nonconcen} imply that for bipartite
graphs with $h_\infty =0$ (or equivalently $\bar{h}_\infty =1$)
the essential spectrum is not concentrated at all since both $0$
and $2$ are contained in the essential spectrum.
\end{rem}

\begin{figure}\begin{center}
\includegraphics[width =
 12cm]{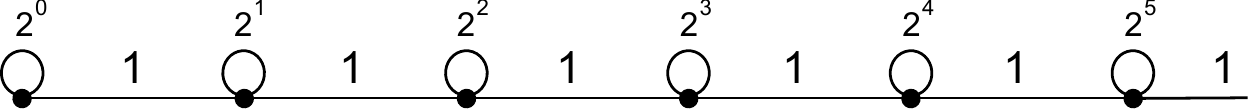}\caption{\label{Fig.4} Counterexample
for Theorem \ref{hbar and h at infinity} if we allow loops in the
graph.}
\end{center}
\end{figure}
We need the following consequence of the spectral theorem, see
\cite{Glazman} Theorem 13 or \cite{Donnelly81} Proposition 2.1:
\begin{theo}\label{essential spectrum intersection}
The interval $[\lambda, \infty)$ intersects the essential spectrum
of a self-adjoint operator $A$ if and only if for all $\epsilon>0$
there exists an infinite dimensional subspace
$\mathcal{H}_\epsilon\subset \mathcal{D}(A)$ of the domain of $A$
such that $(Af - \lambda f + \epsilon f, f)>0$ for all $f\in
\mathcal{H}_\epsilon.$
\end{theo}

\begin{theo}\label{lasts1} Let $2\leq l=\sup_{x\in
V}\frac{\mu(x)}{\mu_{-}(x)}<\infty$, then
$$\sigma^\mathrm{ess}(\Delta) \cap [1 +
\frac{2\sqrt{l-1}}{l}-2\kappa_\infty, 2]\neq \emptyset,$$ where
$\kappa_\infty := \lim_{K\uparrow \Gamma} \kappa(\Gamma\setminus
K)$, $\kappa(\Gamma\setminus K)= \sup_{x_i\in\Gamma\setminus
K}\kappa(x_i,\Gamma\setminus K)$ and $\kappa(x_i,\Gamma\setminus
K) = \sup_{x\in \Gamma\setminus K}\frac{\mu_0(x)}{\mu(x)}$ where
we calculate $\mu_0$ with respect to $x_i$.
\end{theo}
\begin{proof}
Since $2 \leq l=\sup_{x\in V}\frac{\mu(x)}{\mu_{-}(x)}<\infty$ we
have from Theorem \ref{ecpt1} that for all $r\geq 0$
$$\lambda_{\mathrm{max}}(B(x_0,r))\geq
\lambda_{\mathrm{max}}(V_l(r))-2\kappa(x_0,r).$$ Moreover since
$\lim_{r\to\infty}\lambda_{\mathrm{max}}(V_l(r))=
\lambda_{\mathrm{max}}(R_l)$  we know that for all $\epsilon
>0$ there exists a $r(\epsilon)>0$ such that
$$\lambda_{\mathrm{max}}(V_l(r(\epsilon)))>
\lambda_{\mathrm{max}}(R_l)-\epsilon.$$ Thus for sufficiently
large $r(\epsilon)$ we have
$$\lambda_{\mathrm{max}}(B(x_0,r(\epsilon)))>
\lambda_{\mathrm{max}}(R_l)- 2\kappa(x_0,r(\epsilon))-\epsilon.$$
Since $\Gamma\setminus K$ is an infinite graph we can find
infinitely many points $x_i \in \Gamma\setminus K$ such that
$B(x_i,r(\epsilon))\subset \Gamma\setminus K$ are disjoint balls.
Hence there exists infinitely many functions $f_i$ with finite
disjoint support, $\mathrm{supp}f_i\subset
B(x_i,r(\epsilon))\subset \Gamma\setminus K$ such that
$$\frac{(\Delta f_i, f_i)}{(f_i,f_i)}>
\lambda_\mathrm{max}(R_l)-2\kappa(\Gamma\setminus K) - \epsilon =
1 + \frac{2\sqrt{l-1}}{l}- 2\kappa(\Gamma\setminus K)-\epsilon.$$
Taking an exhaustion $K\uparrow\Gamma$ it follows that there
exists an infinite dimensional subspace of $\ell^2(V,\mu)$ such
that
$$(\Delta f - (1 +
\frac{2\sqrt{l-1}}{l})f+ \kappa_\infty f+\epsilon f, f)_\mu
>0.$$
By Theorem \ref{essential spectrum intersection} we conclude that
$$\sigma^\mathrm{ess}(\Delta) \cap [1 +
\frac{2\sqrt{l-1}}{l}-2\kappa_\infty, 2]\neq \emptyset.$$
\end{proof}
From Lemma \ref{kappa and local clustering} we immediately obtain
the following corollary:
\begin{coro}
Let $2\leq l=\sup_{x\in V}\frac{\mu(x)}{\mu_{-}(x)}<\infty$, then
$$\sigma^\mathrm{ess}(\Delta) \cap [1 +
\frac{2\sqrt{l-1}}{l}-2C(\Gamma)_\infty, 2]\neq \emptyset,$$ where
$C(\Gamma)_\infty := \lim_{K\uparrow \Gamma} \sup_{x\in
\Gamma\setminus K}\frac{\sum_{y\in
\mathcal{C}(x)}\mu_{xy}}{\mu(x)}$ and $\mathcal{C}(x)$:= \{$y:
y\sim x$, $x$ and $y$ are both contained in a circuit of odd
length\}.
\end{coro}
\begin{rem}In the special case of unweighted graphs with
vertex degrees uniformly bounded from above by $l$ where $l$ is an
integer, Fujiwara \cite{Fujiwara96} (Theorem $6$) showed that
$\sigma^\mathrm{ess} \cap [1 + \frac{2}{l}, 2]\neq \emptyset$.
Even in the non-optimal case when we restrict ourselves to integer
valued $l$ our results improve the ones by Fujiwara
\cite{Fujiwara96} if $\kappa_\infty \leq \frac{\sqrt{l-1}-1}{l}$.
Moreover, for bipartite graphs Fujiwara obtained
$\sigma^\mathrm{ess}(\Delta) \cap [1 + \frac{2\sqrt{l-1}}{l},
2]\neq \emptyset$. This is a special case of our result since we
obtain the same estimate under the weaker assumption
$\kappa_\infty=0$ and we allow non-integer values for $l$.
\end{rem}

For some fixed reference point $x_0,$ we give the following
definitions which are introduced in \cite{Urakawa00}.
\begin{defi}
For any (possibly infinite) subset $U\subset \Gamma$ we define
$$M_{-}(x_0,U):=\sup_{x\in U}\frac{\mu_{-}(x)}{\mu(x)},$$
\begin{equation}\label{defn1}M_{-,\infty}(x_0,\Gamma):=\lim_{K\uparrow\Gamma}M_{-}(x_0,\Gamma\backslash K),\end{equation} and
$$\kappa(x_0,U):=\sup_{x\in U}\frac{\mu_{0}(x)}{\mu(x)},$$
\begin{equation}\label{defn3}\kappa_{\infty}(x_0,\Gamma):=\lim_{K\uparrow\Gamma}\kappa(x_0,\Gamma\backslash K).\end{equation}
\end{defi}
For the decomposition principle (Theorem \ref{decomposition
theorem}) we immediately obtain the following:
\begin{lemma}\label{MoharT} Let $\Gamma$ be an infinite graph and
$\Gamma^\prime$ be the graph obtained from $\Gamma$ by adding or
deleting finitely many edges. Then
$\sigma^\mathrm{ess}(\Delta(\Gamma))=
\sigma^\mathrm{ess}(\Delta(\Gamma^\prime))$.
\end{lemma}
\begin{proof}
Since we only change finitely many edges, there exists a finite
subgraph $K\subset \Gamma$ and a finite subgraph $K^\prime \subset
\Gamma^\prime$ such that $\Gamma\setminus K = \Gamma^\prime
\setminus K^\prime$. Thus by the last theorem the essential
spectra of $\Delta(\Gamma)$ and $\Delta(\Gamma^\prime)$ coincide.
\end{proof}A similar result to Lemma \ref{MoharT} was obtained by Mohar \cite{Mohar82} for the adjacency operator.

By using Lemma \ref{MoharT}, we obtain the following estimates for
the essential spectra of graphs in terms of
$M_{-,\infty}(x_0,\Gamma)$ and $\kappa_{\infty}(x_0, \Gamma)$.

\begin{theo}\label{ees1}
Let $\Gamma$ be an infinite graph. Then
\begin{equation}\label{estim2}
\underline{\lambda}^{\mathrm{ess}}(\Gamma)\geq
1-\frac{2\sqrt{l-1}}{l}-\kappa_{\infty}(x_0,\Gamma),
\end{equation}
\begin{equation}
\bar{\lambda}^{\mathrm{ess}}(\Gamma)\leq
1+\frac{2\sqrt{l-1}}{l}+\kappa_{\infty}(x_0,\Gamma),
\end{equation} where $l= (M_{-,\infty}(x_0, \Gamma))^{-1}$ and
$x_0$ is an arbitrary vertex in $\Gamma$.
\end{theo}

\begin{proof}
Since it is well known \cite{Fujiwara96b,DodziukKarp} that
$\underline{\lambda}^{\mathrm{ess}}(\Gamma)+\bar{\lambda}^{\mathrm{ess}}(\Gamma)\leq
2,$ it suffices to show \eqref{estim2}. By the definitions of
$M_{-,\infty}(x_0, \Gamma)$ and $\kappa_{\infty}(x_0,\Gamma),$
\eqref{defn1} and \eqref{defn3}, for any $\epsilon>0,$ there
exists $r_0>0$ such that
\begin{equation}\label{estim5}M_{-,\infty}(x_0, \Gamma)\leq M_{-}(x_0,\Gamma\backslash B(x_0,r_0))\leq M_{-,\infty}(x_0, \Gamma)+\epsilon\end{equation} and
\begin{equation}\label{estim10}\kappa_{\infty}(x_0, \Gamma)\leq \kappa(x_0, \Gamma\backslash B(x_0,r_0))\leq \kappa_{\infty}(x_0, \Gamma)+\epsilon.\end{equation}
Let $\Gamma\backslash
B(x_0,r_0)=\Gamma_1\cup\Gamma_2\cup\cdots\cup\Gamma_k$ where
$\Gamma_i$ ($1\leq i\leq k$) are the connected components of
$\Gamma\backslash B(x_0,r_0).$ Then by Lemma \ref{MoharT},
$$\sigma^{\mathrm{ess}}(\Gamma)=\sigma^{\mathrm{ess}}(\Gamma\backslash B(x_0,r_0))=\cup_{i=1}^k\sigma^{\mathrm{ess}}(\Gamma_i).$$

We want to use Corollary \ref{Lowercomp1} to estimate
$\sigma^{\mathrm{ess}}(\Gamma_i).$ But the reference point $x_0$
is not contained in $\Gamma_i$ and thus
$\frac{\mu_{-}(x)}{\mu(x)}$ and $\frac{\mu_{0}(x)}{\mu(x)}$ will
change if we choose some reference point in $\Gamma_i$. In the
following, we add a new reference point to each $\Gamma_i$ which
preserves the quantities $\frac{\mu_{-}(x)}{\mu(x)}$ and
$\frac{\mu_{0}(x)}{\mu(x)}$.

We add a new vertex $p_i$ to each $\Gamma_i$ and let
$\Gamma_i'=\Gamma_i\cup\{p_i\}.$ For each $z\in \Gamma_i$ with
$d(z,x_0)=r_0+1,$ i.e. $z\in S_{r_0+1}(x_0)\cap \Gamma_i,$ we add
an new edge $zp_i$ in $\Gamma_i'$ with the edge weight
$\mu_{zp_i}=|E(z,B(x_0,r_0))|,$ that is, we identify the ball
$B(x_0,r_0)$ as a new vertex $p_i$ and preserve the edge between
$\Gamma_i$ and $B(x_0,r_0).$ Other edges in $\Gamma_i$ remain
unchanged in $\Gamma_i'$. Then by this construction, for any
$x\in\Gamma_i\subset\Gamma_i'$ the quantities
$\frac{\mu_{-}(x)}{\mu(x)}$ and $\frac{\mu_{0}(x)}{\mu(x)}$ w.r.t.
$x_0$ in $\Gamma$ are the same as those w.r.t. $p_i$ in
$\Gamma_i^\prime.$ Then
\begin{equation}\label{estim6}
M_{-}(p_i, \Gamma_i')=M_{-}(x_0, \Gamma_i),\ \kappa(p_i,
\Gamma_i')=\kappa(x_0, \Gamma_i)\end{equation} since
$\frac{\mu_{-}(p_i)}{\mu(p_i)}=\frac{\mu_{0}(p_i)}{\mu(p_i)}=0$
w.r.t. $p_i$ in $\Gamma_i'.$ Hence by \eqref{estim1} in Corollary
\ref{Lowercomp1} (letting $r\rightarrow\infty$), we have
\begin{equation}\label{estim7}\underline{\lambda}(\Gamma_i')\geq 1-\frac{2\sqrt{( M_{-}(p_i, \Gamma_i'))^{-1}-1}}{( M_{-}(p_i, \Gamma_i'))^{-1}}-\kappa(p_i, \Gamma_i').\end{equation}
Since $\sigma^{\mathrm{ess}}(\Gamma_i')\subset\sigma(\Gamma_i'),$
\begin{equation}\label{estim8}\underline{\lambda}^{\mathrm{ess}}(\Gamma_i')\geq
\underline{\lambda}(\Gamma_i').\end{equation} Hence
\begin{eqnarray*}
\underline{\lambda}^{\mathrm{ess}}(\Gamma)&=&\min_{1\leq i\leq k}\underline{\lambda}^{\mathrm{ess}}(\Gamma_i)\\
&=&\min_{1\leq i\leq k}\underline{\lambda}^{\mathrm{ess}}(\Gamma_i')\\
&\geq&\min_{1\leq i\leq k}\left\{1-\frac{2\sqrt{( M_{-}(p_i, \Gamma_i'))^{-1}-1}}{( M_{-}(p_i, \Gamma_i'))^{-1}}-\kappa(p_i, \Gamma_i')\right\}\\
&\geq&1-\frac{2\sqrt{( M_{-,\infty}(x_0,
\Gamma)+\epsilon)^{-1}-1}}{( M_{-,\infty}(x_0,
\Gamma)+\epsilon)^{-1}}-\kappa_{x_0, \infty}(\Gamma)-\epsilon,
\end{eqnarray*} where we use Lemma \ref{MoharT}, \eqref{estim5}, \eqref{estim10}, \eqref{estim6}, \eqref{estim7} and \eqref{estim8}.
Let $\epsilon\rightarrow 0,$ we prove the theorem.
\end{proof}

Theorem \ref{ees1} yields a more general sufficient condition for
the concentration of the essential spectrum than Theorem 5.4 in
\cite{Urakawa99}. It is a special case of a class of graphs with
concentrated essential spectrum discovered by Higuchi, see
\cite{Fujiwara96}.
\begin{coro}\label{lasts2}
Let $\Gamma$ be an infinite graph. If $M_{-,\infty}(x_0,
\Gamma)=0$ and $\kappa_{\infty}(x_0, \Gamma)=0,$ then
$\sigma^{\mathrm{ess}}(\Gamma)=\{1\}.$
\end{coro}
\textbf{Acknowledgement:} We thank Shiping Liu for his careful
proofreading of our paper.
\bibliography{DirichletLaplace}

\begin{thebibliography}{dlHRV93}

\bibitem[Alo96]{Alon96}
N.~Alon.
\newblock {Bipartite subgraphs}.
\newblock {\em Combinatorica}, 16:301--311, 1996.

\bibitem[BBI01]{Burago01}
D.~Burago, Yu. Burago, and S.~Ivanov.
\newblock {\em A course in metric geometry}.
\newblock Number~33 in Graduate Studies in Mathematics. American Mathematical
  Society, Providence, RI, 2001.

\bibitem[BJ]{Bauer12}
F.~Bauer and J.~Jost.
\newblock {Bipartite and neighborhood graphs and the spectrum of the normalized
  graph Laplacian}.
\newblock {\em To appear in: Communications in Analysis and Geometry,
  http://arxiv.org/pdf/0910.3118v4.pdf}.

\bibitem[Bro84]{Brooks}
R.~Brooks.
\newblock {On the spectrum of noncompact manifolds with finite volume}.
\newblock {\em Math. Z.}, 187(3):425--432, 1984.

\bibitem[CG98]{Coulhon98}
T.~Coulhon and A.~Grigor'yan.
\newblock {Random walks on graphs with regular volume groth}.
\newblock {\em Geometric and Functional Analysis}, 8:656--701, 1998.

\bibitem[CGY96]{Chung-Grigoryan-Yau96}
F.~Chung, A.~Grigor'yan, and S.~Yau.
\newblock {Upper bounds for eigenvalues of the discrete and continuous Laplace
  operators}.
\newblock {\em Advances in Mathematics}, 117:165--178, 1996.

\bibitem[CGY97]{Chung-Grigoryan-Yau97}
F.~Chung, A.~Grigor'yan, and S.~Yau.
\newblock {\em {Eigenvalues and diameters for manifolds and graphs}}.
\newblock International Press, 1997.

\bibitem[CGY00]{Chung-Grigoryan-Yau00}
F.~Chung, A.~Grigor'yan, and S.~Yau.
\newblock {Higher eigenvalues and isoperimetric inequalities on Riemannian
  manifolds and graphs}.
\newblock {\em Communications in Analysis and Geometry}, 5:969--1026, 2000.

\bibitem[Cha84]{Chavel84}
I.~Chavel.
\newblock {\em {Eigenvalues in Riemannian geometry}}.
\newblock Academic Press, 1984.

\bibitem[Che75a]{Cheng75a}
S.Y. Cheng.
\newblock Eigenfunctions and eigenvalues of laplacian.
\newblock {\em Proc. Symp. Pure Math.}, 27:185–193, 1975.

\bibitem[Che75b]{Cheng75b}
S.Y. Cheng.
\newblock Eigenvalue comparison theorems and its geometric applications.
\newblock {\em Math. Z.}, 143:289--297, 1975.

\bibitem[Chu97]{Chung97}
Fan R.~K. Chung.
\newblock {\em Spectral Graph Theory (CBMS Regional Conference Series in
  Mathematics, No. 92)}.
\newblock {American Mathematical Society}, 1997.

\bibitem[CLY12]{Chung12}
F.~Chung, Y.~Lin, and S.T. Yau.
\newblock {Harnack inequalities on graphs with Ricci curvature bounded below}.
\newblock {\em preprint}, 2012.

\bibitem[Cou99]{Coulhon99}
T.~Coulhon.
\newblock Random walks and geometry on infinite graphs.
\newblock {\em Lecture notes on analysis in metric spaces}, 1999.

\bibitem[CY95]{Chung-Yau95}
F.~Chung and S.~Yau.
\newblock {Eigenvalues of graphs and Sobolev inequalities}.
\newblock {\em Combinatorics, Probability and Computing}, 4:11--26, 1995.

\bibitem[Del99]{Delmotte99}
T.~Delmotte.
\newblock {Parabolic Harnack inequalities and estimates of Markov chains on
  graphs}.
\newblock {\em Revista Mathematica Iberoamericana}, 15:181--232, 1999.

\bibitem[DGLS01]{Davies01}
E.~Davies, G.~Gladwell, J.~Leydold, and P.~Stadler.
\newblock {Discrete nodal domain theorems}.
\newblock {\em Linear Algebra and its Applications}, 336:51--60, 2001.

\bibitem[DK86]{DodziukKendall}
J.~Dodziuk and W.~S. Kendall.
\newblock {Combinatorial Laplacians and the isoperimetric inequality}.
\newblock {\em From Local Times to Global Geometry, Control and Physics, Pitman
  Res. Notes Math. Ser. 150, K. D. Ellworthy, ed.}, pages 68--74, 1986.

\bibitem[DK88]{DodziukKarp}
J.~Dodziuk and L.~Karp.
\newblock Spectral and function theory for combinatorial laplacians.
\newblock {\em In Geometry of Random Motion, Ithaca, NY, Contempo- rary
  Mathematics}, 73:25--40, 1988.

\bibitem[DL79]{Donnelly79}
H.~Donnelly and P.~Li.
\newblock Pure point spectrum and negative curvature for noncompact manifolds.
\newblock {\em Duke Math. J.}, 46:497--503, 1979.

\bibitem[dlHRV93]{HarpeRobertsonValette}
P.~de~la Harpe, G.~Robertson, and A.~Valette.
\newblock {On the spectrum of the sum of generators for a finitely generated
  group}.
\newblock {\em Israel J. Math.}, 81(1-2):65--96, 1993.

\bibitem[Dod84]{Dodziuk84}
J.~Dodziuk.
\newblock Difference equations, isoperimetric inequality and transience of
  certain random walks.
\newblock {\em Trans. Amer. Math. Soc.}, 284:787--794, 1984.

\bibitem[Don81]{Donnelly81}
H.~Donnelly.
\newblock On the essentail spectrum of a complete riemannian manifold.
\newblock {\em Topology}, 20:1--14, 1981.

\bibitem[DR94]{Desai94}
M.~Desai and V.~Rao.
\newblock {A characterization of the smallest eigenvalue of a graph}.
\newblock {\em Journal of Graph Theory}, 18:181--194, 1994.

\bibitem[DS91]{Diaconis91}
P.~Diaconis and D.~Stroock.
\newblock {Geometric Bounds for Eigenvalues of Markov Chains}.
\newblock {\em The Annals of Applied Probability}, 1:36--61, 1991.

\bibitem[DSC96]{Diaconis96}
P.~Diaconis and L.~Saloff-Coste.
\newblock Logarithmic sobolev inequalities for finite markov chains.
\newblock {\em Ann. Appl. Probab.}, 6:695--750, 1996.

\bibitem[DSV03]{Davidoff03}
G.~Davidoff, P.~Sarnak, and A.~Valette.
\newblock {\em Elementary Number Theory, Group Theory, and Ramanujan Graphs}.
\newblock Cambridge University Press, 2003.

\bibitem[Dur96]{Durrett96}
R.~Durrett.
\newblock {\em Probability: theory and examples}.
\newblock Duxbury Press, Belmont,CA, second edition edition, 1996.

\bibitem[Fri93]{Friedman93}
J.~Friedman.
\newblock {Some geometric aspects of graphs and their eigenfunctions}.
\newblock {\em Duke Mathematical Journal}, 69:487--525, 1993.

\bibitem[Fuj96a]{Fujiwara96b}
K.~Fujiwara.
\newblock {Growth and the spectrum of the Laplacian of an infinite graph}.
\newblock {\em Tohoku. Math. J.}, 48:293--302, 1996.

\bibitem[Fuj96b]{Fujiwara96}
K.~Fujiwara.
\newblock The laplacian on rapidly branching trees.
\newblock {\em Duke Math. J.}, 83(1):191--202, 1996.

\bibitem[FY94]{Chung94aharnack}
F.Chung and S.-T. Yau.
\newblock A harnack inequality for homogeneous graphs and subgraphs.
\newblock {\em Communication in Analysis and Geometry}, 2:628--639, 1994.

\bibitem[Gla65]{Glazman}
I.~Glazman.
\newblock {\em Direct Methods of Qualitive Spectral Analysis of Singular
  Differential Operators}.
\newblock Daniel Davey, 1965.

\bibitem[Gri09]{Grigoryan09}
A.~Grigor'yan.
\newblock {\em {Analysis on Graphs}}.
\newblock Lecture Notes, University Bielefeld, 2009.

\bibitem[Gro81]{Gromov81}
M.~Gromov.
\newblock Groups of polynomial growth and expanding maps.
\newblock {\em Inst. Hautes \'Eudes Sci. Publ. Math. No.}, 53:53--73, 81.

\bibitem[HJ06]{Horn90}
R.~Horn and C.~Johnson.
\newblock {\em {Matrix analysis}}.
\newblock Cambridge University Press, 2006.

\bibitem[HL98]{Hofmeister98}
T.~Hofmeister and H.~Lefmann.
\newblock {On $k$-partite subgraphs}.
\newblock {\em Ars Combin.}, 50:303--308, 1998.

\bibitem[Kel10]{Keller10}
Matthias Keller.
\newblock The essential spectrum of the laplacian on rapidly branching
  tessellations.
\newblock {\em Mathematische Annalen}, 346:51--66, 2010.

\bibitem[Lub94]{Lubotzky94}
Alexander Lubotzky.
\newblock {\em Discrete groups, expanding graphs and invariant measures}.
\newblock Birkh\"auser Verlag, Basel, 1994.

\bibitem[Moh82]{Mohar82}
B.~Mohar.
\newblock The spectrum of an infinite graph.
\newblock {\em Linear Algebra Appl.}, 48:245--256, 1982.

\bibitem[Moh10]{Mohar10}
Bojan Mohar.
\newblock {A strengthening and a multipartite generalization of the
  Alon-Boppana-Serre theorem.}
\newblock {\em Proc. Am. Math. Soc.}, 138(11):3899--3909, 2010.

\bibitem[Que96]{Quenell96}
G.~T. Quenell.
\newblock {Eigenvalue comparisons in graph theory}.
\newblock {\em Pacific J. Math.}, 2:443--461, 1996.

\bibitem[Sco05]{Scott05}
A.~Scott.
\newblock {\em {Judicious partitions and related problems}}.
\newblock LMS Lecture Note Series 327, 2005.

\bibitem[Ura99]{Urakawa99}
H.~Urakawa.
\newblock Eigenvalue comparison theorems of the discrete laplacians for a
  graph.
\newblock {\em Geometriae Dedicata}, 74:95--112, 1999.

\bibitem[Ura00]{Urakawa00}
H.~Urakawa.
\newblock {The spectrum of an infinite graph}.
\newblock {\em Canad. J. Math.}, 52(5), 2000.

\bibitem[Woe00]{Woess00}
W.~Woess.
\newblock {\em Random walks on infinite graphs and groups}.
\newblock Number 138 in Cambridge Tracts in Mathematics. Cambridge University
  Press, Cambridge, 2000.

\end{thebibliography}
\bibliographystyle{alpha}

\end{document}